\documentclass[en, secnum]{cls/myamspaper}

\title{The Furstenberg--Zimmer structure theorem for stationary random walks}

\usepackage[top=40mm, bottom=40mm, left=30mm, right=30mm]{geometry}

\usepackage{mathrsfs}

\bibliography{bib/bibliography.bib}

\newcommand{\calH}{\mathcal{H}}
\newcommand{\calK}{\mathcal{K}}
\newcommand{\HS}{\operatorname{HS}}
\newcommand{\Homeo}{\operatorname{Homeo}}

\DeclareMathOperator*{\olim}{o-lim}

\begin{document}

\maketitle

\begin{abstract}
  We prove the following version of the Furstenberg--Zimmer structure theorem for stationary
  actions: Any stationary 
  action of a locally compact second-countable 
  group is a weakly mixing extension of a measure-preserving distal system.
\end{abstract}

One of the fundamental goals in the study of dynamical systems is to 
understand their structure, for its intrinsic interest and as a means for classifying dynamical
systems or understanding their longterm behavior. This has motivated the development 
of several structure theorems such as the Furstenberg--Zimmer theorem, proven 
independently by Furstenberg \cite{Furs1977} and Zimmer \cite{Zimm1976}
and best known for being instrumental to Furstenberg's famous proof of 
Szemer\'{e}di's theorem by means of a multiple recurrence property for arbitrary 
measure-preserving dynamical systems.
This recurrence property can be derived by elementary means for the special classes of 
isometric systems and weakly mixing systems. While it is not true that every measure-preserving system decomposes
into an isometric and weakly mixing system, the key insight behind the Furstenberg--Zimmer structure
theorem is that such a decomposition is possible in terms of \emph{extensions}. Thereby,
the theorem splits a system into a weakly mixing extension and a \emph{distal} part built from isometric extensions.

\begin{theorem*}[Furstenberg--Zimmer]
  Let $G \curvearrowright \uX$ be a measure-preserving group action. Then it is a weakly
  mixing extension of a distal action $G \curvearrowright \uX_\ud$.
\end{theorem*}

Since, the Furstenberg--Zimmer structure theorem has inspired numerous other 
development, including a version for $\sigma$-finite measure spaces \cite{AminiSwid2022},
refinements into sharper structure theorems by Host--Kra \cite{HoKr2005} and Ziegler \cite{Ziegler2007},
the use of related ideas for von Neumann algebras by Popa \cite{Popa2007}, and more recently 
extensions of the Furstenberg--Zimmer structure theorem to the case of arbitrary group actions 
on arbitrary probability spaces, see \cite{Jamn2020pre} and \cite{EHK2021}. Among the numerous 
applications of these structure-theoretic developments is also the existence of infinite sumsets 
$B_1 + \dots + B_k$ in any set of positive density, recently proved in \cite{KMRR2022} to 
expand on the Erd\H{o}s sumset conjecture.
However, there also
are many natural group actions that are not measure-preserving and thus lie beyond the scope of the 
Furstenberg--Zimmer theory.

Many spaces come equipped with natural measures $\mu$ but most maps do not preserve this measure, though 
they often preserve at least the measure class $[\mu]$ of measures mutually absolutely continuous 
to $\mu$ which motivates the study of nonsingular dynamics, see \cite{Danilenko2011}.
For example, two generic transformations of a space will usually have no algebraic relations and 
thus generate a free and hence nonamenable group. Therefore, for a random walk on a compact space $K$,
i.e., a countable continuous group action $G \curvearrowright K$ with an appropriate probability measure 
$m$ on $G$, there does not generally exist an invariant measure. However, one can always find 
\emph{stationary} measures $\nu$ satisfying the on-average invariance
\begin{equation*}
    \sum_{g\in\supp(m)} m(g) g_*\nu = \nu.
\end{equation*}
These are the natural class of measures for random walks commonly employed in probability 
theory and they have been studied 
extensively in different contexts due to their connections to discrete Schrödinger operators and
random difference equations (\cite{BougerolLacroix1985}, \cite{Kesten1973})  
Diophantine approximation (\cite{SimmonsWeiss2019}, \cite{ProhaskaSert2020}, 
  \cite{KhalilLuethi2021}),
  as well as their relevance for
  results on orbit closures, equidistribution, and measure classification (\cite{BenoistQuint2011}, 
  \cite{BenoistQuint2013}, \cite{BenoistQuint2013a}, \cite{EskinLindenstrauss2018}, \cite{BenoistQuint2012}),
  and random matrix products (\cite{Furstenberg1963b}, \cite{FurstenbergKesten1960}) which in turn are 
  connected to different machine learning paradigms \cite{reinforcement}.

As it turns out, isometric and distal systems also may naturally occur as the building blocks 
of stationary actions, as the following example shows.

\begin{example*}
  An important and elementary class of actions that usually do not admit invariant 
  measures are actions on projective spaces.
  Pick some $\gamma \in (0,1)$ and let $\alpha, \beta \in [0,2\pi)$ be angles
  linearly independent over $\Q$. Set
  \begin{equation*}
    A = \begin{pmatrix}
         \gamma & 1 & 1 \\
         & \cos(\alpha) & -\sin(\alpha) \\
         & \sin(\alpha) & \cos(\alpha)
        \end{pmatrix}, \qquad
    B = \begin{pmatrix}
         \gamma & 1 & 1 \\
         & \cos(\beta) & -\sin(\beta) \\
         & \sin(\beta) & \cos(\beta)
        \end{pmatrix},
  \end{equation*}
  let $\Gamma_m$ be the subgroup of $\GL_3(\R)$ generated by $A$ and $B$, and set 
  $m \defeq \frac{1}{3}\delta_A + \frac{1}{3}\delta_B + \frac{1}{6}\delta_{A^{-1}} + \frac{1}{6}\delta_{B^{-1}}$.
  What are the stationary measures for the action of $(\Gamma_m,m)$ on $\P(\R^3)$? It can be shown that 
  apart from the trivial invariant measure corresponding to the unique fixed point of this system,
  there is a unique $m$-stationary measure $\mu$ on $\P(\R^3)$. Moreover, the system $\Gamma_m \curvearrowright (\P(\R^3),\mu)$
  admits as a natural factor the uniquely ergodic rotational action $\Gamma_m \curvearrowright (\P(\R^2),\nu)$; see 
  \cite[Theorem 1.1]{AounSert2022} for a proof thereof in a more general setting and
  \cref{ex:mainexample} for a different proof and a detailled discussion.
\end{example*}

This raises the question: Is there a version of the Furstenberg--Zimmer structure theorem for stationary 
actions?
This seems plausible because in \cite[Theorem 3.1]{Bjorklund2017}, Björklund showed that the equivalence of 
weak mixing and having no 
nontrivial isometric factors is also true for stationary actions. Indeed, the goal of the present article is to 
answer this question affirmatively.

\begin{theorem*}
  Let $(G,m) \curvearrowright \uX$ be a stationary $(G,m)$-action.
  Then it is a weakly mixing extension of a distal and measure-preserving action
  $G \curvearrowright \uX_\ud$.
\end{theorem*}
We prove this more generally for locally compact second countable groups $G$ and probability measures
$m$ on $G$ such that $\supp(m)$ generates a dense subgroup of $G$, see \cref{thm:stationarymainthm}. 
Apart from combining ideas from 
several previous works, this requires two new results that 
are of independent interest:
First, at key moments, it is necessary to 
transfer properties from the subgroup generated by 
$\supp(m)$ to its closure. To do this, 
we developed the following general continuity 
result (see \cref{thm:FEmagic}).

\begin{theorem*}
  Let $G$ be a second countable locally compact group and $G \curvearrowright \uX$ a nonsingular
  measurable action on 
  a standard probability space. Then the induced action of $G$ on $\uL^\infty(\uX)$ is strongly 
  continuous with 
  respect to the $\uL^1$-norm. 
\end{theorem*}

Second, we significantly strengthen this result by proving 
in \cref{factor-algebra-corr} that any nonsingular measurable action of a locally compact second 
countable group on a standard probability space is isomorphic to a \emph{continuous} action on a compact 
metric space. This simplifies several arguments and brings to light hidden continuity properties of nonsingular 
group actions.

The Furstenberg--Zimmer theorem for stationary actions is just one of several existing
structure theorems that have meanwhile been proven for stationary actions. 
For example, in \cite{NevoZimmer2002} it is shown that every stationary action $(G,m) \curvearrowright (X,\mu)$ for 
a connected noncompact semisimple Lie group with finite center admits a maximal projective factor of the form 
$G/Q$ where $Q$ is a parabolic subgroup. In \cite[Theorem 1]{NevoZimmer2002} they show that if all noncompact 
simple factors of $G$ are of real rank at least two, then this maximal projective factor is trivial if and 
only if the measure $\mu$ is not just stationary but $G$-invariant.
An abstract version of this reminiscent of the PI structure theorem \cite[Chapter 14]{Auslander1988}, \cite{EllisGlasnerShapiro1975}
was proven in
\cite[Theorem 4.3]{FurstenbergGlasner2010}: Every stationary action of a locally compact second countable measured 
group $(G,m)$ with an admissible probability measure $m$ is, modulo an $m$-proximal extension, a relatively 
measure-preserving extension of an $m$-proximal system.
This structure theorem was later used in \cite{FurstenbergGlasner2013} to prove a multiple recurrence theorem for 
stationary actions. Whether the generalization of the Furstenberg--Zimmer theorem to stationary actions 
can also be used to derive multiple recurrence properties and related applications is the subject of ongoing work.
We point out that simultaneously with the author, U.\ Bader and Y.\ Vigder derived essentially the same result with 
different techniques in a forthcoming work.

\textbf{Organization.} \cref{sec:classicalFZ} starts with a swift review of the classical 
Furstenberg--Zimmer theorem and the underlying key ideas that serve as a 
reference and roadmap for later sections.
\cref{sec:topmodels} lays the basis for later sections by explaining how topological models allow to 
readily switch between factors
of a nonsingular dynamical system and the $\uC^*$-subalgebras corresponding to these factors. In preparation for 
\cref{sec:stationaryFZ}, \cref{sec:nonsingularFZ} collects some tools required for working with weakly mixing 
and isometric extensions of nonsingular systems. As a side-product, we prove a version of the 
Furstenberg--Zimmer structure theorem for relatively measure-preserving
extensions of nonsingular (not necessarily stationary) actions. \cref{sec:stationarydichotomy}
and \cref{sec:stationaryrelativedichotomy} are respectively dedicated to the Kronecker dichotomy for 
stationary actions and its relative version.
Finally, \cref{sec:stationaryFZ} concludes with a proof of the structure theorem for stationary actions 
and \cref{sec:appendix} gathers relevant background information on Hilbert modules essentially taken from
\cite{EHK2021} on which the current article builds.

\textbf{Terminology \& assumptions.} Unless otherwise specified, all Hilbert spaces are assumed to be separable,
all compact spaces metrizable, all groups $G$ locally compact and second countable (lcsc), 
all Haar measures $\um_G$ to be left-invariant, and all measure spaces to be standard probability spaces. 
We usually abbreviate probability spaces as $\uX = (X, \Sigma, \mu)$ or $(X,\mu)$ when the $\sigma$-algebra is 
inessential. The typical abbreviations for measure spaces will thus be $\uX = (X,\mu)$, $\uY = (Y,\nu)$, and 
$\uZ = (Z, \zeta)$.
If $G$ is a lcsc group, a continuous action 
of $G$ on a compact space $K$ is a group action of $G$ on $K$ such that the induced map
$G\times K \to K$ is continuous. Similarly, a measurable action of $G$ on a 
probability space $\uX$ is a group action of $G$ on $\uX$ such that the induced map 
$G\times\uX\to\uX$ is measurable. If a group $G$ acts on a probability 
space $\uX$ or a compact space $K$, we sometimes write $G\curvearrowright \uX$ or $G \curvearrowright K$ to indicate 
this. A \textbf{measured group} $(G, m)$ consists of a lcsc group $G$ and a Borel probability measure $m$ on $G$ such that 
$\supp(m)$ generates a dense subgroup of $G$. (We explicitly do not require the common assumption that $\supp(m)$ generate 
$G$ as a \emph{semi}group.)
The convolution of two Borel probability measures $\mu$ and $\nu$ on a 
group $G$ is, for Borel subsets $E \subset G$, defined as
\begin{equation*}
  (\mu * \nu)(E) 
  \defeq \int_G\int_G \1_E(gh) \dnu(h) \dmu(g)
  = \int_G (g_*\nu)(E) \dmu(g).
\end{equation*}
Convolutional powers are denoted by $m^{*k}$ for $k\in \N$. If $\mu$ is a probability measure on 
a group $G$, we denote by $\check{\mu}$ the measure defined by $\check{\mu}(E) = \mu(E^{-1})$ for measurable
subsets $E \subset G$. If $(G, m)$ is a measured group and $G \curvearrowright \uX$ is an action on a 
probability space, we define the convolution $m*\mu$ as above by
\begin{equation*}
  (m * \mu)(E) 
  \defeq \int_G (g_*\mu)(E) \dm(g) \qquad \forall E \in \Sigma_X.
\end{equation*}
This can also be seen as the pushforward of $m\otimes\mu_X$ under the action $G\times X \to X$.
In case $m*\mu = \mu$, we call the action of $(G,m)$ on $\uX$ \textbf{stationary} and $\mu$ an $m$-stationary 
measure. By an \textbf{extension} between continuous actions $G \curvearrowright K$ and $G \curvearrowright L$ 
on compact spaces, we mean a $G$-equivariant continuous surjection $q \colon K \to L$. Similarly, 
by an extension between measurable actions $G \curvearrowright \uX$ and $G \curvearrowright \uY$ 
on probability spaces, we mean a measure-preserving map $\pi \colon \uX \to \uY$ 
that for every $g\in G$ 
is $g$-equivariant almost everywhere.

For a compact space $K$ and a probability space $\uX$, we denote by $\uC(K)$ the space of complex-valued
continuous functions on $K$ and by $\uL^p(\uX)$, $p\in [1, \infty]$, the space of equivalence classes
of complex-valued $p$-integrable functions. If $\phi\colon K \to L$ is a continuous map between compact spaces, 
we define the corresponding \textbf{Koopman operator} by
\begin{equation*}
  T_\phi \colon \uC(L) \to \uC(K), \quad f \mapsto f\circ \phi.
\end{equation*}
Similarly, if $\phi\colon \uX \to \uY$ is measure-preserving, we can define the Koopman operator 
on all $\uL^p$-spaces, $p\in [1, \infty]$, via
\begin{equation*}
  T_\phi \colon \uL^p(\uY) \to \uL^p(\uX), \quad f \mapsto f\circ \phi.
\end{equation*}
In this case, $T_\phi$ is a \textbf{Markov operator} ($T_\phi$ is positive and $T_\phi\1_\uX = \1_\uX)$
and even a \textbf{bi-Markov operator} ($T_\phi'\1_\uX = \1_\uX$, i.e., $T_\phi$ preserves the integral). 
If $\phi$ is merely a \textbf{nonsingular} map, i.e., $\phi_*\mu_X$ and $\mu_Y$ are mutually absolutely 
continuous, then we can still define the Koopman operator for $p = \infty$
via the same identity. A general overview about nonsingular dynamics can be found in \cite{Danilenko2011}. 
For a probability space $\uX$, we denote
by $\E\colon \uL^p(\uX) \to \C$ the expectation ($p\in [1, \infty]$), 
i.e., $\E(f) = \int_X f\dmu_X$. If $\pi\colon \uX \to \uY$ is a measure-preserving map between probability 
spaces, we denote by $\E_\uY \colon \uL^p(\uX) \to \uL^p(\uY)$ the corresponding conditional expectation
where $p\in [1,\infty]$.

\textbf{Acknowledgements.} The author thanks A.\ Gorodnik and R.\ Nagel for helpful discussions,
M.\ Björklund for suggesting the question, and U.\ Bader for an interesting discussion around 
the topic.

\section{The classical Furstenberg--Zimmer Structure Theorem} \label{sec:classicalFZ}

In this section, we recall 
the Furstenberg--Zimmer structure theorem from \cite{Furs1977} and \cite{Zimm1976} 
and briefly review the ideas behind it. The first key idea we recall is that there is a 
dichotomy between two useful classes of dynamical systems, 
\emph{isometric} and \emph{weakly mixing} systems.

\begin{definition}
  Let $G \curvearrowright \uX$ be a measure-preserving action. It is called 
  \textbf{isometric} if 
  \begin{equation*}
    \uL^2(\uX) =
    \overline{\bigcup \left\{ M \subset \uL^\infty(\uX) \mmid \begin{matrix}
                                                          M \text{ is a } G\text{-invariant, } \\
                                                          \text{finite-dimensional } \text{subspace}
                                                        \end{matrix}
    \right\}}^{\|\cdot\|_{\uL^2(\uX)}}.
  \end{equation*}
  It is called \textbf{weakly mixing} if the diagonal action $G \curvearrowright \uX\times\uX$
  is ergodic.
\end{definition}

Isometric systems earn their name since they admit isometric topological models;
weakly mixing systems admit a more intuitive characterization that, however,
shall not be relevant in this article, see \cite[Definition 9.13, Theorem 9.19]{EFHN2015}.
Every measure-preserving system has a maximal isometric factor, the so-called 
\textbf{Kronecker factor}, and it satisfies the following dichotomy.

\begin{theorem}[Kronecker dichotomy]\label{thm:krodichotomy}
  Let $G \curvearrowright \uX$ be a measure-preserving action. Then either $G \curvearrowright \uX$ 
  is weakly mixing or its Kronecker factor is nontrivial.
\end{theorem}

\begin{proof}
The standard proof of \cref{thm:krodichotomy} requires the following three ingredients that can be found
in the appendix, see \cref{rem:HSbasics}. To state them, let $G\curvearrowright\uX$ and $G\curvearrowright\uY$ 
be measure-preserving $G$-actions.
\begin{itemize}
 \item The assignment
 \begin{align*}
   I\colon \uL^2(\uX\times\uY) \to \HS(\uL^2(\uX), \uL^2(\uY)), \quad (I_kf)(y) \defeq \int_X k(x, y) f(x) \dmu(x)
 \end{align*}
 defines an isomorphism between $\uL^2(\uX\times\uY)$ and the Hilbert space $\HS(\uL^2(\uX), \uL^2(\uY))$ 
 of Hilbert--Schmidt operators between $\uL^2(\uX)$ and $\uL^2(\uY)$.
 \item A function $k\in\uL^2(\uX\times \uY)$ is $G$-invariant if and only if $I_k T_g = T_g I_k$ for every $g\in G$.
 \item A self-adjoint Hilbert--Schmidt operator $K \in \HS(\uL^2(\uX))$ admits (by virtue of the spectral theorem) 
 a canonical decomposition into finite-rank Hilbert--Schmidt operators. By continuous functional calculus, 
 a bounded operator (such as $T_\phi$) that commutes with $K$ also commutes with the finite-rank operators 
 occurring in its canonical decomposition.
\end{itemize}

With these prerequisites, the proof of \cref{thm:krodichotomy} is straight-forward: If $G\curvearrowright \uX$ 
is not weakly mixing,
then there is a nonconstant $G$-invariant function $k\in \uL^2(\uX\times\uX)$. The Hilbert--Schmidt operator $I_k$
is then $G$-equivariant, i.e., $I_k T_g = T_g I_k$ for all $g\in G$, and by decomposing 
\begin{equation*}
  I_k = \frac{I_k + I_k^*}{2} + \frac{I_k - I_k^*}{2}
\end{equation*}
we may assume that $I_k$ is self-adjoint. 
Thus, the action of $G$ commutes with the projections onto the eigenspaces of $I_k$, i.e., it 
leaves the finite-dimensional eigenspaces of $I_k$ invariant. Furthermore, it 
can be shown that every finite-dimensional $G$-invariant subspace
$M \subset \uL^2(\uX)$ can be approximated by finite-dimensional $G$-invariant
subspaces in $\uL^\infty(\uX)$, see \cref{lem:moduleseverywhere}. Consequentially, the Kronecker factor of $G \curvearrowright\uX$ must be nontrivial.

Conversely, if the Kronecker factor of $G\curvearrowright\uX$ is nontrivial, 
then there exists a nontrivial
finite-dimensional $G$-invariant subspace $M \subset \uL^\infty(\uX)$. 
The orthogonal projection 
$P_M \colon \uL^2(\uX) \to \uL^2(\uX)$ onto $M$ then defines a nontrivial $G$-equivariant Hilbert--Schmidt
operator and as such must be of the form $P_M = I_k$ for a nontrivial $G$-invariant function 
$k\in \uL^2(\uX\times\uX)$. Thus, $G \curvearrowright\uX$ cannot be weakly mixing in this case.
\end{proof}

Historically, Furstenberg proved multiple recurrence properties for dynamical systems
to derive Szemerédi's theorem. It is easy to check these multiple recurrence properties 
for isometric and weakly mixing system, but not every system decomposes into these two 
types of systems. However, every system can be decomposed into isometric and weakly mixing
\emph{extensions}. We recall the relevant definitions and, in preparation for later
sections, state them for nonsingular actions (with the exception of weakly mixing extensions which
are discussed in a later section).

\begin{definition}
   Let $\pi\colon \uX \to \uY$ be a measure-preserving map. 
   \begin{enumerate}[(i)]
    \item The \textbf{conditional $\uL^2$-space} of the extension is
    \begin{equation*}
      \uL^2(\uX|\uY) \defeq \left\{ f \in \uL^2(\uX) \mmid \E_\uY(|f|^2) \in \uL^\infty(\uY) \right\}.
    \end{equation*}
    It is a \emph{Hilbert module} over $\uL^\infty(\uY)$ with multiplication and $\uL^\infty(\uY)$-valued
    scalar product given by
    \begin{alignat*}{3}
      \cdot \colon \uL^\infty(\uY)\times\uL^2(\uX|\uY) &\to \uL^2(\uX|\uY), \quad 
      (f, g) \mapsto (g\circ f)\cdot g, \\
      (\cdot|\cdot)_\uY\colon \uL^2(\uX|\uY)\times \uL^2(\uX|\uY) &\to \uL^2(\uX|\uY), \quad 
      (f,g) \mapsto \E_\uY(f\overline{g}).
    \end{alignat*}
  \end{enumerate}
  Now let $\pi\colon \uX \to \uY$ be an extension of nonsingular $G$-actions.
  \begin{enumerate}[resume, label=(\roman{enumi})]
    \item The extension is called \textbf{ergodic} if every a.e.\ $G$-invariant 
    measurable set $A \subset X$
    is, up to some nullset, of the form $A = \pi^{-1}(B)$ for an a.e.\ 
    $G$-invariant measurable set $B \subset Y$.
    \item The extension is called \textbf{isometric} if
    \begin{equation*}
      \uL^2(\uX) = 
      \overline{\bigcup \left\{ \Gamma \subset \uL^\infty(\uX) \mmid \begin{matrix}
                                                            \Gamma \text{ is a } G\text{-invariant, finitely-} \\
                                                            \text{generated } \uL^\infty(\uY)\text{-submodule}
                                                          \end{matrix}
      \right\}}^{\|\cdot\|_{\uL^2(\uX)}}.
    \end{equation*}
    The system $G\curvearrowright\uX$ is called \textbf{isometric} if the extension $\uX \to \mathrm{pt}$ is 
    isometric, i.e., if
    \begin{equation*}
      \uL^2(\uX) = 
      \overline{\bigcup \left\{ F \subset \uL^\infty(\uX) \mmid \begin{matrix}
                                                            F \text{ is a } G\text{-invariant, finite-} \\
                                                            \text{dimensional subspace}
                                                          \end{matrix}
      \right\}}^{\|\cdot\|_{\uL^2(\uX)}}.
    \end{equation*}    \item The extension is called \textbf{distal} if there are an ordinal $\eta_0$ 
    and a projective system $((X_\eta)_{\eta \leq \eta_0}, (\pi_\eta^\sigma)_{\eta\leq \sigma\leq \eta_0})$
    of nonsingular $G$-actions such that 
    \begin{itemize}
     \item $\pi_1^{\eta_0} = \pi$,
     \item $\pi_\eta^{\eta+1}$ is an isometric extension for every $\eta < \eta_0$,
     \item $X_\eta = \lim_{\mu < \eta} \uX_\mu$ for every limit ordinal $\eta \leq \eta_0$.
    \end{itemize}
  \end{enumerate}
  Finally, assume that $\pi\colon \uX \to \uY$ is an extension of measure-preserving $G$-actions.
  \begin{enumerate}
    \item The extension is called \textbf{weakly mixing} if the relatively independent 
    joining $\pi \times_\uY \pi \colon \uX \times_\uY \uX \to \uY$ is an ergodic extension of $\uY$.
  \end{enumerate}
\end{definition}

\begin{remark}
  From a conceptual point of view, the definition of isometric extensions of nonsingular actions
  should not involve $\uL^2(\uX)$ since in the nonsingular realm, there is no induced action 
  on $\uL^2(\uX)$ or $\uL^2(\uX|\uY)$. It would be more appropriate to
  work purely in $\uL^\infty(\uX)$ and use 
  the equivalent definition in terms of 
  \emph{order closure} that requires that
  \begin{equation*}
        \uL^\infty(\uX) = 
      \operatorname{ocl}\bigcup \left\{ \Gamma \subset \uL^\infty(\uX) \mmid \begin{matrix}
                                                            \Gamma \text{ is a } G\text{-invariant, finitely-} \\
                                                            \text{generated } \uL^\infty(\uY)\text{-submodule}
                                                          \end{matrix}
      \right\}.
  \end{equation*}
  See the appendix for a discussion of order-convergence (which, for bounded sequences,
  is the same as almost everywhere convergence, see \cite[Lemma 7.5]{EHK2021}). We shall use 
  $\uL^2$-closures for the sake of accessibility but the reader familiar with 
  order convergence may wish to instead take the less common but 
  conceptually cleaner approach.
\end{remark}

With these notions, the Furstenberg--Zimmer structure theorem, inspired by Furstenberg's earlier 
structure theorem \cite{Furstenberg1963} for distal actions in topological dynamics, can be stated as follows.

\begin{theorem}\label{thm:classicalFZ}
  Let $\pi\colon \uX \to \uZ$ be an extension of measure-preserving $G$-actions. Then there are 
  a weakly mixing extension $\alpha\colon \uX \to \uY$ and a distal extension $\beta\colon \uY \to \uZ$ 
  such that the diagram
  \begin{equation*}
   \xymatrix{
    \uX \ar[rd]_-\alpha \ar[rr]^-\pi && \uZ \\
    & \uY \ar[ru]_-\beta &
   }
  \end{equation*}
  commutes.
\end{theorem}

The proof reduces essentially to the following generalization of the Kronecker dichotomy 
\cref{thm:krodichotomy} to extensions.

\begin{theorem}[Relative Kronecker dichotomy]\label{thm:relkrodichotomy}
  An extension $\pi \colon \uX \to \uZ$ of measure-preserving $G$-actions is weakly mixing 
  if and only if there is no 
  isometric intermediate extension of $\uZ$.
\end{theorem}

\begin{remark}\label{rem:ingredients}
The proof of \cref{thm:relkrodichotomy} can be carried out in complete analogy 
to its special case \cref{thm:krodichotomy}. To that end, the $\uL^2$-space 
$\uL^2(\uX)$ is replaced by the conditional 
$\uL^2$-space $\uL^2(\uX|\uZ)$ 
and the conditionally independent joining $\uX \times_\uZ \uX$ (see \cite[Examples 6.3]{Glasner2003})
replaces the product $\uX\times\uX$.
Since the conditional $\uL^2$-space $\uL^2(\uX|\uZ)$ forms a so-called \emph{Hilbert module} over the 
$\uC^*$-algebra $\uL^\infty(\uZ)$, 
one can make use of results from the theory of Hilbert modules.
Unfortunately, general Hilbert modules fail to \enquote{relativize} Hilbert spaces in 
many important ways, as evidenced by the failure of results such as the Fr\'{e}chet--Riesz
representation theorem, complementability of closed submodules, the spectral theorem, and many other essential
parts of Hilbert space theory. However, $\uL^2(\uX|\uZ)$ belongs to a special class 
of Hilbert modules, so-called \emph{Kaplansky--Hilbert modules}, which do not suffer from these
problems. We do not enter into the details of Kaplansky--Hilbert modules here
and refer the reader to \cite{EHK2021} for proofs and details. In light of this, it is 
not surprising that the following can be shown for extensions.
\begin{itemize}
 \item There is a natural notion of Hilbert--Schmidt homomorphisms on $\uL^2(\uX|\uZ)$
 and the assignment
 \begin{align*}
   I\colon \uL^2(\uX\times_\uZ\uX) \to \HS(\uL^2(\uX|\uZ)), \quad (I_kf)(x) \defeq \int_{X_{\pi(x)}} k(x, y) f(y) \dmu_{\pi(x)}(y)
 \end{align*}
 defines an isomorphism between $\uL^2(\uX\times\uX)$ and the Hilbert module $\HS(\uL^2(\uX|\uZ))$ 
 of Hilbert--Schmidt homomorphisms on $\uL^2(\uX|\uZ)$. The reader may take this isomorphism as a 
 definition for now; a more general version \cref{thm:KHiso} is proven in the appendix.
 \item A function $k\in\uL^2(\uX\times_\uZ \uX)$ is $G$-invariant if and only if $I_k T_g = T_g I_k$.
 \item A self-adjoint Hilbert--Schmidt homomorphism $K\in\uL^2(\uX|\uZ)$ admits (by virtue of a spectral theorem) 
 a canonical decomposition into finite-rank Hilbert--Schmidt homomorphisms. By some functional calculus arguments, 
 a bounded operator (such as $T_g$ for $g\in G$) that commutes with $K$ also commutes with the finite-rank homomorphisms 
 occurring in its canonical decomposition.
\end{itemize}
With these ingredients, the proof of \cref{thm:relkrodichotomy} can be done in complete analogy to the 
proof of \cref{thm:krodichotomy}. In the following sections, we adapt these ideas to prove a similar dichotomy 
and structure theorem for stationary actions.
\end{remark}

Above, we implicitly used the measure disintegration theorem which we quickly recall
for later reference.

\begin{theorem}\label{thm:disintegration}
  Let $\pi\colon (K, \mu) \to (L, \nu)$ be a continuous 
  measure-preserving map between compact metrizable probability spaces.
  Then there exist a $\nu$-a.e.\ uniquely determined family 
  $\{\mu_l \mid l\in F\}$ of probability measures on $K$ with 
  $\supp(\mu_l) \subset K_l \defeq \pi^{-1}(l)$ for $\nu$-a.e.\ $l\in L$ such that
  for any measurable bounded function $f\colon K \to \C$, the 
  assignment
  \begin{equation*}
    l \mapsto \int_{K_l} f\dmu_l
  \end{equation*}
  is measurable and satisfies
  \begin{equation*}
    \int_L \int_{K_l} f\dmu_l \dnu(l) = \int_K f\dmu.
  \end{equation*}
  Moreover, $\pi$ is essentially invertible if and only if for $\nu$-almost
  every $l\in L$ the fiber measure $\mu_l$ is a Dirac mass.
\end{theorem}
\begin{proof}
  The first part is the usual formulation and can be found in many sources, see \cite[Theorem 5.3.1]{AGS2008}; we 
  include a short proof of the invertibility statement for the sake of completeness. First,
  suppose $\pi$ is essentially invertible, i.e., there are sets $A \subset K$ 
  and $B \subset L$ of full measure such that $\pi|_A\colon A \to B$ is bijective
  and its inverse $s\colon B \to A$ is measurable. Then for every $f\in \uC(K)$,
  \begin{equation*}
   \langle f, \mu\rangle = \langle f, s_*\pi_*\mu\rangle = \langle f\circ s, \nu\rangle
  \end{equation*}
  and so it follows that any disintegration $(\mu_l)_{l\in L}$ of $\mu$ must 
  agree $\nu$-almost everywhere with the Dirac measures $(\delta_{s(l)})_{l\in B}$.
  
  Conversely, suppose $\mu$ admits a disintegration of the form $(\mu_l)_{l\in L}$
  such that for some set $B \subset L$ of full measure and some function $s\colon B \to K$,
  $\mu_l = \delta_{s(l)}$ for all $l\in B$. Then $\delta_{s(\cdot)}\colon B \to \uC(K)'$ is weak*-measurable,
  i.e., measurable w.r.t.\ the Borel $\sigma$-algebra of the set $\{\delta_x \mid x\in K\}$ equipped
  with the weak* topology. Since this set is homeomorphic to $K$, it follows that $s\colon B \to K$ 
  is Borel measurable. Now, for all $l \in B$ one has $(\pi\circ s)(l) = l$ since 
  $\supp(\delta_{s(l)}) = \supp(\mu_l) \subset K_l$. Thus, $\pi\circ s = \id_L$ $\nu$-a.e. Finally, 
  \begin{equation*}
   \mu([s\circ \pi \neq \id_K])
   = \int_L \mu_l([s\circ \pi \neq \id_K])\dnu(l) 
   = \int_B \underbrace{\delta_{s(l)}([s\circ \pi \neq \id_K])}_{=0} \dnu(l)
   = 0.
  \end{equation*}
  Thus, also $s\circ \pi = \id_K$ $\mu$-a.e., which shows that $s$ is an essential inverse for $\pi$.
\end{proof}

\begin{remark}\label{rem:stationary_dis}
  The uniqueness property above in particular shows that if $\pi\colon (K, \mu) \to (L, \nu)$ 
  is an extension of measure-preserving $G$-actions for a lcsc group $G$, then 
  the disintegration is also equivariant $\nu$-a.e.\ in the sense that $s_*\mu_{s^{-1}l} = \mu_{l}$ 
  for all $s\in G$ and $\nu$-a.e.\ $l\in L$. If, more generally, $\pi$ is an extension of a 
  measure-preserving $G$-action on $(L,\nu)$ by a $(G,m)$-stationary $G$-action on $(K,\mu)$, then the uniqueness of 
  the disintegration yields that $\int_G s_*\mu_{s^{-1}l}\dm(s) = \mu_l$ for $\nu$-a.e.\ $l\in L$. 
  If both $\mu$ and $\nu$ are merely stationary, then one can still observe that 
  \begin{align*}
    \int_G s_*\mu \dm(s)
    = \int_G \int_L s_*\mu_l \dnu(l) \dm(s)
    = \int_G \int_L s_*\mu_{s^{-1}l}\frac{\ud s_*\nu}{\ud\nu}(l)\dnu(l) \dm(s)
  \end{align*}
  where we used the \textbf{Radon--Nikodym cocycle} (see \cite[Definition 1.3]{NevoZimmer2000})
  \begin{equation*}
    \rho_\nu \colon G\times L \to \R, \quad \rho_\nu(s, l) \defeq \frac{\ud (s^{-1})_*\nu}{\ud\nu}(l).
  \end{equation*}
  Thus, the uniqueness of disintegrations yields
  \begin{equation*}
    \mu_l = \int_G \left(\frac{\ud s_*\nu}{\ud\nu}(l)\right) s_*\mu_{s^{-1}l}\dm(s) \qquad \text{for $\nu$-a.e. } l\in L. 
  \end{equation*}
  This observation will be essential in the proof of \cref{lem:bundle_inv}.
\end{remark}

\section{Topological models for nonsingular actions} \label{sec:topmodels}

For the proof of the main result, it will be necessary to construct 
\emph{topological models}. A standard technique 
for doing this is the correspondence between factors of a 
system $\uX$ and 
$\uC^*$-subalgebras of $\uL^\infty(\uX)$. It allows for elegant constructions 
and arguments involving factors/topological models and shall also be used in later
sections.
The ideas are not novel but usually only presented for measure-preserving actions, 
see, e.g., \cite[Chapter 12]{EFHN2015}. In particular, that every 
nosingular action of a locally compact second countable group $G$ on a probability
space $\uX$ admits a compact metric topological model on which $G$ acts \emph{continuously}
appears to be a new result. It hinges on the following surprising continuity property
of nonsingular actions that is adapted from \cite[Theorem 10.2.3]{HillePhillips}.

\begin{theorem}\label{thm:FEmagic}
  Let $G$ be a second countable locally compact group and $G \curvearrowright \uX$ a nonsingular
  measurable action on 
  a standard probability space. Then the induced action of $G$ on $\uL^\infty(\uX)$ is strongly 
  continuous with 
  respect to the $\uL^1$-norm.
\end{theorem}
\begin{proof}
  Since the $G$-action is nonsingular, we can define the group homomorphism
  \begin{align*}
     \pi\colon G \to \mathscr{L}(\uL^\infty(\uX)) \qquad g \mapsto T_{g^{-1}}.
  \end{align*}
  Observe that for $f\in \uL^\infty(\uX)$, we can regard the map $\pi(\cdot)f$ 
  as taking values in $\uL^1(\uX)$. Moreover, for every
  $h\in \uL^\infty(\uX)$ the expression
  \begin{align*}
    \langle h, \pi(g)f\rangle = \int_X f(g^{-1}x)h(x) \dmu(x)
  \end{align*}
  depends measurably on $g$ because the action $G\times X\to X$, $(g, x)\mapsto gx$
  is measurable by assumption. (See \cite[Theorem 1.7.15]{TaoMeasure} for the 
  measure-theoretic details.) This means that $\pi(\cdot)f$ is, regarded 
  as a map to $\uL^1(\uX)$, weakly measurable and \cite[Theorem 3.5.3]{HillePhillips} yields that it is even 
  strongly measurable\footnote{This means that 
  for each $f\in \uL^\infty(\uX)$, the map $g\mapsto \pi(g)f$ is the almost-everywhere limit 
  limit of finitely-valued measurable functions.} because 
  $\uL^1(\uX)$ is separable. Thus, regarding $\pi(\cdot)f$ as a 
  map with values in $\uL^1(\uX)$, we may integrate it in the sense of Bochner 
  integrals.
  
  Now, let $K\subset G$ be a compact set of Haar measure 1, $U \subset G$ 
  a compact symmetric unit neighborhood,
  $h\in U$, $f\in \uL^\infty(\uX)$ and note that
  \begin{align*}
    \pi(h)f - f 
    = \int_K \pi(g)\left[\pi(g^{-1}h)f - \pi(g^{-1})f\right] \ud\um_G(g).
  \end{align*}
  Then setting $F\colon UK\to \uL^1(\uX)$, $F(g) = \pi(g^{-1})f$, we obtain that 
  $F$ is strongly measurable and bounded 
  since $\pi$ has these properties on the compact set $UK$. Extending 
  $F$ to all of $G$ by zero, we can regard it as an element of $\uL^1(G,\um;\uL^1(\uX))$ 
  that satisfies $F(h^{-1}g) = \pi(g^{-1}h)f$ for all $h\in U$, $g\in K$.
  Since the left-regular representation of $G$ on $\uL^1(G,\um_G;\uL^1(\uX))$ 
  is strongly continuous (see \cref{lem:Elrrep} below), we know that 
  \begin{equation*}
   \lim_{h\to e} \int_G \left\|F(h^{-1}g) - F(g)\right\|_{\uL^1(\uX)} \ud\um_G(g) = 0.
  \end{equation*}
  Since for $h\in U$
  \begin{equation*}
   \int_K \left\|F(h^{-1}g) - F(g)\right\|_{\uL^1(\uX)} \ud\um_G(g)
   = \int_K \left\|\pi(g^{-1}h)f - \pi(g^{-1})f\right\|_{\uL^1(\uX)} \ud\um_G(g),
  \end{equation*}
  we can conclude that 
  \begin{equation*}
    \lim_{h\to e} \int_K \left\|\pi(g^{-1}h)f - \pi(g^{-1})f\right\|_{\uL^1(\uX)} \ud\um_G(g) = 0.
  \end{equation*}
  Now, fix a sequence $(h_n)_n$ in $G$ with $h_n \to e$. Then the above convergence in $\uL^1(K, \um_G|_K)$ 
  implies that, by passing 
  to a subsequence, we may assume that
  \begin{equation*}
    \lim_{n\to\infty} \left\|\pi(g^{-1}h_n)f - \pi(g^{-1})f\right\|_{\uL^1(\uX)} = 0 \qquad \text{for } \um_G\text{-a.e. } g\in K.
  \end{equation*}
  By \cref{lem:bpcont}, it follows that also 
  \begin{equation*}
    \lim_{n\to\infty} \left\|\pi(g)\left[\pi(g^{-1}h_n)f - \pi(g^{-1})f\right]\right\|_{\uL^1(\uX)} = 0 \qquad \text{for } \um_G\text{-a.e. } g\in K.
  \end{equation*}
  Since 
  \begin{equation*}
    \left\|\pi(g)\left[\pi(g^{-1}h_n)f - \pi(g^{-1})f\right]\right\|_{\uL^\infty(\uX)} \leq 2\|f\|_{\uL^\infty(\uX)},
  \end{equation*}
  we may invoke the dominated convergence theorem to conclude that
  \begin{align*}
    \|\pi(h_n)f - f\|_{\uL^1(\uX)}
    \leq \int_K \left\|\pi(g)\left[\pi(g^{-1}h_n)f - \pi(g^{-1})f\right]\right\|_{\uL^1(\uX)} \ud\um_G(g)
    \xrightarrow[n\to\infty]{}0.
  \end{align*}
  Since the sequence $(h_n)_n$ was chosen arbitrarily, it follows that $\pi(\cdot)f$ is 
  $\uL^1$-continuous at $e\in G$. 
  If $h_n \to h$ is 
  a general convergent sequence in $G$, use the identity 
  \begin{equation*}
    \pi(h_n)f - \pi(h)f = \pi(h)\left(\pi(h^{-1}h_n)f - f\right)
  \end{equation*}
  and \cref{lem:bpcont} to conclude that $\pi(\cdot)f$ is $\uL^1$-continuous on all of $G$.
\end{proof}

\begin{lemma}\label{lem:bpcont}
  Let $\phi \colon \uX \to \uX$ be a nonsingular map. Then the Koopman operator
  \begin{equation*}
    T_\phi\colon \uL^\infty(\uX) \to \uL^\infty(\uX)
  \end{equation*}
  maps $\|\cdot\|_{\uL^1(\uX)}$-convergent sequences that are $\|\cdot\|_{\uL^\infty(\uX)}$-bounded
  to $\|\cdot\|_{\uL^1(\uX)}$-convergent sequences.
\end{lemma}
\begin{proof}
   Let $(f_n)_n$ be a sequence in $\uL^\infty(\uX)$ that is uniformly bounded and converges
   to $f\in \uL^\infty(\uX)$ in the $\uL^1$-norm. It suffices to show that every subsequence
   of $(T_\phi f_n)_n$ has a subsquence that converges to $T_\phi f$ in $\uL^1$. 
   Thus, replacing $(f_n)_n$ with a subsequence, we may 
   assume that $f_n \to f$ almost everywhere. Since $\phi$ is nonsingular, it follows that 
   $T_\phi f_n \to T_\phi f$ almost everywhere and the uniform boundedness assumption 
   combined with the dominated convergence theorem yields the claim.
\end{proof}

The following lemma is certainly not new and merely included for the sake of completeness.

\begin{lemma}\label{lem:Elrrep}
  Let $G$ be a locally compact group, $E$ a Banach space and 
  fix a left-invariant Haar measure $\mathrm{m}$ on $G$.
  Let $p\in[1,\infty)$.
  \begin{enumerate}[(i)]
    \item The map 
    \begin{align*}
      L\colon G \to \mathscr{L}\left(\uL^p(G, \um; E)\right), \quad (L_hf)(g) \defeq f(h^{-1}g)
    \end{align*}
    is a strongly continuous representation of $G$.
    \item The map 
    \begin{align*}
      R\colon G \to \mathscr{L}\left(\uL^p(G, \um; E)\right), \quad (R_hf)(g) \defeq f(gh)
    \end{align*}
    is a strongly continuous representation of $G$ (even though
    the measure is only left-invariant).
  \end{enumerate}
  We call these representations the $E$-left-regular and $E$-right-regular
  representations of $G$.
\end{lemma}
\begin{proof}
  We only need to show that $L$ and $R$ are strongly continuous, so
  pick $f\in\uL^p(G, \um; E)$.
  We cannot directly invoke the dominated convergence theorem to
  conclude that  
  \begin{align*}
    \lim_{h\to0} \int_G \|f(h^{-1}g)-f(g)\|^p\dm(g)
    = \lim_{h\to 0} \|L_hf - f\|_{\uL^p}^p
    = 0
  \end{align*}
  since 
  the function $f$ is not continuous. But note that the space 
  $\uC_\uc(G, E)$ of continuous, compactly supported $E$-valued 
  functions on $G$ is dense in $\uL^p(G,\um; E)$: To see this,
  it suffices to verify that for every measurable set $A \subset G$
  of finite measure and every $e\in E$, the function $e\1_A$ 
  can be approximated by $\uC_\uc(G,E)$ which follows since $\1_A$ 
  can be approximated in $\uL^p(G,\um_G)$ by functions in $\uC_\uc(G)$.
  
  Now, pick $\epsilon> 0$
  and a function $f_\epsilon \in\uC_\uc(G, E)$ such that 
  $\|f - f_\epsilon\|_{\uL^p}^p \leq \epsilon$ and hence 
  \begin{align*}
    \limsup_{h\to 0} \int_G \|f(h^{-1}g) - f(g)\|^p\dg 
    &\leq \limsup_{h\to 0} \int_G \|f(h^{-1}g) - f_\epsilon(h^{-1}g)\|^p\dg \\
    & \qquad +\limsup_{h\to 0} \int_G \|f_\epsilon(h^{-1}g) - f_\epsilon(g)\|^p\dg \\
    & \qquad +\limsup_{h\to 0} \int_G \|f_\epsilon(g) - f(g)\|^p\dg \\
    & \leq \epsilon + 0 + \epsilon \\
    &= 2\epsilon.
  \end{align*}
  Since $\epsilon > 0 $ was arbitrary, $L$ is strongly continuous. As for
  $R$, note that even though $R_g$ is not an isometry for all $g\in G$ in general, 
  $R_g$ still is a well-defined operator on $\uL^p(G, \um; E)$ of norm 
  $\leq \Delta(g^{-1})^{\nicefrac{1}{p}}$ by elementary properties of the modular
  function $\Delta\colon G\to \R_{>0}$. Now the proof is completely analogous to 
  that for $L$, except that the modular function appears in passing in the estimate
  \begin{align*}
    \limsup_{h\to 0} \int_G \|f(gh) - f_\epsilon(gh)\|^p\dg
    \leq \limsup_{h\to 0}\Delta\big(h^{-1}\big)\int_G \|f(g) - f_\epsilon(g)\|^p\dg 
    \leq \epsilon
  \end{align*}
  but disppears again since it is continuous.
\end{proof}

We are now ready to discuss the construction of topological models. Note that our topological
models not merely represent a measurable action on an abstract probability space as a measurable
action on a compact metric probability space. Rather, we show that nonsingular actions by locally compact
second countable groups always can be represented as \emph{continuous} actions on compact metric spaces.

\begin{construction}[Topological Models]\label{factor-algebra-corr}
  Let $G \curvearrowright (X, \mu)$ be a nonsingular action and $\pi\colon (X, \mu) \to (Y, \nu)$ 
  a factor map. Then $\mathcal{A}_\pi \defeq T_\pi(\uL^\infty(Y, \nu))$ defines a $G$-invariant
  unital $\uC^*$-subalgebra of $\uL^\infty(X, \mu)$. Conversely, if $\mathcal{A} \subset \uL^\infty(X, \mu)$
  is a $G$-invariant unital $\uC^*$-subalgebra of $\uL^\infty(X, \mu)$, then by Gelfand's representation 
  theorem there exist a compact space $K$ and a $\uC^*$-isomorphism $\Phi\colon \mathcal{A} \to \uC(K)$, 
  see \cite[Section 1.4]{Dixmier1977} or \cite[Section 4.4]{EFHN2015}. Let $\nu \defeq \Phi_*\mu\in\uC(K)'$ 
  be the linear form 
  that $\mu$ induces on $\uC(K)$ via $\Phi$. By the Riesz--Markov--Kakutani representation theorem,
  we can identify $\nu$ with a unique probability measure on $K$. It is easy to verify that the measure 
  $\nu$ must have full support since the corresponding linear form is strictly positive, i.e., 
  $\langle f, \nu \rangle > 0$ for every nonzero $0 \leq f \in\uC(K)$. Since $\Phi$ preserves the measures, 
  it is an $\uL^1$-isometry and thus we may extend it as $\Phi\colon \uL^1(X,\mu) \to \uL^1(K, \nu)$.
  
  Moreover, for every $g\in G$,
  let $S_g\colon \uC(K) \to \uC(K)$ be the $\uC^*$-automorphism on $\uC(K)$ induced by $T_g$ via $\Phi$.
  Every $\uC^*$-automorphism $S$ of $\uC(K)$ is of the form $S = T_\phi$ for a uniquely determined 
  homeomorphism $\phi\colon K\to K$ and therefore the operators $(S_g)_{g\in G}$ are the Koopman operators
  associated with a uniquely determined continuous action $G \curvearrowright K$. To see that $\nu$ is 
  nonsingular with respect to this $G$-action, we show that for every $g\in G$ there is an $h_g \in \uL^1(K, \nu)$
  such that $g_*\nu = h_g\nu$. To that end, denote for every $g\in G$ by $\xi_g\in\uL^1(X,\mu)$ the 
  unique function with $g_*\mu = \xi_g\mu$ that exists by virtue of the Radon--Nikodym theorem. If we denote
  by $\E_{\mathcal{A}}\colon \uL^1(X,\mu) \to \overline{\mathcal{A}}$ the conditional expectation
  onto the $\uL^1$-closure of $\mathcal{A}$, then for every $f\in\uC(K)$
  \begin{align*}
    \int_K S_g f \dnu 
    &= \int_X T_g \Phi^{-1}(f) \dmu 
    = \int_X \Phi^{-1}(f) \xi_g \dmu \\
    &= \int_X \Phi^{-1}(f) \E_{\mathcal{A}}(\xi_g)\dmu 
    = \int_K f \underbrace{\Phi(\E_{\mathcal{A}}(\xi_g))}_{\eqdef h_g} \dnu.
  \end{align*}
  Thus, for every $g\in G$ there is a unique function $h_g \in \uL^1(K, \nu)$ with $g_*\nu = h_g\nu$
  which shows that the $G$-action on $(K,\nu)$ is nonsingular. 
  
  Finally, we discuss the existence of a point factor 
  map $\pi\colon (X,\mu) \to (K, \nu)$:
  If $\uY$ and $\uZ$ are standard probability spaces, an operator 
  $T\colon \uL^1(\uY) \to \uL^1(\uZ)$ is induced by an (almost everywhere uniquely determined) 
  measure-preserving measurable map $\phi \colon \uZ \to \uY$ if and only if 
  \begin{enumerate}[1)]
   \item $T\1 = \1$ and $\int_\uZ Tf \dmu_\uZ = \int_\uY f \dmu_\uY$ for all $f\in \uL^1(\uY)$,
   \item $T|f| = |Tf|$ for all $f\in \uL^1(\uY)$,
  \end{enumerate}
  see von Neumann's theorem \cite[Proposition 7.19, Theorem 7.20]{EFHN2015}.
  Now, if $\mathcal{A} \cong \uC(K)$ is a separable $\uC^*$-algebra, then $K$ is a compact 
  metrizable space and $(K,\nu)$ a standard probability space, so that von Neumann's theorem 
  shows that $\Phi^{-1}\colon \uL^1(K, \nu) \to \uL^1(X,\mu)$ is induced by a measure-preserving
  measurable map $\pi\colon (X,\mu)\to (K, \nu)$. It is straightforward to verify that $\pi$ 
  intertwines the $G$-actions and so $\pi$ is a factor map. Thus, if we say that to 
  $G$-invariant $\uC^*$-subalgebras $\mathcal{A},\mathcal{B} \subset \uL^\infty(X,\mu)$ are 
  \textbf{equivalent} if they have the same $\uL^1$-closure, then separable $\mathcal{A}$ and 
  $\mathcal{B}$ are equivalent if and only if there is a point isomorphism between the corresponding
  topological models.
  
  This almost completes the correspondence 
  between factors and subalgebras, except for one problem: not every $\uC^*$-subalgebra 
  of $\uL^\infty(X,\mu)$ is separable. Therefore, we would like to know that for every $G$-invariant 
  $\uC^*$-subalgebra of $\uL^\infty(X,\mu)$, there is an equivalent $G$-invariant $\uC^*$-subalgebra
  that is separable with respect to $\|\cdot\|_{\uL^\infty(X,\mu)}$ to conclude that we can always
  construct compact metric models.
  
  Let $\mathcal{A} \subset \uL^\infty(X,\mu)$ be a $G$-invariant $\uC^*$-subalgebra. By
  taking a topological model $G \curvearrowright (K,\nu)$ associated to 
  $\mathcal{A}$ and regarding $\mathcal{A}$ as embedded in $\uL^\infty(K,\nu)$, we may assume
  that $\mathcal{A}$ is dense in $\uL^1(X,\mu)$.
  Since $(X,\mu)$ is a separable measure space,
  $\uL^1(X,\mu)$ is separable. It is not true that $G$ acts strongly continuously on $\uL^\infty(X,\mu)$
  with respect to the norm $\|\cdot\|_{\uL^\infty(\uX)}$ but by 
  \cref{lem:uniformcont} below, 
  the set
  \begin{equation*}
    \mathcal{M} \defeq \left\{ f\in \uL^\infty(\uX) \mmid g \mapsto T_{g^{-1}}f \text{ is } \|\cdot\|_{\uL^\infty(\uX)}\text{-continuous}\right\}
  \end{equation*}
  is a $\uC^*$-subalgebra of $\uL^\infty(\uX)$ that is dense in $\uL^1(\uX)$. Thus,
  if $S \subset G$ is a countable dense subset, then we can find a separable $\uC^*$-subalgebra
  $\mathcal{A}' \subset \mathcal{M}$ that is $S$-invariant and dense in $\uL^1(\uX)$. By the 
  strong continuity, this subalgebra must then also be $G$-invariant.
  Thus, we have shown that $\mathcal{A}$ has an equivalent subalgebra $\mathcal{A}'$ that 
  is separable. 
  There is no canonical choice
  for this subalgebra but if $\mathcal{B}\subset \uL^\infty(X,\mu)$
  is another such subalgebra and $\pi'\colon (X,\mu) \to (K', \nu')$ is the corresponding factor 
  map, then von Neumann's theorem from above shows that there is a unique 
  essentially invertible measure-preserving $G$-equivariant map $\phi\colon (K,\nu) \to (K',\nu')$ such that 
  the following diagram commutes:
  \begin{equation*}
   \xymatrix{
    & (X,\mu) \ar[ld]_{\pi} \ar[rd]^\pi & \\
    (K,\nu) \ar[rr]^{\phi} & & (K',\nu')
   }
  \end{equation*}
  Thus, for every unital $G$-invariant $\uC^*$-subalgebra 
  $\mathcal{A} \subset \uL^\infty(X,\mu)$ there is (up to isomorphy) a unique 
  standard probability factor associated to it.
  
  In the special case that $\mathcal{A} \subset \uL^\infty(X, \mu)$ is a unital 
  $G$-invariant $\uC^*$-subalgebra that is $\|\cdot\|_{\uL^1}$-dense in $\uL^1(X,\mu)$,
  the isomorphism $\Phi\colon \mathcal{A} \to \uC(K)$ above extends to an isometric isomorphism
  $\Phi\colon \uL^1(X, \mu) \to \uL^1(K,\nu)$ and in this case we call the action 
  $G \curvearrowright(K, \nu)$ a \textbf{topological model} for $G\curvearrowright (X, \mu)$.\footnote{Note that $\mathcal{A}$ need \emph{not} be separable, i.e., $K$ need not be metrizable.
  We will mostly work with metrizable models but shall need this flexibility occasionally.}
  As in the discussion of factors, one can also associate an (up to isomorphy) unique 
  \textbf{standard probability model} $(K',\nu')$ to $\mathcal{A}$. Since $(X,\mu)$ is a standard probability
  space, the isomorphism $\Phi\colon \uL^1(X,\mu) \to \uL^1(K',\nu')$ then 
  induces a point isomorphism $\phi\colon(K',\nu') \to (X,\mu)$ by von Neumann's theorem.
  In what follows, we will always clarify whether we talk of the \enquote{topological model} or the \enquote{standard probability model} associated to a $\uC^*$-algebra $\mathcal{A} \subset \uL^\infty(X,\mu)$ by 
  using the corresponding term (though almost all models will be standard). If $\pi\colon \uX \to \uY$ is an extension of nonsingular $G$-actions
  and $\mathcal{A} \subset \uL^\infty(\uX)$ and $\mathcal{B} \subset \uL^\infty(\uY)$ are $\uL^1$-dense unital 
  $G$-invariant $\uC^*$-subalgebras with $T_\pi(\mathcal{B}) \subset \mathcal{A}$, $T_\pi$ induces 
  a canonical extension $q\colon (K, \mu) \to (L,\nu)$ between the respective standard/topological models and we 
  call such an extension $q\colon (K, \mu) \to (L,\nu)$ a \textbf{topological model for the extension} 
  $\pi\colon \uX\to\uY$. Finally, note that one can also talk about topological models without dynamics 
  by assuming that the action of $G$ is trivial in the above discussion.
\end{construction}

\begin{lemma}\label{lem:uniformcont}
  Let $G$ be a locally compact second countable group  and $G \curvearrowright \uX$ a nonsingular 
  action on a probability space. Then 
  \begin{equation*}
    \mathcal{M} \defeq \left\{ f\in \uL^\infty(\uX) \mmid g \mapsto T_{g^{-1}}f \text{ is } \|\cdot\|_{\uL^\infty(\uX)}\text{-continuous}\right\}
  \end{equation*}
  is a $\uC^*$-subalgebra of $\uL^\infty(\uX)$ that is dense in $\uL^1(\uX)$.
\end{lemma}
\begin{proof}
  It is elementary to check that $\mathcal{M}$ is a $\uC^*$-subalgebra; the nontrivial
  part is denseness in $\uL^1(\uX)$.
  To prove this, we will borrow ideas from \cite[Theorem 3.3]{EGK2018} where this was 
  proven for $G = \R$ and measure-preserving actions. Let $(U_n)_n$ be a decreasing neighborhood
  base of $e\in G$ and $(\phi_n)_n$ be a sequence 
  in $\uC_\uc(G)$ with $\phi_n \geq 0$, $\int_G \phi_n \,\ud\um_G = 1$, and $\phi_n \subset U_n$ for every $n\in \N$.
  For notational simplicity, we exceptionally fix a right- instead of left-invariant Haar measure $\um_G$ (the otherwise
  appearing modular function is immaterial to the argument).
  Then for $f\in \uL^\infty(\uX)$ and any sequence $(h_n)_n$ in $G$ that converges to $e\in G$,
  \begin{align*}
     &\left|T_{h_n^{-1}}\left(\int_{U_n} T_{g^{-1}}f \phi_n(g)\,\ud\um_G(g)\right) - \int_{U_n} T_{g^{-1}}f \phi_n(g)\,\ud\um_G(g)\right| \\ 
     &\quad\leq \int_G |T_{g^{-1}}f|\cdot \big|\phi_n\big(gh_n^{-1}\big) - \phi_n(g)\big|\,\ud\um_G(g) \\
     &\quad\leq \|f\|_{\uL^\infty(\uX)}\big\|R_{h_n^{-1}}\phi_n - \phi_n\big\|_{\uL^1(G,\um_G)} \to 0.
  \end{align*}
  This shows that $\int_{U_n} T_{g^{-1}}f \phi_n(g)\,\ud\um_G(g) \in \mathcal{M}$ for every $n\in\N$. 
  Moreover,
  \begin{equation*}
    \left\|f - \int_{U_n} T_{g^{-1}}f \phi_n(g)\,\ud\um_G(g) \right\|_{\uL^1(\uX)}
    \leq \int_{U_n} \|f-T_{g^{-1}}f\|_{\uL^1(\uX)} \phi_n(g)\,\ud\um_G(g) \to 0
  \end{equation*}
  where we used the continuity ensured by \cref{thm:FEmagic} and the dominated 
  convergence theorem. Thus, $\mathcal{M}$ is dense in $\uL^1(\uX)$.
\end{proof}

\begin{lemma}\label{lem:flowcont}
  Let $G$ be a locally compact group, $K$ a compact metric space, and $G \curvearrowright K$ an
  action by continuous maps. Then this is a continuous action of $G$ if and only if the 
  induced action on $\uC(K)$ is strongly continuous.
\end{lemma}
\begin{proof}
  If the action of $G$ is continuous, the strong continuity is easy to verify. Conversely,
  if the action of $G$ on $\uC(K)$ is strongly continuous, then picking appropriate Urysohn
  functions, the continuity of the action is straight-forward to verify.
\end{proof}

\section{The Furstenberg--Zimmer structure theorem for relatively measure-preserving extensions} \label{sec:nonsingularFZ}

As it turns out, the Furstenberg--Zimmer structure theorem in its \enquote{relative} 
formulation \cref{thm:classicalFZ} for extensions does not require the involved systems to be 
measure-preserving. In this section, we show that virtually the same arguments apply if 
the extension is a \emph{relatively measure-preserving} extension of nonsingular systems.
Such extensions naturally occur, for example, in the structure theorem for stationary actions 
\cite[Theorem 4.3]{FurstenbergGlasner2010} which asserts that every stationary action 
$(G,m) \curvearrowright (X,\mu)$ is, modulo an $m$-proximal extension, a measure-preserving 
extension of an $m$-proximal system. (See the next section for definitions of these terms.)
Another example comes from the fact that every stationary system is a relatively measure-preserving
extension of its \emph{Radon--Nikodym factor}, the smallest factor with the same Furstenberg 
entropy (see \cite[Section 1.2]{NevoZimmer2000}).\footnote{Note, however, that in both cases,
the authors have stronger restrictions on their measured groups $(G,m)$ than we do.}
This \enquote{relative} version of the Furstenberg--Zimmer theorem is of independent interest
but will also be instrumental in the proof of the stationary structure theorem, see the 
proof of \cref{thm:relstatdichotomy}.

\begin{definition}
  Let $\pi\colon (X, \mu) \to (Y, \nu)$ be an extension of nonsingular $G$-actions. 
  \begin{enumerate}[(i)]
   \item The 
    extension $\pi$ is called a
    \textbf{relatively measure-preserving extension} if the disintegration $(\mu_y)_{y\in Y}$
    of $\mu$ w.r.t.\ $\pi$ is $G$-equivariant, i.e., for every $g\in G$ one has $g_*\mu_{y} = \mu_{gy}$ for $\nu$-a.e.\ $y\in Y$.
   \item The extension $\pi$ is called \textbf{weakly mixing} if
   for every ergodic relatively measure-preserving extension $\rho\colon(Z, \zeta) \to (Y, \nu)$ the relatively independent 
   joining $(X\times_Y Z, \mu \otimes_Y \zeta) \to (Y,\nu)$ is an ergodic extension. The system 
   $(X,\mu)$ is called \textbf{weakly mixing} if the extension $(X,\mu) \to \mathrm{pt}$ is weakly mixing,
   i.e., if for every ergodic measure-preserving
    action $G \curvearrowright \uY$ the diagonal action $G \curvearrowright \uX\times\uY$
    is ergodic.
  \end{enumerate}
\end{definition}

\begin{remark}
  The definition of weakly mixing nonsingular actions (and extensions) might seem unexpected 
  because it appears asymmetric in that the systems $G\curvearrowright\uY$ are required to be measure-preserving.
  A more natural first attempt at a definition of weakly mixing nonsingular actions
  might be to require the ergodicity of the product action 
  $G \curvearrowright \uX\times\uX$, 
  called \textbf{double ergodicity} in \cite{GlasnerWeiss2016}. It is shown in \cite[Theorem 1.1]{GlasnerWeiss2016}
  that double ergodicity implies weak mixing but is generally 
  strictly stronger; see \cite[Proposition 6.1]{GlasnerWeiss2016} for an example
  that is weakly mixing but not doubly ergodic.
  
  In fact, \cite{GlasnerWeiss2016} gives an entire list of possible mixing properties
  but settles on the definition given above for weak mixing. The rationale behind this is that 
  the definition of weak mixing should still satisfy that weakly mixing systems are \enquote{orthogonal}
  to isometric systems. As we will see below, the Kronecker factor of a stationary
  action is \emph{always} measure-preserving. Thus, requiring the nonexistence of correlations
  with measure-preserving systems will characterize that the Kronecker factor 
  of a stationary action vanishes, whereas requiring the nonexistence of 
  correlations with the larger class of nonsingular actions (in particular, double ergodicity) 
  would be too strong to yield a characterization. Moreover, when working in the category of stationary 
  actions, an additional reason for 
  requiring measure-preservation in the multiplier property defining weak mixing is to 
  stay within this category: while the product of a stationary and an invariant measure 
  is always stationary, the product of two stationary measures is usually not (see \cite[Section 3]{FurstenbergGlasner2010} 
  for how to adapt the \enquote{product} to the stationary category in terms of joinings).
  Finally, note that the definition of weak mixing for nonsingular actions reduces to 
  the usual notion of weak mixing if the action is measure-preserving and that every 
  weakly mixing extension of nonsingular actions is necessarily ergodic.
\end{remark}

\begin{remark}\label{rem:mRIM}
  Recall that an extension of nonsingular actions is, per definition, a 
  measure-preserving map. However, it need not be a
  relatively measure-preserving extension. As a simple counterexample, let $\lambda$ denote Lebesgue 
  measure on $[0,1]$ and take 
  $(Y,\nu; \psi) \defeq ([0,1], \lambda; x\mapsto x)$, 
  $(X, \mu; \psi) \defeq ([0,1]^2, \lambda^2; (x,y) \mapsto (x, y^2))$, and 
  $\pi\colon X \to Y$, $(x,y) \mapsto x$.
  However, if 
  $\pi\colon (X,\mu) \to (Y,\nu)$ is a relatively measure-preserving extension from 
  a measure-preserving to a nonsingular $G$-action, then the action 
  $G \curvearrowright (X,\mu)$ must also be measure-preserving.
  
  The terminology for extensions with equivariant measure disintegrations is not 
  consistent in the literature. Sometimes a relatively measure-preserving extensions
  (despite being a measure-preserving map) is called a \enquote{measure-preserving factor map/extension}, 
  see \cite{FurstenbergGlasner2010}. We will not use the terminology of 
  \enquote{measure-preserving factor maps}
  to avoid this ambiguity. Rather, our definition is along the lines of 
  \cite[Definition 1.7]{NevoZimmer2000}
  who say that $(X,\mu)$ has \enquote{relatively $G$-invariant measure over $(Y,\nu)$}.

  Finally, observe that if $\pi\colon \uX \to \uZ$ is a relatively measure-preserving 
  extension of nonsingular $G$-actions, then the conditional $\uL^2$-space $\uL^2(\uX|\uZ)$ is 
  $G$-invariant.
\end{remark}

The following are useful characterizations of ergodic and relatively measure-preserving 
extensions that we shall make repeated use of.

\begin{lemma}\label{lem:rmp_char}
  Let $\pi\colon \uX \to \uZ$ be an extension of nonsingular $G$-actions.
  \begin{enumerate}
    \item The 
    extension is relatively measure-preserving if and only if $\uL^2(\uX|\uZ)$ is $G$-invariant 
    and the conditional
    expectation $\E_\uZ \colon \uL^2(\uX|\uZ) \to \uL^\infty(\uZ)$ is $G$-equivariant.
    \item The extension is ergodic if and only if every $G$-invariant 
    function $f \in \uL^2(\uX| \uZ)$ satisfies $f = \E_\uZ(f)\cdot \1_\uX$.
  \end{enumerate}
\end{lemma}
\begin{proof}
  We may assume for simplicity that $\uX = (X,\mu)$ and $\uZ = (Z,\zeta)$ are compact metric
  Borel probability spaces and that $\pi$ is continuous. First, let 
  $\pi\colon \uX \to \uZ$ be relatively measure-preserving. Then
  for $\zeta$-a.e.\ $z\in Z$ and every $f\in \uC(X)$ 
  \begin{align*}
    \E_\uZ(T_gf)(z) 
    = \int_{X_z} f(gx) \dmu_z(x)
    = \int_{X_{gz}} f(x) \dmu_{gz}(x)
    = (T_g\E_{\uZ}f)(z).
  \end{align*}
  Thus, $\E_\uZ T_g f = T_g\E_\uZ f$ for all $f\in \uC(X)$ and thus $\E_\uZ$ is $G$-equivariant. 
  Conversely, suppose $\E_\uZ$ is $G$-equivariant and let $\mathcal{F} \subset \uC(X)$ be a 
  countable dense subset. Then for every $f\in \uC(X)$ and $\zeta$-almost every $z\in Z$
  \begin{equation*}
   \int_{X_{gz}} f(x) \,\mathrm{d}g_*\mu_z(x)
   = \int_{X_z} f(gx) \dmu_z(x)
   = \E_\uZ(T_gf)(z) 
   = (T_g\E_{\uZ}f)(z)
   = \int_{X_{gz}} f(x) \dmu_{gz}(x).
  \end{equation*}
  Since $\mathcal{F}$ is dense in $\uC(X)$, its restriction to $X_z$ is also dense in $\uC(X_z)$
  for every $z\in Z$. Thus, for every $g\in G$ and $\zeta$-almost every $z\in Z$, $g_*\mu_z = \mu_{gz}$. 
  Therefore, the extension is relatively measure-preserving.
  
  Now, if every $G$-invariant function $f\in \uL^2(\uX|\uZ)$ satisfies $f = \E_\uZ(f)\cdot \1_\uX$,
  plugging in characteristic functions shows that the extension $\pi$ must be ergodic.
  Conversely, suppose the extension $\pi$ is ergodic and $f\in \uL^2(\uX|\uZ)$ is $G$-invariant.
  If $f = \1_A$ is a characteristic function, the conclusion $f = \E_\uZ(f)\cdot \1_\uX$ follows 
  swiftly from the definition of ergodic extensions. The general case follows since level sets 
  of a $G$-invariant function are $G$-invariant which allows to approximate $f$ via linear combinations
  of $G$-invariant
  characteristic functions. This concludes the proof.
\end{proof}

\begin{remark}
  If a relatively measure-preserving extension $\pi\colon \uX \to \uZ$ is not ergodic, there thus exists a nonzero $G$-invariant 
  function $f\in \uL^2(\uX|\uZ)$ such that $f \neq \E_\uZ(f) \cdot \1_\uX$. By a 
  cut-off
  argument and replacing $f$ by $f - \E_\uZ(f) \cdot \1_\uX$, one 
  may also arrange that 
  $f\in \uL^\infty(\uX)$ and $\E_\uZ(f) = 0$. (If the extension is not relatively measure-preserving,
  it is not guaranteed that $\E_\uZ(f)$ is $G$-invariant.)
  We will occasionally make 
  repeated use of this characterization 
  of nonergodic extensions.
\end{remark}

\begin{theorem}[Structure Theorem]\label{thm:nonsingularmainthm}
  Let $\pi\colon \uX \to \uZ$ be a relatively measure-preserving extension of nonsingular 
  $G$-actions. Then there are 
  a weakly mixing extension $\alpha\colon \uX \to \uY$ and a distal extension $\beta\colon \uY \to \uZ$ 
  such that the diagram
  \begin{equation*}
   \xymatrix{
    \uX \ar[rd]_-\alpha \ar[rr]^-\pi && \uZ \\
    & \uY \ar[ru]_-\beta &
   }
  \end{equation*}
  commutes.
\end{theorem}

Similarly to the measure-preserving case, the proof reduces to the following generalization of \cref{thm:krodichotomy} 
to extensions.

\begin{theorem}[Relative Kronecker dichotomy]\label{thm:nonsingularrelkrodichotomy}
  A relatively measure-preserving extension $\pi \colon \uX \to \uZ$ of 
  nonsingular $G$-actions is weakly mixing if and only if there is no 
  isometric intermediate extension of $\uZ$.
\end{theorem}

The proof follows the line of reasoning used to prove \cref{thm:krodichotomy} and employs the ingredients 
listed in \cref{rem:ingredients}: the natural correspondence between Hilbert--Schmidt homomorphisms $K\colon \uL^2(\uX|\uZ) \to \uL^2(\uY|\uZ)$
and their kernels (under which $G$-equivariant homomorphisms correspond to $G$-invariant kernels) and a spectral
theorem for Hilbert--Schmidt homomorphisms that allows for approximation by means of projections onto finitely generated
submodules. We collect both ingredients below to prepare for the proof of \cref{thm:nonsingularrelkrodichotomy}.

Since the definition of Hilbert--Schmidt
homomorphisms between $\uL^2(\uX|\uZ)$ and $\uL^2(\uY|\uZ)$ requires some additional terminology, we 
defer the details to the appendix from which we shall use the following result. For the time 
being, the reader may take it as a definition of Hilbert--Schmidt homomorphisms.

\begin{theorem}\label{thm:KHiso}
  Let $\pi\colon \uX \to \uZ$ and $\rho\colon \uY \to \uZ$ be measure-preserving maps between 
  standard probability spaces. Then the assignment
  \begin{equation*}
    I\colon \uL^2(\uX\times_\uZ \uY|\uZ) \to \HS(\uL^2(\uX|\uZ), \uL^2(\uY|\uZ)), \quad 
    (I_kf)(y) \defeq \int_{X_{\rho(y)}\times X_{\rho(y)}} k(x,y)f(x) \dmu_{\rho(y)} 
  \end{equation*}
  defines an isometric isomorphism between Hilbert--Schmidt homomorphisms and their kernels.
\end{theorem}

A quick computation proves the essential observation that, if the maps $\pi$ and $\rho$ above are relatively 
measure-presreving extensions of $G$-actions, then the Hilbert--Schmidt operator $I_k$ intertwines the 
$G$-actions if and only if $k$ is $G$-invariant for the diagonal $G$-action on $\uX\times_\uZ\uY$.

The following lemma provides a different way to construct certain Hilbert--Schmidt homomorphisms which we 
shall need below.

\begin{lemma}\label{lem:moduleprojection}
  Let $\pi\colon \uX \to \uZ$ be a measure-preserving map between standard probability spaces
  and $M \subset \uL^2(\uX|\uZ)$ an $\uL^\infty(\uZ)$-submodule. Denote by
  \begin{equation*}
    \tilde{M} \defeq \operatorname{cl}_{\|\cdot\|_{\uL^2}}(M) \cap \uL^2(\uX|\uZ)
  \end{equation*}
  the $\uL^2$-closure of $M$ within $\uL^2(\uX|\uZ)$. Then the following assertions are true.
  \begin{enumerate}[(i)]
   \item $M = \tilde{M}$ if and only if $\uL^2(\uX|\uZ) = M \oplus M^\perp$ where
   \begin{equation*}
     M^\perp = \left\{f\in \uL^2(\uX|\uZ) \mmid \forall m\in M\colon \E_\uZ(f\overline{m}) = 0 \right\}.
   \end{equation*}
   \item $M$ is a finitely-generated $\uL^\infty(\uZ)$-submodule if and only if $\tilde{M}$ is. 
   In this case, $\tilde{M}$ admits a finite suborthonormal basis.\footnote{See the appendix 
   for a discussion of this notion.}
   \item If $M$ is a finitely-generated $\uL^\infty(\uZ)$-submodule, the orthogonal projection 
   $P \colon \allowbreak \uL^2(\uX|\uZ) \allowbreak \to  \tilde{M}$ induced by the decomposition $\uL^2(\uX|\uZ) = \tilde{M} \oplus \tilde{M}^\perp$ 
   is a Hilbert--Schmidt homomorphism.
  \end{enumerate}
\end{lemma}
\begin{proof}
  These follow from \cite[Proposition 2.12, Proposition 2.18, Lemma 7.5(iii)]{EHK2021}
  respectively.
\end{proof}

The next fact we require from the appendix is the following: Given a bounded module homomorphism 
$T \colon \uL^2(\uX|\uZ) \to \uL^2(\uY|\uZ)$, i.e., a bounded operator with 
$T(fg) = fT(g)$ for all $f\in \uL^\infty(\uZ)$ and $g\in \uL^2(\uX|\uZ)$, 
there is a unique module adjoint $T^*$,
i.e., an operator $T^*\colon \uL^2(\uY|\uZ) \to \uL^2(\uX|\uZ)$ with the property that 
\begin{equation*}
  \E_\uZ(T(f)\cdot\overline{g}) = \E_\uZ(f\cdot \overline{T^*(g)}) \qquad \forall f\in \uL^2(\uX|\uZ), g\in \uL^2(\uY|\uZ).
\end{equation*}
As in the Hilbert space case, a module homomorphism $T\colon \uL^2(\uX|\uZ) \to \uL^2(\uX|\uZ)$
is called \textbf{self-adjoint} if $T = T^*$.
For a Hilbert--Schmidt homomorphism $I_k\in \HS(\uL^2(\uX|\uZ), \uL^2(\uY|\uZ))$ 
with $k \in \uL^2(\uX\times_\uZ \uY|\uZ)$ as above, $(I_k)^* = I_{k^*}$ where 
$k^* \in \uL^2(\uY\times_\uZ \uX|\uZ)$ is defined by $k^*(y,x) = \overline{k(x,y)}$ for 
a.e.\ $(y,x) \in \uY\times_\uZ \uX$.
It is straight-forward to verify that the usual properties of Hilbert space adjoints are also 
satisfied in this setting.
In particular, $T \neq 0$ if and only if $T^*T \neq 0$.

The last fact we require about Hilbert--Schmidt homomorphisms is the following spectral theorem.

\begin{proposition}[{\cite[Theorem 4.1, Proposition 6.5]{EHK2021}}] \label{prop:KHspectralthm}
  Let $\pi\colon \uX \to \uZ$ be a relatively measure-preserving extension 
  of nonsingular $G$-actions and let $A\colon \uL^2(\uX|\uZ) \to \uL^2(\uX|\uZ)$ 
  be a self-adjoint $G$-equivariant Hilbert--Schmidt homomorphism. Then $A$ can be written as 
  the order-convergent sum
  \begin{equation*}
    A = \sum_{j\in\Z} \lambda_j P_{M_j}
  \end{equation*}
  for $G$-invariant functions $\lambda_j\in \uL^\infty(\uZ)$ and $G$-equivariant 
  Hilbert--Schmidt projections $P_{M_j}\colon \uL^2(\uX|\uZ) \to \uL^2(\uX|\uZ)$
  onto finitely generated, pairwise orthogonal submodules $M_j$.
\end{proposition}

The key take-away from this spectral theorem is that if the operator $A$ 
is nontrivial, there must be some $G$-invariant finitely generated $\uL^\infty(\uZ)$-submodule.
(The notion of order-convergence is discussed in the appendix but is immaterial here.)
Next, an observation that will facilitate the proof of \cref{thm:nonsingularrelkrodichotomy} is
that the characterization of weakly mixing extensions for nonsingular systems simplifies
if the extension is relatively measure-preserving.

\begin{lemma}\label{lem:weakmixchar}
  Let $\pi\colon \uX \to \uZ$ be a relatively measure-preserving extension of nonsingular
  $G$-actions. Then $\pi$ is weakly mixing if and only if the relatively independent 
  joining $\pi\times_\uZ\pi\colon \uX \times_\uZ \uX \to \uZ$ is an ergodic extension.
\end{lemma}
\begin{proof}
  It is clear that if $\pi$ is weakly mixing, 
  $\pi\times_\uZ\pi\colon \uX \times_\uZ \uX \to \uZ$ 
  must be ergodic since weak mixing implies that 
  $\pi\colon \uX \to \uZ$ is an ergodic and relatively measure-preserving
  extension. Conversely, assume 
  $\pi\times_\uZ\pi\colon \uX \times_\uZ \uX \to \uZ$ is an ergodic
  extension; in particular, the extension $\uX \to \uZ$ is ergodic. 
  Suppose $\pi$ is not weakly mixing, then we can find an ergodic and 
  relatively measure-preserving extension $\rho\colon \uY \to \uZ$ 
  of nonsingular $G$-actions such 
  that the extension $\pi\times_\uZ\rho\colon \uX\times_\uZ\uY \to \uZ$ 
  is not ergodic. That is,
  there exists a nonzero $G$-invariant 
  $k\in \uL^\infty(\uX\times_\uZ\uY)$ such that $\E_\uZ(k) = 0$. Since 
  the extension $\uX\times_\uZ\uY \to \uZ$ is relatively measure-preserving,
  $\E_\uZ(k)$ is $G$-invariant. Since the extension $\uY \to \uZ$ is ergodic,
  we can conclude that $\E_\uY(k) = \E_\uZ(k)\cdot \1_\uY = 0$.
  
  Let $K\colon \uL^2(\uX|\uZ) \to \uL^2(\uY|\uZ)$ be the 
  Hilbert--Schmidt homomorphism $I_k$ 
  associated to $k$ via the isomorphism $\uL^2(\uX\times_\uZ\uY|\uZ) \cong \HS(\uL^2(\uX|\uZ), \uL^2(\uY|\uZ))$.
  Note that $K\1_\uX = \E_\uY(k) = 0$ and hence $K^*K\1_\uX = 0$. 
  We claim that $K^*K \colon \uL^2(\uX|\uZ) \to \uL^2(\uX|\uZ)$ is 
  $G$-equivariant. To see this, we use, in this order, relative measure-preservation of $\pi$,
  $G$-invariance of $k$, and relative measure-preservation of $\rho$: For $f\in \uL^2(\uX|\uZ)$ 
  and $g\in G$
  \begin{align*}
    (K^*KT_gf)(x) 
    &= \int_{X_{\pi(x)}} f(gx') \int_{Y_{\pi(x)}} k(x,y)\overline{k(x',y)} \dnu_{\pi(x)}(y)\dmu_{\pi(x)}(x') \\
    &= \int_{X_{\pi(gx)}} f(x') \int_{Y_{\pi(x)}} k(x,y)\overline{k(g^{-1}x',y)} \dnu_{\pi(x)}(y)\dmu_{\pi(gx)}(x') \\
    &= \int_{X_{\pi(gx)}} f(x') \int_{Y_{\pi(x)}} k(gx,gy)\overline{k(x',gy)} \dnu_{\pi(x)}(y)\dmu_{\pi(gx)}(x') \\
    &= \int_{X_{\pi(gx)}} f(x') \int_{Y_{\pi(gx)}} k(gx,y)\overline{k(x',y)} \dnu_{\pi(gx)}(y)\dmu_{\pi(gx)}(x') \\
    &= (T_gK^*Kf)(x).
  \end{align*}
  Since $K^*K$ is $G$-intertwining, its kernel $k^* * k$ is $G$-invariant and by ergodicity of 
  $\pi\times_\uZ\pi$, it follows that
  $k^* * k = \E_\uZ(k^* * k) \1_{\uX\times_\uZ\uX}$. Thus, $K^*K = \E_\uZ(k^* * k)\cdot\E_\uZ$. However, 
  \begin{equation*}
    0 = K^*K\1_\uX = \E_\uZ(k^* * k)\cdot\E_\uZ(\1_{\uX}) = \E_\uZ(k^* * k).
  \end{equation*}
  Therefore, $K^*K = 0$, whence $K = 0$. This is a contradiction to the assumption that $k$ is nonzero. Thus,
  our choice of $k$ was impossible and $\pi$ must indeed be weakly mixing.
\end{proof}

The last observation we require is that, given finitely generated submodules, 
we can construct a corresponding isometric intermediate extensions by 
means of topological models. Moreover, it will be essential to know that, if 
the extension is relatively measure-preserving,
it suffices to find the submodule in $\uL^2(\uX|\uZ)$ instead of $\uL^\infty(\uX)$.
For the reader's convenience, we reproduce below the arguments from 
\cite[Lemma 8.3, Proposition 8.5]{EHK2021} where these two claims are proved in the measure-preserving case.

\begin{proposition}\label{prop:isoextmodel}
  Let $\pi\colon \uX \to \uZ$ be an extension of 
  nonsingular $G$-actions. Then
  \begin{equation*}
    \mathcal{A} \defeq \overline{\bigcup\left\{ M \subset \uL^\infty(\uX) \mmid 
    \begin{matrix}
      M \text{ is a finitely generated} \\
      G\text{-invariant } \uL^\infty(\uZ)\text{-submodule} 
    \end{matrix}
    \right\}}^{\|\cdot\|_{\uL^\infty(\uX)}} 
    \subset \uL^\infty(\uX)
  \end{equation*}
  defines a unital $\uC^*$-subalgebra that corresponds to the largest intermediate 
  extension $\uX \to \tilde{\uZ} \to \uZ$ that is an isometric extension of $\uZ$.
\end{proposition}
\begin{proof}
  It is straightforwad to check that the product of two finitely generated $G$-invariant 
  $\uL^\infty(\uZ)$-submodules of $\uL^\infty(\uZ)$ is again such a submodule, so 
  $\mathcal{A} \subset \uL^\infty(\uZ)$ is a $G$-invariant $\uC^*$-subalgebra that contains $\uL^\infty(\uZ)$.
  As such, it gives rise to an intermediate extension $\uX \to \tilde{\uZ} \to \uZ$ by means 
  of the correspondence \cref{factor-algebra-corr} between subalgebras and factors.
  It also follows from this correspondence and the definition of $\mathcal{A}$ that this factor
  is the maximal isometric intermediate extension.
\end{proof}

\begin{lemma}\label{lem:moduleseverywhere}
  Let $\pi\colon \uX \to \uZ$ be a relatively measure-preserving extension of 
  nonsingular $G$-actions. Then the following subspaces coincide:
  \begin{align*}
    E &= \overline{\bigcup\left\{ M \subset \uL^2(\uX|\uZ) \mmid 
    \begin{matrix}
      M \text{ is a finitely generated} \\
      G\text{-invariant } \uL^\infty(\uZ)\text{-submodule} 
    \end{matrix}
    \right\}}^{\|\cdot\|_{\uL^2(\uX)}}  \\
    F &= \overline{\bigcup\left\{ M \subset \uL^\infty(\uX) \mmid 
    \begin{matrix}
      M \text{ is a finitely generated} \\
      G\text{-invariant } \uL^\infty(\uZ)\text{-submodule} 
    \end{matrix}
    \right\}}^{\|\cdot\|_{\uL^2(\uX)}} 
  \end{align*}
\end{lemma}
\begin{proof}
  We only need to prove that $F \subset E$. So let $M \subset \uL^2(\uX|\uZ)$ be a finitely generated
  $G$-invariant $\uL^\infty(\uZ)$-submodule. Without loss of generality, we may assume that 
  $M = \tilde{M}$ (see \cref{lem:moduleprojection} for this notation), i.e., $M$ equals its $\uL^2$-closure within $\uL^2(\uX|\uZ)$. Pick a $\uZ$-suborthonormal
  basis $\mathcal{B} = \{e_1, \dots, e_d\}$ of $M$ as provided by \cref{lem:moduleprojection}.
  
  We shall construct $G$-invariant functions $\eta_n$ with $0\leq \eta_n \leq 1$
  such that $\mathcal{B}_n \defeq \eta_n \mathcal{B} \subset \uL^\infty(\uX)$ for all $n\in\N$
  and $\eta_n e_j \to e_j$ in $\uL^2(\uX)$ for $j = 1,\dots, d$. Since $\eta_n$ is $G$-invariant, 
  $\mathcal{B}_n$ generates a $G$-invariant submodule and thus this will show that $F \subset E$.
  To find the functions $\eta_n$, let 
  \begin{equation*}
    e \defeq \sum_{j=1}^d e_j\otimes_\uZ \overline{e_j} \in \uL^2(\uX\times_\uZ \uX | \uZ)
  \end{equation*}
  and 
  \begin{equation*}
    m \defeq \sum_{j=1}^d |e_j|^2 \in \uL^2(\uX|\uZ).
  \end{equation*}
  Under the isomorphism $\uL^2(\uX\times_\uZ\uX|\uZ) \cong \HS(\uL^2(\uX|\uZ))$, $e$ is the kernel of the 
  orthogonal projection $P_M$ onto $M$. Since the extension $\uX \to \uZ$ is relatively measure-preserving,
  $G$-invariance of $M$, $G$-equivariance of $P_M$, and $G$-invariance of $e$ are all equivalent. Thus, 
  $e$ is $G$-invariant and so is 
  \begin{equation*}
    |e|^2 = \sum_{j,k = 1}^d e_j \overline{e_k} \otimes \overline{e_j}e_k.
  \end{equation*}
  We claim that also $m$ is $G$-invariant. To see this, 
  take the conditional expectation of $|e|^2$ onto the second factor:
  \begin{align*}
    \E_\uX |e|^2 
    = \sum_{j,k=1}^d  \E_\uZ(e_j\overline{e_k}) \overline{e_j}e_k 
    = \sum_{j,k=1}^d |e_j|_\uZ^2 |e_j|^2 
    = \sum_{j,k=1}^d \left||e_j|_\uZ e_j\right|^2 
    = \sum_{j,k=1}^d |e_j|^2
    = m.
  \end{align*}
  Thus, $m$ is $G$-invariant. Now, the functions 
  \begin{equation*}
    \eta_n \defeq \frac{\sqrt{m}\wedge n}{\sqrt{m} + \frac{1}{n}}
  \end{equation*}
  are $G$-invariant, satify $0\leq \eta_n \leq 1$, $\eta_ne_j \to e_j$ and $\eta_n e_j \in \uL^\infty(\uX)$ 
  for $j=1,\dots, d$ and $n\in\N$. Thus, $F \subset E$.
\end{proof}

\begin{proof}[Proof of \cref{thm:nonsingularrelkrodichotomy}]
   First, suppose 
  \begin{equation*}
   \xymatrix{
    (X, \mu) \ar[r]^{p} & (\tilde{Z}, \tilde{\zeta}) \ar[r]^q & (Z, \zeta)
   }
  \end{equation*}
  is a nontrivial
  isometric intermediate extension of $(Z, \zeta)$. Then there exists a finitely-generated 
  $\uL^\infty(\uZ)$-submodule $M \subset \uL^2(\tilde{\uZ}|\uZ)$ such that $M$ is not contained 
  in $\uL^\infty(\uZ)$. Without loss of generality, assume that $M = \tilde{M}$ and 
  let $P_M\colon \uL^2(\tilde{\uZ}|\uZ) \to M$ be the orthogonal projection 
  onto $M$. By \cref{lem:moduleprojection}, $P_M\in \HS(\uL^2(\tilde{\uZ}|\uZ))$. Since relative 
  measure-preservation is equivalent to equivariance of conditional expectations (see \cref{lem:rmp_char}),
  the decomposition $\uL^2(\tilde{\uZ}|\uZ) = M \oplus M^\perp$ is readily verified to be $G$-invariant. Thus, 
  $P_M$ is $G$-equivariant. Via the isomorphism $\HS(\uL^2(\tilde{\uZ}|\uZ)) \cong \uL^2(\tilde{\uZ}\times_\uZ\tilde{\uZ}|\uZ)$ 
  from \cref{thm:KHiso}, we may write $P_M = I_k$ for a uniquely determined $k \in \uL^2(\tilde{\uZ}\times_\uZ\tilde{\uZ}|\uZ)$.
  As observed above, the $G$-equivariance of $P_M$ is equivalent to invariance of $k$ since the extension 
  is relatively measure-preserving. Since $P_M$ does not map
  into $\uL^\infty(\uZ)$, it cannot be of the form $f \cdot \E_\uZ$ for some $f\in \uL^\infty(\uZ)$ 
  and thus $k\neq \E_\uZ(k)\cdot \1_{\tilde{\uZ}\times_\uZ\tilde{\uZ}}$ since otherwise $P_M = \E_\uZ(k)\cdot\E_\uZ$.
  This shows that the extension 
  $\tilde{\uZ} \times_\uZ \tilde{\uZ} \to \uZ$ is not ergodic. By extension, the 
  composite extension 
  \begin{equation*}
   (X\times_Z X, \mu\otimes_Z \mu) \to (\tilde{Z}\times_Z\tilde{Z}, \tilde{\zeta}\otimes_Z\tilde{\zeta}) \to (Z,\zeta)
  \end{equation*}
  is not ergodic either, which proves that the extension $\pi\colon\uX \to \uZ$ is not weakly mixing.
  
  Conversely, suppose $\pi\colon \uX \to \uZ$ is not weakly mixing. 
  Then by \cref{lem:weakmixchar}, the relatively independent 
  joining $\pi\times_\uZ\pi\colon \uX \times_\uZ \uX \to \uZ$ is not ergodic.
  We distinguish two cases: If the extension $\pi\colon\uX \to \uZ$ is not ergodic, 
  then the $\uC^*$-subalgebra
  \begin{equation*}
    \mathcal{A}\defeq \uC^*\left(T_\pi(\uL^\infty(\uZ)), \uL^\infty(\uX)^G\right) \subset \uL^\infty(\uX) 
  \end{equation*}
  is readily verified to lead to a nontrivial isometric extension which proves
  the claim. So suppose now that the extension $\pi\colon \uX \to \uZ$ is ergodic
  but the extension $\uX\times_\uZ\uX \to \uZ$ is not.
  Then, there exists a $G$-invariant 
  $k \in \uL^\infty(\uX\times_\uZ\uX)$ such that $k$ is nonzero and $\E_\uZ(k) = 0$.
  As in the proof of \cref{lem:weakmixchar}, the ergodicity and relative 
  measure-preservation of the extension 
  $\uX \to \uZ$ imply that also $\E_\uX(k) = 0$.
  
  Let $K \colon \uL^2(\uX|\uZ) \to \uL^2(\uX|\uZ)$
  be the Hilbert--Schmidt homomorphism corresponding to $k$. 
  Since $\pi\colon \uX \to \uZ$ is relatively measure-preserving, 
  $G$-invariance of $k$ translates into $G$-equivariance of $K$. 
  Moreover, the conditions $k \neq 0$ and $K\1_\uX = \E_\uX(k) = 0$ 
  entail that $K \neq 0$ and 
  \begin{equation*}
    K\1_\uX = \E_\uZ(k)\cdot\1_\uX  = 0.
  \end{equation*}
  Therefore, $K$ cannot be an $\uL^\infty(\uZ)$-multiple of $\E_\uZ$ since otherwise $K = 0$,
  a contradiction.
  By \cref{prop:KHspectralthm}, we can write $K$ as an order-convergent series
  \begin{equation*}
    K = \sum_{j\in\Z} \lambda_j P_{M_j}
  \end{equation*}
  for $G$-invariant functions $\lambda_j\in \uL^\infty(\uZ)$ and $G$-equivariant 
  Hilbert--Schmidt projections $P_{M_j}\colon \uL^2(\uX|\uZ) \to \uL^2(\uX|\uZ)$
  onto finitely generated, pairwise orthogonal submodules $M_j$.
  
  Assume for the sake of contradiction that all $M_j$ lie in $\uL^\infty(\uZ)$.
  We know that $K \neq 0$ and $K\1_X = 0$. By orthogonality of the submodules $M_j$,
  the sets $\lambda_j M_j \subset M_j$ are pairwise orthogonal for $j\in\Z$ and thus
  we can conclude that $\lambda_j P_{M_j} \1_X = 0$ for all $j\in \Z$. Since each 
  $P_{M_j}$ is a $\uL^\infty(\uZ)$-module homomorphism, it follows that 
  $\lambda_j P_{M_j}(\uL^\infty(\uZ)) = \{0\}$. However, if $M_j \subset \uL^\infty(\uZ)$ 
  for all $j\in\Z$, then this means 
  \begin{equation*}
    \lambda_j P_{M_k} = 0 \qquad \forall j\in\Z.
  \end{equation*}
  Thus $K = 0$, a contradiction. Therefore, there must be a $j_0\in\Z$ such that 
  $M_{j_0} \not\subseteq \uL^\infty(\uZ)$. As in the first part of the proof, 
  since the projection $P_{M_{j_0}}$ is $G$-equivariant,
  its range $M_{j_0}$ is $G$-invariant and so we have found a nontrivial finitely-generated 
  $G$-invariant submodule in $\uL^2(\uX|\uZ)$. By \cref{lem:moduleseverywhere}, we 
  thus find a nontrivial finitely-generated $G$-invariant $\uL^\infty(\uZ)$-submodule in $\uL^\infty(\uX)$
  and thus \cref{prop:isoextmodel} concludes the proof by providing the desired nontrivial
  isometric intermediate extension.
\end{proof}

\begin{proof}[Proof of \cref{thm:nonsingularmainthm}]
  We proceed by transfinite recursion. For the induction start, set $\uX_0 = \uZ$.
  Now suppose $\mu$ is an ordinal and we have constructed a projective system
  $((\uX_\eta)_{\eta < \mu}, (\pi_\eta^\sigma)_{\eta\leq \sigma < \mu})$ of successive 
  isometric extensions between factors of $\uX$.
  \begin{itemize}
   \item If $\mu$ is a limit ordinal, set $\uX_\mu = \lim_{\from} \uX_\eta$.
   If $\uX \to \uX_\mu$ is weakly mixing, we are done. Otherwise, repeat the induction step.
   \item If $\mu$ is a successor ordinal and the extension $\uX \to \uX_{\mu-1}$ is 
   weakly mixing, we are done. Otherwise, by \cref{thm:nonsingularrelkrodichotomy} there is an isometric intermediate extension 
   $\uX \to \uX_\mu \to \uX_{\mu-1}$ and we repeat the induction step.
  \end{itemize}
  This recursion terminates after countably many steps, for otherwise $\uL^2(\uX)$ 
  would contain an uncountable orthonormal subset. This proves the desired decomposition.
\end{proof}

\section{The Kronecker dichotomy for stationary actions} \label{sec:stationarydichotomy}

In this section, we prove a generalization of the Kronecker dichotomy \cref{thm:krodichotomy}
to stationary actions in \cref{thm:statdichotomy}, essentially following \cite[Section 4]{Bjorklund2017}. 
Apart from a slight change to the proof, the only difference is that we introduce the \emph{Kronecker factor}
of a nonstationary action since it will be important for the structure-theoretic later parts. Therefore, we
formulate the dichotomy in terms of the Kronecker factor, just as in \cref{thm:krodichotomy}.

\textbf{Kronecker factors for stationary actions.} We start by recalling the definitions of weak mixing, 
stationary actions, and discussing the Kronecker factor of a stationary action.
If $(G,m)$ is a measured group, a nonsingular measurable action $G \curvearrowright \uX = (X,\Sigma, \mu)$ of $G$ on a 
probability space is called \textbf{$m$-stationary}, if $m*\mu = \mu$ where
\begin{equation*}
  m*\mu\defeq \int_G g_*\mu \,\mathrm{d}m(g).
\end{equation*}
If $(G,m)$ is a measured group and $G$ acts nonsingularly on a probability space $\uX$,
we may define
\begin{equation*}
  P_m\colon \uL^\infty(\uX) \to \uL^\infty(\uX), \quad 
  f \mapsto \int_G T_g f \dm(g)
\end{equation*}
to be the \textbf{averaging Markov operator} corresponding to the nonsingular action of $(G,m)$ on $\uX$. The 
nonsingular measure 
$\mu$ is $m$-stationary if and only if $\int P_mf \dmu = \int f \dmu$ for all $f\in\uL^\infty(\uX)$, i.e., if 
$P_m$ is a bi-Markov operator. In this case, $P_m$ extends to a linear contraction on $\uL^1(\uX)$.

Since the action $G\curvearrowright \uX$ is nonsingular, for every $s\in G$ there is a 
function $0\leq h_s\in\uL^1(\uX)$ such that $s_*\mu = h_s\mu$. As in the proof of \cref{thm:FEmagic},
one can verify that the function $G \to \uL^1(\uX)$, $s \mapsto h_s$ is strongly measurable which 
will justify integrating it along $G$.
Stationarity of $\mu$ can then be 
rephrased as 
\begin{equation*}
  \int_G  h_s \dm(s) = \1_\uX.
\end{equation*}
A straight-forward computation yields that the adjoint $P_m^*\colon \uL^\infty(\uX) \to \uL^\infty(\uX)$
can be expressed as 
\begin{equation*}
  P_m^* f(x) = \int_G h_s(x) f(s^{-1}x) \dm(s).
\end{equation*}
Clearly, $P_m^*$ is positive and $P_m^*\1_\uX = \1_\uX$, so $P_m^*$ is a Markov operator on $\uL^\infty(\uX)$.
Moreover, since $P_m\1_\uX = \1_\uX$, $P_m^*$ preserves the integral and is thus a bi-Markov operator 
on $\uL^1(\uX)$. As such, it also extends to a linear contraction on $\uL^1(\uX)$ that, using a monotone
approximation argument, is given by the same expression. The reader should note the resemblance of $P_m^*$ to
comments made in \cref{rem:stationary_dis}; this is not a coincidence and will be exploited in the proof 
of \cref{lem:bundle_inv}.

\textbf{Conditional measures.} 
We recall the following basic facts about \textbf{conditional measures} for random walks, see 
\cite[Lemma II.2.1]{BougerolLacroix1985}.
These are measures obtained from the martingale convergence 
theorem by conditioning on the random walk already having followed an infinite trajectory 
$\omega = \dots g_n, g_{n-1}, \dots, g_1$ and allow to naturally decompose a stationary measure according
to the random walks trajectories.

\begin{lemma}\label{lem:condmeasures}
  Let $G$ be a lcsc group acting measurably on a second-countable locally 
  compact space $B$, let $m$ be a probability measure on $G$, and let $\mu$ be an $m$-stationary
  probability measure on $B$. If $X_1,X_2, \dots \colon (\Omega,\mathcal{F}, \P) \to G$ is a sequence 
  of i.i.d.\ $G$-valued random variables on a probability space $(\Omega,\mathcal{F},\P)$ 
  with common distribution $m$, then the following 
  assertions are true.
  \begin{enumerate}[(i)]
    \item Let $f\colon B \to \C$ be a $\mu$-integrable measurable function. Then there is a 
    set $A_f \subset \Omega$ of measure 
    $\P(A_f) = 1$ such that for every $\omega \in A_f$ the limit 
    \begin{equation*}
      (\Gamma f)(\omega) \defeq \lim_{n\to\infty} \underbrace{\int_B f \,\ud (X_1(\omega)\cdots X_n(\omega))_*\mu}_{\eqdef(\Gamma_n f)(\omega)}
    \end{equation*}
    exists and defines a function $\Gamma f \in \uL^1(\Omega, \P)$.
    \item The linear map 
    \begin{equation*}
      \Gamma \colon \uL^1(B,\mu) \to \uL^1(\Omega, \P), \quad f \mapsto \Gamma f
    \end{equation*}
    is a bi-Markov operator, i.e., $\Gamma$ is positive, $\Gamma \1_B = \1_\Omega$, and $\Gamma$ preserves the 
    integral. In particular, $\Gamma$ is an $\uL^1$ contraction.

    \item For almost all $\omega\in\Omega$
      the limit 
      \begin{equation*}
        \mu_\omega \defeq \lim_{n\to\infty} X_1(\omega)_*\cdots X_n(\omega)_*\mu
      \end{equation*}
      exists in the weak* topology, i.e., for every $f\in \uC_0(B)$
      \begin{equation*}
        \lim_{n\to\infty} \int_B f \,\ud(X_1(\omega)\cdots X_n(\omega))_*\mu 
        = \int_B f\dmu_\omega.
      \end{equation*}
    \item More generally, for every countable subset $\mathcal{F}\subset \uL^1(B,\mu)$,
    there is a set $A_\mathcal{F} \subset \Omega$ of measure $\P(A_\mathcal{F}) = 1$
    such that for all $f\in\mathcal{F}$ and $\omega\in A_\mathcal{F}$
      \begin{equation*}
        (\Gamma f)(\omega) = \lim_{n\to\infty} \int_B f \,\ud(X_1(\omega)\cdots X_n(\omega))_*\mu 
        = \int_B f\dmu_\omega.
      \end{equation*}
    \item The conditional measures are natural in the sense that if $\pi\colon (B,\mu) \to (C,\nu)$
    is a measurable, measure-preserving factor map to another (necessarily stationary) 
    action $G \curvearrowright(C, \nu)$ on a second countable locally compact space, then almost 
    surely $\pi_*\mu_\omega = \nu_\omega$.
    \item Conditional measures are almost surely preserved by choosing a different topological model
    $G \curvearrowright (B', \mu')$ for the action $G \curvearrowright (B, \mu)$, i.e., if 
    $\pi\colon (B,\mu) \to (B',\mu')$ is a $G$-equivariant essentially invertible measurable
    map, then almost surely $\pi_*\mu_\omega = \mu_\omega'$.
    \item If $f\colon B \to \C$ is a $\mu$-integrable measurable function, then 
    it is also $\mu_\omega$-integrable almost surely and one has the barycenter equation
    \begin{equation*}
      \int_B f\dmu = \E_\omega\left(\int_B f \dmu_\omega\right).
    \end{equation*}

    \item Let $\pi\colon (B,\mu) \to (C,\nu)$ be a measurable, measure-preserving factor map to another stationary action 
    $G \curvearrowright (C, \nu)$. Suppose $\pi$ is relatively 
    measure-preserving and $(\mu_c)_{c\in C}$ is a disintegration of $\mu$ over $\nu$. Then 
    almost surely 
    \begin{equation*}
      \mu_\omega = \int_C \mu_c \dnu(c).
    \end{equation*}
    \item Suppose $\omega\in\Omega$ is such that the limit $\mu_\omega$ exists, then it 
    satisfies the following invariance property: 
    For $\lambda \defeq \sum_{n=0}^\infty 2^{-(n+1)}m^{*n}$ one has that $\lambda$-almost 
    every $g\in G$ satisfies
    \begin{equation*}
    \mu_\omega = \lim_{n\to\infty} (X_1(\omega)\cdots X_n(\omega))_*g_*\mu.
    \end{equation*}
  \end{enumerate}
\end{lemma}
\begin{proof}
  By linearity, we only need to prove (i) for positive functions. For positive 
  and measurable $f\colon B \to \R$,
  \begin{align*}
    \int_\Omega \int_B f \,\ud (X_1(\omega)\cdots X_n(\omega))_*\mu  \dP(\omega)
    = \int_B f \,\ud m^{*n}*\mu 
    = \int_B f \dmu.
  \end{align*}
  This shows that $\Gamma_nf \in \uL^1(\Omega,\P)$ and 
  $\|\Gamma_n f\|_{\uL^1(\Omega,\P)} \leq \|f\|_{\uL^1(B, \mu)}$.
  It is easy to verify that the sequence $(\Gamma_nf)_n$ defines a martingale, see 
  \cite[Lemma II.2.1]{BougerolLacroix1985}. Since its $\uL^1$-norm is bounded by $\|f\|_{\uL^1(B, \mu)}$,
  the martingale convergence theorem shows that there is a function $\Gamma f \in \uL^1(\P, \Omega)$ 
  such that $\Gamma_n f \to \Gamma f$ almost surely.
  
  Assertion (ii) is readily verified using that $\mu$ is $m$-stationary and the inequality 
  $|\Gamma f| \leq \Gamma |f|$ that follows from positivity. For (iii), let $\mathcal{F} \subset \uC_0(B)$
  be a countable dense subset of $\uC_0(B)$ and choose for every $f\in\mathcal{F}$ a set $A_f$ as in (i).
  If we set $A_\mathcal{F} \defeq \bigcap_{f\in\mathcal{F}} A_f$, then for all $\omega \in A_{\mathcal{F}}$
  \begin{equation*}
    \Gamma f(\omega) = \lim_{n\to\infty}\int_B f \,\ud (X_1(\omega)\cdots X_n(\omega))_*\mu 
  \end{equation*}
  exists. Thus, the sequence $(X_1(\omega)\cdots X_n(\omega))_*\mu$ of probability measures 
  converges to some probability measure $\mu_\omega$ on $\mathcal{F}$ and hence on $\uC_0(B)$.
  
  For (iv), let $\mathcal{F} \subset \uL^1(B,\mu)$ be a countable subset and 
  let $A \subset \Omega$ be a subset of measure $\P(A) = 1$ such that for all 
  $\omega \in A$ we have $(X_1(\omega)\cdots X_n(\omega))_*\mu \to \mu_\omega$ in the 
  weak* topology. We may assume without loss of generality that $\mathcal{F} = \{f\}$ 
  and that for all $\omega\in A$
  \begin{equation*}
    (\Gamma f)(\omega) = \lim_{n\to\infty} \int_B f \,\ud (X_1(\omega)\cdots X_n(\omega))_*\mu
  \end{equation*}
  exists. We claim that for almost every $\omega$ one has $(\Gamma f)(\omega) = \int_B f \dmu_\omega$.
  Recall that the smallest subspace of the bounded measurable functions $\uB(B)$ that contains 
  $\uC_0(B)$ and is closed 
  under bounded pointwise convergence of sequences is precisely $\uB(B)$. So suppose 
  $(g_n)_n$ is a sequence of bounded measurable functions such that almost surely
  \begin{equation*}
    (\Gamma g_n)(\omega) = \int_B g_n \dmu_\omega.
  \end{equation*}
  If the $g_n$ are uniformly bounded and converge pointwise to $g\in\uB(B)$,
  it follows from Lebesgue's theorem that $\int_B g_n \dmu_\omega \to \int_B g \dmu_\omega$.
  On the other hand, $g_n \to g$ in $\uL^1(B,\mu)$, so $\Gamma g_n \to \Gamma g$ in $\uL^1(\Omega,\P)$.
  By passing to a subsequence, we may assume that $\Gamma g_n \to \Gamma g$ almost everywhere w.r.t.\ $\P$.
  Thus, almost surely $(\Gamma g)(\omega) = \int_B g\dmu_\omega$. It follows that, for every $g\in \uB(B)$,
  almost surely $(\Gamma g)(\omega) = \int_B g\dmu_\omega$. Repeating a similar argument using cutoffs 
  and the dominated convergence theorem 
  proves the statement for arbitrary $g\in\uL^1(B,\mu)$ and hence for the function $f$ initially considered.
  
  Assertions (v) and (vi) follow from (iv). Part (vii) is a reformulation of (ii) and 
  for statement (ix), see \cite[Lemma II.2.1]{BougerolLacroix1985}.
  It remains to prove (viii), so let $\pi\colon (B,\mu) \to (C,\nu)$ be a measure-preserving and $G$-equivariant measurable
  map to another stationary action $G \curvearrowright (C, \nu)$.
  
  Then by (iii), there is a set $F \subset \Omega$ of measure $\P(F) = 1$ such that 
  for all $\omega \in F$, the following weak*-limits exist:
  \begin{align*}
    \mu_\omega = \lim_{n\to\infty} (\omega_1\cdots \omega_n)_*\mu, \\
    \nu_\omega = \lim_{n\to\infty} (\omega_1\cdots \omega_n)_*\nu.
  \end{align*}
  Let
  $\mathcal{F} \subset \uC_0(B)$ be a countable dense subset and disintegrate 
  $\mu$ over $\nu$ as $(\mu_c)_{c\in C}$. Then applying (iv) to the set 
  \begin{equation*}
   \mathcal{G} \defeq \{c \mapsto \langle f, \mu_c\rangle \mid f \in \mathcal{F}\}
  \end{equation*}
  we obtain a set $G \subset F \subset \Omega$ of measure $\P(G) = 1$ such that 
  for all $\omega \in G$
  \begin{equation*}
    \lim_{n\to\infty} \int_C \langle f, \mu_c\rangle \,\ud (\omega_1 \dots \omega_n)_* \nu(c) 
    = \int_C \langle f, \mu_c\rangle \dnu_\omega(c).
  \end{equation*}
  Thus, the assumption that $\pi\colon (B,\mu) \to (C,\nu)$ is relatively measure-preserving
  implies that for every $\omega \in G$ and $f\in \mathcal{F}$
  \begin{align*}
    \int_B f(x) \dmu(x) 
    &= \lim_{n\to\infty} \int_B f(x) \,\ud(\omega_1\cdots \omega_n)_*\mu \\
    &= \lim_{n\to\infty} \int_B f(\omega_1\cdots\omega_n x) \,\ud\mu(x) \\
    &= \lim_{n\to\infty} \int_C \int_{X_c} f(\omega_1\cdots\omega_n x) \,\ud\mu_c(x) \dnu(c) \\
    &= \lim_{n\to\infty} \int_C \langle f, (\omega_1\cdots\omega_n)_*\mu_{c}\rangle \dnu(c) \\
    &= \lim_{n\to\infty} \int_C \langle f, \mu_{\omega_1\cdots\omega_nc}\rangle \dnu(c) \\
    &= \lim_{n\to\infty} \int_C \langle f, \mu_{c}\rangle \,\ud(\omega_1\cdots\omega_n)_*\nu(z) \\
    &= \int_C \langle f, \mu_c \rangle \dnu_\omega(c) \\
    &= \int_C \int_{X_c} f(x) \dmu_c(x) \dnu_\omega(x).
  \end{align*}
  Thus, for relatively measure-preserving extensions, the disintegration of $\mu_\omega$ w.r.t.\ 
  $\nu_\omega$ is almost surely
  \begin{equation*}
    \mu_\omega = \int_C \mu_c \dnu_\omega(c).
  \end{equation*}
\end{proof}

\begin{remark}\label{rem:condmeas} Before we proceed, we collect the following remarks on conditional measures.
  \begin{enumerate}[(i)]
   \item Per definition, conditional measures can only be considered for actions on (locally)
  compact spaces. However, we can always choose a standard probability model and the lemma
  above shows that the conditional measures do not depend on the choice of topological model.
  We may therefore speak of conditional measures on standard probability spaces without ambiguity.
   \item The action in \cref{lem:condmeasures} above is measure-preserving if and only if 
  all conditional measures $\mu_\omega$ equal $\mu$ and it is called $m$-\textbf{proximal} 
  if $\mu_\omega$ is a Dirac measure for almost every $\omega\in\Omega$. In this case,
  every factor is $m$-proximal and in particular, there can be no nontrivial measure-preserving 
  factor. More generally, a factor map $\pi\colon (B,\mu) \to (C, \nu)$ of $m$-stationary actions
  is called a \textbf{proximal factor map} if the induced maps $(B,\mu_\omega) \to (C,\nu_\omega)$ 
  are almost surely isomorphisms.

  \item Note that our definition of proximal extensions differs slightly from that 
  given in \cite{FurstenbergGlasner2010} (see the moderating text before 
  \cite[Proposition 1.2]{FurstenbergGlasner2010}). There, it is only required
  that the map $\pi\colon (X,\mu_\omega) \to (Y, \nu_\omega)$ is $\P$-almost
  surely $\mu_\omega$-almost everywhere one-to-one; measurability of the inverse
  and hence essential invertibility is not required there. However, they 
  work in the setting of standard Borel spaces and it is a result from descriptive
  set theory that if $f\colon X \to Y$ is a Borel measurable map between standard
  Borel spaces and $A \subset X$ is a Borel measurable subset such that $f|_A$ is
  injective, then $f(A)$ is Borel measurable and $f|_A \colon A \to f(A)$ 
  is a Borel isomorphism (see \cite[Corollary 15.2]{Kechris1995}). 
  Thus, in the setting of standard probability spaces, 
  the two definitions of proximal extensions are equivalent. We choose to work with 
  the more convenient definition in terms of isomorphy to avoid further reference to 
  descriptive set theory.
  \end{enumerate}  
\end{remark}

\textbf{The Halmos--von Neumann representation theorem for stationary actions.} 
Next, we turn towards the stationary version of the classical representation theorem 
for isometric actions.

\begin{lemma}\label{lem:HvN}
  Let $G \curvearrowright \uX$ be a nonsingular ergodic action. Then 
  the following assertions are equivalent.
  \begin{enumerate}[a)]
    \item The action of $G$ on $\uX$ is isometric.
    \item The action of $G$ on $\uX$ is isomorphic to 
    a minimal action of $G$ on a compact homogeneous space via rotations. 
  \end{enumerate}
  If these hold and $\mu$ is $m$-stationary for a probability measure on $G$, 
  then it is $G$-invariant and it is the 
  unique $m$-stationary measure of the action. Moreover, the homogeneous space in b)
  is necessarily metrizable in this case.
\end{lemma}

\begin{proof}
  Define the $\uC^*$-subalgebra
  \begin{equation*}
    \mathcal{A} \defeq \overline{\bigcup\left\{F \subset \uL^\infty(\uX) \mmid \begin{matrix}
                                                                                F \text{ finite-dimensional } \\
                                                                                \text{and } G\text{-invariant}
                                                                               \end{matrix}
  \right\}}^{\uL^\infty(\uX)}.
  \end{equation*}
  Since $\mathcal{A}$ is $G$-invariant, it corresponds to a topological model $G \curvearrowright (M, \mathcal{B}, \nu)$ 
  for the action with
  \begin{equation*}
    \uC(M) = \overline{\bigcup\left\{F \subset \uC(M) \mmid \begin{matrix}
                                                                                F \text{ finite-dimensional } \\
                                                                                \text{and } G\text{-invariant}
                                                                               \end{matrix}
  \right\}}^{\uC(M)}.
  \end{equation*}
  It follows that the Ellis semigroup $\uE(M, G)$ is a compact group of continuous maps that acts transitively 
  and hence uniquely ergodically on $M$ and so $G \curvearrowright M$ is isomorphic to the action 
  $G \curvearrowright \uE(M,G)/\operatorname{Stab}(x)$
  for any $x\in M$.
  
  Now suppose additionally that $\mu$ is an $m$-stationary ergodic measure and $X = K/H$ is a 
  homogeneous space of some compact group $K$ on which $G$ acts uniquely ergodically via a map 
  $\phi\colon G \to K$. Consider a random walk of $G$ on $X$ with law $m$ and let 
  $\mu_\omega$ be one of the conditional measures. Let 
  $\lambda \defeq \sum_{n=0}^\infty 2^{-(n+1)}m^{*n}$, then \cref{lem:condmeasures} shows that for every $g\in\supp(\lambda)$
  \begin{equation*}
     \lim_{n\to\infty} X_1(\omega)_*\cdots X_n(\omega)_*\mu
    = \lim_{n\to\infty} X_1(\omega)_*\cdots X_n(\omega)_*g_*\mu.
  \end{equation*}
  In other words,
  \begin{equation*}
    \lim_{n\to\infty} \phi\left(X_1(\omega)\right)\cdots \phi\left(X_n(\omega)\right)\mu
   =  \lim_{n\to\infty} \phi\left(X_1(\omega)\right)\cdots \phi\left(X_n(\omega)\right)\phi(g)\mu.
  \end{equation*}
  If we use that $K$ is compact and pass to a subsequence, we may assume that 
  $\phi(X_1(\omega))\cdots \phi(X_n(\omega))$ converges to some element $k\in K$. In that 
  case, we obtain $k\mu = kg\mu$ and hence $\mu = g\mu$. Thus, we have shown that $\mu$ is 
  invariant under $\supp(\lambda)$ and since $\supp(m)$ generates $G$ as a group, it follows
  that $\mu$ is $G$-invariant. Since the action of $G$ on $K/H$ is uniquely ergodic, $\mu$ must be the 
  Haar measure. Finally, in this case, the metrizability of $K/H$ follows purely from the 
  fact that $\uL^1(\uX) \cong \uL^1(K/H)$ is separable, see \cref{lem:metrizable}.
\end{proof}

\begin{lemma}\label{lem:metrizable}
  Let $K$ be a compact group and $H \subset K$ a closed subgroup, let
  $(N_j)_{j\in J}$ be the collection of irreducible $K$-left-invariant subspaces
  of $\uL^2(K/H)$. Then the $N_j$ are finite-dimensional, mutually orthogonal,
  \begin{equation*}
   \uL^2(K/H) = \overline{\lin\left\{ N_j \mmid j \in J\right\}}^{\uL^2(K/H)}
  \end{equation*}
  and 
  \begin{equation*}
   \uC(K/H) = \overline{\lin\left\{ N_j \mmid j \in J\right\}}^{\uC(K/H)}.
  \end{equation*}
  In particular, $\uL^2(K/H)$ is separable if and only if $\uC(K/H)$ is separable.
\end{lemma}
\begin{proof}
  If $H = \{e\}$, the statements can be found in \cite[Section 5.2]{Folland2015}. To reduce
  the general case to this, use the conditional 
  expectation 
  \begin{equation*}
    P_H \colon \uL^2(K) \to \uL^2(K/H), \quad (P_Hf)(x) = \int_H f(xh)\dm(h)
  \end{equation*}
  and the fact that $P_H$ restricts to a surjective contraction 
  $P_H \colon \uC(K) \to \uC(K/H)$.
\end{proof}

\begin{corollary}\label{cor:kro_mpr}
  Every nonsingular action $G \curvearrowright (X, \mu)$ has a maximal isometric factor.
  If $(G,m)$ is a measured group and the measure of the maximal isometric factor is $m$-stationary, 
  then it has to be invariant.
\end{corollary}
\begin{proof}
  First, let 
  \begin{equation*}
    \mathcal{A} \defeq \overline{\bigcup\left\{F \subset \uL^\infty(\uX) \mmid \begin{matrix}
                                                                                F \text{ finite-dimensional } \\
                                                                                \text{and } G\text{-invariant}
                                                                               \end{matrix}
  \right\}}^{\uL^\infty(\uX)}
  \end{equation*}
  as in the proof of \cref{lem:HvN} and let $G \curvearrowright (K, \nu)$ be a standard
  probability model for this subalgebra. It is clear from the correspondence between 
  factors and subalgebras discussed in \cref{factor-algebra-corr} that this is the 
  maximal isometric factor of $G \curvearrowright (X, \mu)$. For the invariance statement,
  we may assume without loss of generality that $G \curvearrowright (X,\mu)$ itself is 
  isometric. In that case, let $M \subset \uL^\infty(X,\mu)$ be a finite-dimensional 
  $G$-invariant subspace. Then the action of $G$ on $M$ induces a map $\Phi\colon G \to \Aut(M)$
  which maps $G$ to a bounded subgroup of $\Aut(M)$. Thus, $\Phi(G)$ is relatively compact 
  in $\Aut(M)$; denote its closure by $K \defeq \overline{\Phi(G)}$.
  
  Since the Haar measure $\um_K$ is $K$-invariant, any $\um_K$-stationary measure is 
  $K$-invariant. Thus, $G$-invariance of the restriction $\mu|_M$ of $\mu$ to the subspace $M\subset \uL^\infty(X,\mu)$ 
  is equivalent to 
  \begin{equation*}
    \um_K * \mu|_M = \mu|_M.
  \end{equation*}
  To prove this equality, it suffices to prove that
  \begin{equation*}
    \lim_{N\to\infty}\frac{1}{N}\sum_{n=1}^N m^{*n} = \um_K.
  \end{equation*}
  For this, we can employ \cref{lem:HvN}: Since $\Phi(G)$ is dense in $K$, the 
  left-action of $G$ on $(K,\um_K)$ is ergodic and isometric and thus $\um_K$ 
  is the unique stationary measure under the $G$-action. Thus, Breiman's law
  states that for every probability measure $\nu$ on $K$
  \begin{equation*}
    \lim_{N\to\infty}\frac{1}{N}\sum_{n=1}^N m^{*n}*\nu = \um_K.
  \end{equation*}
  Setting $\nu = \delta_e$ shows that $\mu|_M $ is $G$-invariant. Since $G \curvearrowright (X,\nu)$
  is isometric, the $G$-invariant finite-dimensional subspaces $M \subset \uL^\infty(\uX)$ 
  are dense in $\uL^2(\uX)$. Thus, $\mu$ is $G$-invariant.
\end{proof}

We call this maximal isometric factor the \textbf{Kronecker factor} of the system.

\textbf{The Kronecker dichotomy for stationary actions.} The following theorem
and its proof closely follow \cite[Section 4]{Bjorklund2017}
and the proof given there. We include them for the reader's convenience since
we shall generalize the arguments below.

\begin{theorem}\label{thm:statdichotomy}
  A stationary action $(G, m) \curvearrowright (X, \mu)$ is weakly mixing 
  if and only if its Kronecker factor is trivial.
\end{theorem}
\begin{proof}
  For the easy direction, assume that the action of $G$ on $(X, \mu)$ is 
  weakly mixing (and in particular, ergodic). Suppose there exists a nontrivial isometric factor $(Y, \nu)$  
  of $(X, \mu)$ with factor map $\pi\colon X \to Y$. Then $\nu$ is $m$-stationary 
  and hence $G$-invariant by \cref{cor:kro_mpr}.
  Since $(Y, \nu)$ is a nontrivial isometric factor, there is a nontrivial 
  finite-dimensional $G$-invariant subspace $M \subset \uL^\infty(Y, \nu)$ 
  and we denote by $P_M \colon \uL^2(Y, \nu) \to M$ the corresponding orthogonal projection.
  The projection $P_M\in\mathscr{L}(\uL^2(Y, \nu))$ is a Hilbert--Schmidt operator and 
  under the canonical identification 
  $\operatorname{HS}(\uL^2(Y, \nu)) \cong \uL^2(Y\times Y, \nu\otimes\nu)$, $P_M$ corresponds 
  to a function $k\in \uL^2(Y\times Y, \nu\otimes \nu)$. As explained above, the fact that 
  $P_M$ intertwines with the dynamics on $\uL^2(Y, \nu)$ implies that $k$ is invariant under 
  the diagonal action on $\uL^2(Y\times Y, \nu \otimes \nu)$. We thus obtain a 
  nontrivial 
  fixed function $k\in \uL^2(Y\times Y, \nu \otimes \nu)$ and this lifts to the nontrivial 
  fixed function $k \circ (\pi \times \id_Y) \in \uL^2(X\times Y, \mu \otimes \nu)$. Thus, 
  the action of $G$ on $(X, \mu)$ cannot be weakly mixing, a contradiction. This shows 
  that weak mixing implies that the Kronecker factor is trivial.
  
  Conversely, suppose that the action of $G$ on $(X, \mu)$ is not weakly mixing; we 
  need to show that there is a nontrivial isometric factor. So suppose $G \curvearrowright (Y, \nu)$
  is an ergodic measure-preserving action such that there is a nonconstant $G$-invariant function
  $k\in \uL^\infty(X\times Y, \mu\otimes\nu)$. We may assume that 
  $k \perp \1\otimes\1$, i.e., $\int_{X\times Y} k \,\mathrm{d}\mu\otimes\nu = 0$. 
  Denote by $B_1$ the closed unit ball in $\uL^2(Y, \nu)$ endowed with the 
  weak topology
  and consider the map
  \begin{equation*}
    p_k \colon X \to B_1, \quad x \mapsto k(x, \cdot).
  \end{equation*}
  Since $k$ is $G$-invariant almost everywhere, this map is almost everywhere equivariant
  for every $s\in G$ and thus becomes a factor map if we endow 
  $B_1$ with the push-forward measure. 
  \cref{lem:ball_inv} below shows that since $G$ acts unitarily on $\uL^2(Y,\nu)$
  and hence isometrically on $B_1$, this measure is invariant.
  Thus, if we can show that
  \begin{equation*}
    \Phi\colon B_1\times B_1 \to \R, \quad f, g \mapsto \langle f, g\rangle
  \end{equation*}
  is a nonconstant invariant function on $B_1\times B_1$,
  the usual Kronecker dichotomy
  will imply that there exists a nontrivial isometric factor which concludes the proof. 
  That $\Phi$ is essentially nonconstant is equivalent to proving that 
  \begin{equation*}
    \tilde{\Phi} \defeq \Phi\circ (p_k\times p_k) \colon X\times X \to \R, \quad (x, x') \mapsto \int_Y k(x, y)k(x', y) \dnu(y)
  \end{equation*}
  is not essentially constant. To see this, observe that if $K\colon \uL^2(X,\mu) \to \uL^2(Y,\nu)$ is the Hilbert--Schmidt 
  integral operator corresponding to $k$ under the isomorphism 
  $\uL^2(X\times Y, \mu\otimes\nu) \cong \operatorname{HS}(\uL^2(X,\mu), \uL^2(Y,\nu))$, then 
  the kernel 
  of the Hilbert--Schmidt integral operator $K^*K$ is readily computed to be $\tilde{\Phi}$. Since $k$ is nonzero, $K$ is nonzero. 
  Moreover, since $\ker(K^*) = \overline{\ran}(K)^\perp$, $K^*K \neq 0$. 
  The statement that $\tilde{\Phi}$ is not essentially constant is equivalent to 
  $K^*K$ not being a multiple of the orthogonal projection $\E = \uI_{\1\otimes\1}$ 
  onto the constant functions. Suppose, for the sake of contradiction, that 
  $K^*K = c\E$. Then $K^*K = K^*K\E$ and since $K^*$ is injective on $\ran(K)$,
  we may cancel $K^*$.
  Thus, $K = K\E$ and an elmentary computation shows that thus, $k = \1_\uX\otimes f$
  where $f\in \uL^\infty(\uY)$ is given by $f(y) = \int_\uX k(x,y)\dmu(x)$.
  Invariance of $k$ translates into invariance of $f$. Since $\uY$ was assumed
  to be ergodic, we see that $k$ is constant. This contradicts $k$ being nonconstant,
  so the assumption $K^*K = c\E$ must have been impossible.
  This shows that 
  $\tilde{\Phi}\in \uL^\infty(X\times X, \mu \otimes \mu)$ is a nonconstant fixed function.
  
  Thus, we have found a measure-preserving 
  factor $(B_1, (p_k)_*\mu)$ and a nonconstant invariant function $\Phi$ in 
  $\uL^\infty(B_1\times B_1, (p_k)_*\mu \otimes (p_k)_*\mu)$. By the classical Kronecker dichotomy
  \cref{thm:krodichotomy},
  there is a nontrivial isometric factor of this system which concludes the proof.
\end{proof}

\begin{lemma}[{\cite[Lemma 4.3]{Bjorklund2017}}]\label{lem:ball_inv}
  Let $H$ be a Hilbert space, $H_\uw$ denote $H$ endowed with the weak topology, 
  and $(G,m) \curvearrowright (H_\uw,\mu)$ be  a stationary action 
  by unitary operators. Then $\mu$ is $G$-invariant.
\end{lemma}
\begin{proof}
  This is proved in \cite[Lemma 4.3]{Bjorklund2017} for the case that $G$ is countable
  and $\supp(\mu)$ 
  is contained in the unit ball $B_1$ but the general case follows as a consequence.
  We give a more general proof in \cref{lem:bundle_inv} below.
\end{proof}

One example of a weakly mixing action is given by the Poisson boundary $\Pi(G,m)$ of a measured group $(G,m)$
which is the universal $m$-proximal system for $(G,m)$.

\begin{theorem}[{\cite[Theorem 3.1]{Bjorklund2017}}]\label{lem:poisson}
  Let $(G, m)$ be a countable measured group, denote its Poisson boundary by $(Z, \zeta)$, and 
  let $(G, \check{m}) \curvearrowright (X, \check{\mu})$ be an ergodic stationary action. Then the diagonal action 
  $G \curvearrowright (Z\times X, \zeta\otimes\check{m})$ is ergodic.
\end{theorem}

As a corollary, one obtains that the action of $(G,m)$ on its Poisson boundary is weakly mixing.

\section{The relative Kronecker dichotomy for stationary actions} \label{sec:stationaryrelativedichotomy}

The goal of this section is to extend the dichotomy \cref{thm:statdichotomy} 
to the setting of extensions in \cref{thm:relstatdichotomy}.

\begin{theorem}\label{thm:relstatdichotomy}
  Let $\pi\colon (X, \mu) \to (Z, \zeta)$ be an extension of stationary $(G,m)$-actions. 
  Then exactly one of the following is true.
  \begin{enumerate}[(i)]
   \item The extension $\pi$ is weakly mixing.
   \item There is a nontrivial intermediate extension 
    $(X, \mu) \to (\tilde{Z},\tilde{\zeta}) \to (Z, \zeta)$
    such that the extension $(\tilde{Z},\tilde{\zeta}) \to (Z, \zeta)$ is isometric
    and relatively measure-preserving.
  \end{enumerate}
\end{theorem}

The proof of \cref{thm:statdichotomy} used that the closed unit ball in a Hilbert space 
is a weak*-compact subset. In the relative setting, we will need to think about Hilbert bundles
instead of Hilbert spaces to make an analogous statement and then adapt the invariance 
lemma \cref{lem:ball_inv}.

\begin{definition}
  A \textbf{(continuous) Hilbert bundle} $p\colon E \to L$ consists of topological spaces 
  $E$ and $L$ and a continuous, open surjection $p\colon E\to L$ such that $L$ is compact
  and
  \begin{enumerate}
   \item every fiber $E_l \defeq p^{-1}(l)$ is a Hilbert space,
   \item if we define $E\times_L E \defeq \{ (e, f) \in E\times E \mid p(e) = p(f)\}$,
   then the maps
   \begin{alignat*}{3}
     +&\colon E \times_L E \to E, \quad &&(e, f) \mapsto e+f, \\
     \cdot & \colon \C\times E\to E, \quad &&(\lambda, e) \mapsto \lambda e
   \end{alignat*}
   are continuous,
   \item the scalar product 
   \begin{equation*}
    (\cdot|\cdot)\colon E\times_L E \to \C, \quad (e, f) \mapsto (e|f)_{E_{p(e)}} 
   \end{equation*}
   is continuous,
   \item for each $l\in L$ one obtains a neighborhood base of $0_l \in E_l$ 
   by considering the sets
   \begin{equation*}
    \{ e\in E \mid p(e) \in U, \|e\| < \epsilon\}
   \end{equation*}
   for neighborhoods $U\subset L$ of $l$ and $\epsilon > 0$.
  \end{enumerate}
  A map $\sigma\colon K \to E$ is called a \textbf{section} of the Hilbert bundle if $p\circ \sigma = \id_K$.
  We denote the space of all continuous sections by $\Gamma(E)$ which becomes a Banach space with
  the natural norm $\|\sigma\| = \sup_{l\in L} \|\sigma(l)\|_{E_l}$. The bundle $E$ is called \textbf{separable} 
  if there is a countable subset 
  $\mathcal{F} \subset \Gamma(E)$ such that for each $l\in L$ the set $\{\sigma(l)\mid \sigma \in \mathcal{F}\}$
  is dense in $E_l$.
  
  The \textbf{weak topology} on $E$ is the initial topology induced by the maps
  \begin{alignat*}{3}
    p &\colon E \to L, \quad &&e\mapsto p(e) \\
    (\cdot | \sigma(p(\cdot)))&\colon E \to \C, \quad && e \mapsto (e | \sigma(p(e)))
  \end{alignat*}
  for $\sigma \in \Gamma(E)$. We denote $E$ endowed with the weak topology by $E_\uw$. 
  
  A \textbf{unitary representation} or \textbf{unitary action} $\pi$ of a group $G$ on a Hilbert bundle $p\colon E \to L$ 
  is a continuous group action $\pi\colon G \to \Homeo(E)$ such that there exists 
  a continuous action $G \to \Homeo(L)$ 
  that satisfies
  \begin{itemize}
   \item $p(\pi(g)e) = gp(e)$ for every $e\in E$ and $g\in G$ and
   \item $\pi(g)|_{E_l}\colon E_l \to E_{gl}$ is a unitary operator 
   for every $l\in L$ and $g\in G$.
  \end{itemize}
  Such a representation $\pi$ induces an action of $G$ on the space of sections 
  $\sigma$ of $p\colon E\to L$ via $(s.\sigma)_l \defeq \pi(s)\sigma_{s^{-1}l}$.
\end{definition}

\begin{remark}
  All Hilbert bundles that occur will be continuous, so we usually simply speak of Hilbert 
  bundles. One of the standard references for Hilbert bundles, and more generally bundles of 
  topological vector spaces, is \cite{Gierz1982}. Gierz defines Hilbert bundles somewhat
  differently in \cite[Definition 1.5]{Gierz1982} but shows in \cite[Theorems 2.4, 2.9]{Gierz1982}
  that the definition we gave above is equivalent. We emphasize that by \cite[Theorem 2.9]{Gierz1982},
  Hilbert bundles are \enquote{full}, i.e., for every element of a Hilbert bundle there is a continuous 
  section that passes through it. In particular, the norm and weak topology on a Hilbert bundle
  are both Hausdorff.
\end{remark}

We recall the following key examples for Hilbert bundles from 
\cite[Definition 6.2]{EdKr2022}. Part of this is based on unpublished notes 
kindly shared by M.\ Haase in private communication.

\begin{example}\label{ex:canonical_disintegration}
  Suppose $q\colon K \to L$ is a continuous surjection between compact spaces
  and that $(\mu_l)_{l\in L}$ is a weak$^*$-continuous family of probability 
  measures on $K$ such that $\supp(\mu_l) \subset K_l$. Set 
  \begin{equation*}
    \uL^2_q(K, \mu) \defeq \bigcup_{l\in L} \uL^2(K_l, \mu_l)
  \end{equation*}
  and let $p\colon \uL^2_q(K, \mu) \to L$ be the canonical map. Endow 
  $\uL^2_q(K, \mu)$ with the topology generated by the sets 
  \begin{equation*}
    V(F, U, \epsilon) \defeq \left\{ f\in \uL^2_q(K, \mu) \mmid p(f) \in U, \| f - F|_{K_{p(f)}}\|_{\uL^2(K_{p(f)},\mu_{p(f)})} < \epsilon \right\}
  \end{equation*}
  for $F\in \uC(K)$, open $U\subset L$, and $\epsilon > 0$. Then it is standard to verify that 
  $p\colon \uL^2_q(K, \mu)\to L$ defines a continuous Hilbert bundle. 
    
  Given a relatively measure-preserving extension $p\colon \uX \to \uY$ of nonsingular
  actions of a lcsc group $G$,
  one can associate to it a canonical Hilbert bundle. To see this,
  observe that if we set 
  \begin{align*}
    \mathcal{A} &\defeq \left\{ f\in \uL^\infty(\uX) \mmid g \mapsto T_{g^{-1}}f \text{ is } \|\cdot\|_{\uL^\infty(\uX)}\text{-continuous}\right\}, \\
    \mathcal{B} &\defeq \left\{ f\in \uL^\infty(\uY) \mmid g \mapsto T_{g^{-1}}f \text{ is } \|\cdot\|_{\uL^\infty(\uX)}\text{-continuous}\right\},
  \end{align*}
  then since the extension is relatively measure-preserving, $\E_\uY(\mathcal{A}) = \mathcal{B}$ (use \cref{lem:rmp_char})
  and so 
  by \cref{lem:uniformcont}, these $\uC^*$-algebras gives rise to 
  a topological model $q\colon (K, \mu_K) \to (L, \mu_L)$ for the 
  extension $p\colon \uX \to \uY$. The corresponding conditional expectation $\E_q\colon \uL^1(K, \mu_K) \to \uL^1(L, \mu_L)$ then 
  restricts to a continuous operator $\E_q\colon \uC(K) \to \uC(L)$.
  
  For $l\in L$, let $\delta_l\in \uM(L)\cong \uC(L)'$ be the Dirac measure in $l$ and 
  set $\mu_l \defeq \E_q'\delta_l \in \uC(K)'$. Since $\E_q\colon \uC(K) \to \uC(L)$ is continuous, 
  this defines a weak$^*$-continuous family of probability measures on $K$. That $\supp(\mu_l) \subset K_l$
  for every $l\in L$ follows from the fact that $\E_q$ is a $\uC(L)$-module homomorphism. Since 
  $\E_q$ preserves the integral of functions, for $f\in \uC(K)$
  \begin{equation*}
   \int_L \int_{K_l} f \dmu_l \dmu_L(l)
   = \int_L \int_{K_l} \left\langle f, \E_q'\delta_l\right\rangle \dmu_L(l)
   = \int_L (\E_Lf)(l)\dmu_L(l)
   = \int_K f \dmu_K.
  \end{equation*}
  Thus, $(\mu_l)_{l\in L}$ is a disintegration of $\mu_K$
  and, as above, we can associate the Hilbert bundle $\uL^2_q(K, (\mu_l)_{l\in L})$ to it. The upside
  that one obtains a continuous Hilbert bundle comes at the price that this Hilbert bundle is in no 
  way separable, even if the initial probability spaces $\uX$ and $\uY$ were separable. The following 
  lemma provides a remedy to this drawback.
\end{example}

\begin{lemma}\label{lem:disint_model}
  Let $p\colon \uX \to \uY$ be a relatively measure-preserving extension of nonsingular $G$-actions
  where $G$ is a lcsc group.
  Then there is a topological model $q\colon (K, \mu)\to (L,\nu)$ of continuous $G$-actions on
  compact metric spaces such that $\mu$ admits a unique weak$^*$-continuous $G$-equivariant disintegration over $(L, \nu)$.
\end{lemma}
\begin{proof}
  We start with the canonical (not necessarily metrizable) topological model
  $q\colon (K, \mu) \to (L, \nu)$
  constructed in \cref{ex:canonical_disintegration}. 
  Let $\mathcal{A}_1 \subset \uC(K)$ be a separable $G$-invariant $\uC^*$-subalgebra 
  of $\uC(K)$ that is dense in $\uL^1(K, \mu)$
  and let 
  \begin{equation*}
    \E_{L} \colon \uC(K)\cong \uL^\infty(K, \mu) \to \uL^\infty(L, \nu) \cong \uC(L)
  \end{equation*}
  be the conditional expectation. Let $\mathcal{B}_1 \defeq \uC^*_G(\E_{L}(\mathcal{A}_1)) \subset \uC(L)$ be the 
  $G$-invariant $\uC^*$-subalgebra generated by $\E_{L}(\mathcal{A}_1)$. Recursively define two increasing 
  sequences of separable, $G$-invariant $\uC^*$-subalgebras
  via
  $\mathcal{A}_{n+1} \defeq \uC^*_G(\mathcal{A}_n \cup T_q(\mathcal{B}_n))$ and 
  $\mathcal{B}_{n+1} \defeq \uC^*_G(\E_{L}(\mathcal{A}_{n+1}))$. Set 
  \begin{equation*}
   \mathcal{A} \defeq \overline{\bigcup_{n\in\N} \mathcal{A}_n}^{\|\cdot\|_{\uC(K)}}, \qquad 
   \mathcal{B} \defeq \overline{\bigcup_{n\in\N} \mathcal{B}_n}^{\|\cdot\|_{\uC(L)}}.
  \end{equation*}
  By construction, $\mathcal{A} \subset \uC(K)$ and $\mathcal{B} \subset \uC(L)$ are 
  $G$-invariant, separable $\uC^*$-subalgebras that are dense in $\uL^1(K, \mu)$
  and $\uL^1(L, \nu)$, respectively. Moreover, $T_q(\mathcal{B}) \subset \mathcal{A}$
  and the conditional expectation restricts to a $G$-equivariant map $\E_{L} \colon \mathcal{A} \to \mathcal{B}$.
  Thus, the resulting topological model $q'\colon (K', \mu') \to (L', \nu')$ has a $G$-equivariant 
  weak*-continuous disintegration $(\mu_l')_{l\in L'}$ given by the adjoint of the conditional expectation via
  $\mu_l' \defeq \E_L'\delta_l$. This is unique due to the disintegration theorem \cref{thm:disintegration}.
\end{proof}

It is not hard to show that the closed \enquote{unit ball} in a Hilbert bundle 
is compact in the weak topology. This can be found in \cite[Proposition 2.7]{Renault2012}
and also follows from \cite[Proposition 15.3]{Gierz1982}. Both work with slightly different definitions of the 
weak(*) topology but it is not hard to show that they are all equivalent.

\begin{proposition}\label{prop:weaklycompact}
  Let $p\colon E\to L$ be a Hilbert bundle over a compact space. Then 
  every norm-bounded subset of $E$ is relatively compact in the weak topology. 
  In particular, the weakly closed set $B_1 \defeq \{x\in E \mid \|x\| \leq 1\}$ 
  is weakly compact.
\end{proposition}

\begin{remark}\label{rem:weaktopmetrizable}
  If $p\colon E \to L$ is a separable Hilbert bundle over a compact \emph{metric} space,
  then the weak topology on $E$ is metrizable. To see this, use the Stone--Weierstrasse 
  theorem for Hilbert bundles (see \cite[Theorem 4.2]{Gierz1982}) that states: If $M \subset \Gamma(E)$ 
  is a submodule that is dense in each fiber of $E$, then $M$ is dense in $\Gamma(E)$. From 
  this, it follows that for a separable Hilbert bundle over a compact metric space, 
  $\Gamma(E)$ is a separable Banach space. With this knowledge, writing down a compatible 
  metric for the weak topology on $E$ reduces to standard arguments.
\end{remark}

In \cite[Lemma 3.3]{Bjorklund2017} it is shown that a bounded harmonic function for
a stationary $(G,m)$-action on a probability space is necessarily almost everywhere $G$-invariant. 
Below, we adapt this proof to show that the assumption of boundedness is not required 
and that the same statement is also true for functions that are harmonic under the adjoint 
random walk which we expressed earlier via the operator $P_m^* \colon \uL^1(\uX) \to \uL^1(\uX)$
\begin{equation*}
  P_m^*f(x) = \int_G h_s(x)f(s^{-1}x)\dm(s)
\end{equation*}
where $h_s \defeq \frac{\ud s_*\mu}{\ud\mu}$.
We collect this invariance
statement here as we shall exploit the underlying idea in the proof of 
\cref{lem:bundle_inv} below to prove a similar statement in the case of Hilbert bundles.

\begin{lemma}\label{lem:newharmonicinv}
  Let $(G,m) \curvearrowright \uX$ be a stationary action on a probability space
  and $f\in \uL^1(\uX)$. If $P_mf = f$ or $P_m^*f = f$, then $f$ is $G$-invariant.
\end{lemma}
\begin{proof}
  Recall the general fact that if $T\colon \uL^1(\uX) \to \uL^1(\uX)$ is a 
  bi-Markov operator, then $\fix(T) = \{f\in \uL^1(\uX)\mid Tf = f\}$ is a 
  closed sublattice, i.e., $f\in \fix(T)$
  implies $|f| \in \fix(T)$. To see this, note that since $T$ is positive, 
  $|Tf| \leq T|f|$ and since $T$ preserves the integral,
  \begin{equation*}
    \int_X |f| \dmu 
    = \int_X |Tf| \dmu
    \leq \int_X T|f| \dmu
    = \int_X |f| \dmu.
  \end{equation*}
  Thus, $|f| = T|f|$ so that $\fix(T) \subset \uL^1(\uX)$ is indeed a sublattice.
  In particular, the sublattice $\fix_{\uL^\infty}(T) = \fix(T) \cap \uL^\infty(\uX)$
  is dense in $\fix(T)$ (use that $\1_\uX \in \fix(T)$ in order to employ a cut-off 
  approximation).
  Moreover, $\fix_{\uL^\infty}(T)$ is a closed sublattice in $\uL^\infty(\uX)$ and thus
  \cite[Theorem 7.23]{EFHN2015} shows that $\fix_{\uL^\infty}(T)$ is a $\uC^*$-subalgebra
  of $\uL^\infty(\uX)$.
  
  As we have seen, both $P_m$ and $P_m^*$ are bi-Markov operators on $\uL^1(\uX)$,
  so it is clear that to prove the claim, we may restrict ourselves to real-valued functions 
  $f\in \uL^\infty(\uX)$. The case $P_mf = f$ can be proved just as the case of countable $G$
  (see \cite[Lemma 3.3]{Bjorklund2017}), so we only prove the completely analogous case $P_m^*f = f$.
  In this case, 
  \begin{align*}
    \int_{\uX} \int_G& h_s(x)\left( f(x) - f(s^{-1}x) \right)^2 \dm(s) \dmu(x) \\
    &= \int_{\uX} \int_G h_s(x)\left( f(x)^2 - 2f(x)f(s^{-1}x) + f(s^{-1}x)^2 \right) \dm(s) \dmu(x) \\
    &= \int_G \int_{\uX} f(x)^2 \,\ud(s_*\mu)(x) \dm(s)
    -2 \int_{\uX} f(x)P_m^*f(x)\dmu(x) + \int_{\uX} \left(P_m^*(f^2)\right)(x) \dmu(x) \\
    &= \int_{\uX} f^2(x) - 2f^2(x) + f^2(x) \dmu(x).
  \end{align*}
  Here, we used stationarity of $\mu$ for the first term, $f\in\fix(P_m^*)$ for the second, and 
  for the third we exploited that $f^2 \in \fix(P_m^*)$. 
  Thus, $f$ is invariant under $m$-almost every $s\in G$. By \cref{thm:FEmagic}, 
  this 
  extends to all $s\in \supp(m)$ and since $\supp(m)$ generates a dense 
  subgroup of $G$, applyingg \cref{thm:FEmagic} again, we conclude that  
  $f$ is $G$-invariant.
\end{proof}

We are now ready to prove the key observation that will enable us to derive that 
the isometric extensions in the structure theorem for stationary actions are,
in fact, relatively measure-preserving. To that end, recall that if $p\colon E \to L$
is a Hilbert bundle, then $E_\uw$ denotes $E$ endowed with the weak topology.

\begin{lemma}\label{lem:bundle_inv}
  Let $\pi\colon(G, m) \curvearrowright (E_\uw,\xi)$ be a unitary and stationary action where  
  $p\colon E \to Z$ is a separable Hilbert bundle over a compact metrizable space $Z$.
  Let $\zeta \defeq p_*\xi$. Then $p\colon (E_\uw, \xi) \to (Z, \zeta)$ is a relatively
  measure-preserving extension of $(G,m)$-actions.
\end{lemma}

\begin{proof}
  First, we assume that $\xi$ is supported on the weakly compact unit cylinder 
  $B \defeq \{ e \in E \mid \|e\| \leq 1\}$.
  Since the Hilbert bundle is separable, \cref{thm:disintegration} combined with 
  \cref{rem:weaktopmetrizable} allows us to disintegrate 
  the measure $\xi$ over $Z$ as $(\xi_z)_{z\in Z}$. 
  We need to show that $\pi(s)_*\xi_{s^{-1}z} = \xi_z$ for all $s\in G$ and 
  $\zeta$-almost every $z\in Z$.
  
  To prove this invariance, it suffices to consider special functions: Since 
  $p\colon E \to Z$ is a separable bundle, there is  
  a countable fiberwise dense sequence $(u_j)_{j\in\N}$ in $\Gamma(E)$. 
  By the Stone--Weierstrass theorem, a dense subspace
  of $\uC(B)$ can be constructed if one takes the linear span of
  the constant function $\1_B$
  and the functions $\phi\colon B\to \C$ of the form
  \begin{equation*}
   \phi(v) = ( u_1(p(v)) | v ) \cdots ( u_k(p(v)) | v ).
  \end{equation*}
  It thus suffices to prove
  \begin{equation*}
    \int_{B_z} \phi(v) \,\mathrm{d}\pi(s)_*\xi_{s^{-1}z}(v) = \int_{B_z} \phi(v) \,\mathrm{d}\xi_{z}(v) \qquad \text{for } \zeta\text{-a.e. } z\in Z
  \end{equation*}
  for each of these generating functions $\phi(v) = ( u_1(p(v)) | v ) \cdots ( u_k(p(v)) | v )$. 
  Denote by $\mathcal{F}$ the countable collection of these functions and fix $\phi \in \mathcal{F}$.
  
  By means of the disintegration, we can 
  expand the expression
  \begin{align*}
   \int_{B_{z}} \phi(v) \,\mathrm{d}\pi(s)_*\xi_{s^{-1}z}(v)
   &= \int_{B_z} (u_1(p(v))|v) \cdots (u_k(p(v))| v) \,\mathrm{d}\pi(s)_*\xi_{s^{-1}z}(v) \\
   &=  \int_{B_{s^{-1}z}} ( u_1(z)| \pi(s)v ) \cdots ( u_k(z)| \pi(s)v ) \dxi_{s^{-1}z}(v)  \\
   &=  \Bigg( (u_1\otimes \cdots \otimes u_k)(z) \,\Bigg|\, \pi^{\otimes k}(s)\underbrace{\int_{B_{s^{-1}z}} v\otimes \cdots \otimes v \dxi_{s^{-1}z}(v)}_{\eqdef\sigma_{s^{-1}z}} \Bigg) \\
   &=  \left( (u_1\otimes \cdots \otimes u_k)(z) \mmid \pi^{\otimes k}(s)\sigma_{s^{-1}z} \right).
  \end{align*}
  Here, we defined 
  a section $\sigma \colon Z \to E^{\otimes k}$ by 
  $\sigma_z \defeq \int_{B_z} v\otimes \cdots \otimes v \dxi_z(v)$ using 
  that $p(v) = z$ for $v\in\supp(\xi_z)$ for $\zeta$-a.e.\ $z\in Z$
  since $\supp(\xi_z) \subset p^{-1}(z)$ for $\zeta$-a.e.\ $z\in Z$.
  Note that the vector-valued integral $\sigma_z = \int_{B_z} v\otimes \cdots \otimes v \dxi_z(v)$ 
  is understood in the weak sense, i.e., it is the unique vector $\sigma_z \in E_z^{\otimes k}$ with
  \begin{equation*}
    \forall e\in E_z^{\otimes k} \colon \quad (e|\sigma_z) = \int_{B_z} (e|v\otimes\cdots\otimes v) \dxi_z(v)
  \end{equation*}
  that exists by virtue of the Riesz--Fr\'{e}ch\'{e}t representation theorem. 
  Our goal is to show that $\zeta$-a.e.\ $\pi^{\otimes k}(s)\sigma_{s^{-1}z} = \sigma_z$ which will prove the equivariance 
  of the disintegration.
  Set $h_s^\zeta \defeq \frac{\ud s_*\zeta}{\ud \zeta}$ for every $s\in G$. As 
  observed in \cref{rem:stationary_dis}, by virtue of the uniqueness statement in \cref{thm:disintegration}
  the disintegration $(\xi_z)_{z\in Z}$ is stationary in the sense that
  \begin{equation}\label{eq:disint_stat}
    \int_G h_s^\zeta(z)\pi(s)_*(\xi_{s^{-1}z})\dm(s) = \xi_z \qquad \text{for } \zeta\text{-a.e.\ } z\in Z.
  \end{equation}
  Combine this with the definition of $\sigma$ to conclude that
  \begin{equation}\label{eq:sigma_id}
    \sigma_z = \int_G h_s^\zeta(z) \pi^{\otimes k}(s) \sigma_{s^{-1}z} \dm(s) \qquad \text{for } \zeta\text{-a.e.\ } z\in Z.
  \end{equation}
  We may therefore attempt to proceed as in the proof of \cref{lem:newharmonicinv} to
  prove that $\sigma$ is $G$-invariant under the action of $G$, i.e.,  $\sigma_z
  = (s.\sigma)_z = \pi^{\otimes k}(s)\sigma_{s^{-1}z}$. We first perform the computation
  mimicking \cref{lem:newharmonicinv} and justify 
  measurability of all occurring functions afterwards.
  \begin{align*}
   \int_\uZ&  \int_G h_s^\zeta(z)
   \left(\sigma_z - (s.\sigma)_z \mmid \sigma_z - (s.\sigma)_z\right) \dm(s) \dzeta(z) \\
   &= \int_\uZ \int_G h_s^\zeta(z) 
   \Big( (\sigma_z|\sigma_z) 
   - 2\Re(\sigma_z | \pi^{\otimes k}(s)\sigma_{s^{-1}z}) \\
   &\qquad\qquad\qquad\qquad + (\pi^{\otimes k}(s)\sigma_{s^{-1}z}|\pi^{\otimes k}(s)\sigma_{s^{-1}z})\Big) \dm(s) \dzeta(z) \\
   &= \int_G\int_\uZ \|\sigma_z\|^2 \,\ud(s_*\zeta)(z)\dm(s)
   -2 \Re\int_\uZ \left(\sigma_z \mmid \int_G h_s^\zeta(z) \pi^{\otimes k}(s)\sigma_{s^{-1}z}\dm(s)\right) \dzeta(z) \\
   & \qquad \qquad + \int_\uZ \int_G h_s^\zeta(z)
   \left(\sigma_{s^{-1}z}\mmid\sigma_{s^{-1}z}\right)  \dm(s)\dzeta(z) \\
   &= \int_\uZ \|\sigma_z\|^2 - 2\|\sigma_z\|^2 \dzeta(z) + \int_G\int_\uZ 
   \left(\sigma_{s^{-1}z}\mmid\sigma_{s^{-1}z}\right) \,\ud (s_*\zeta)(z)\dm(s) \\
   &= -\int_\uZ \|\sigma_z\|^2 \dzeta(z) + \int_\uZ \int_G
   \left(\sigma_{z}\mmid\sigma_{z}\right) \dm(s) \dzeta(z) \\
   &= 0.
  \end{align*}
  Note that it is here that we use the unitarity of $\pi$.
  Thus, if we can prove that the maps
  \begin{alignat*}{3}
   &Z\to \R, \quad z \mapsto ((s.\sigma)_z|(s.\sigma)_z)_{E_z} \qquad &(s \in \supp(m)), \\
   &Z\to \R, \quad z \mapsto ((s.\sigma)_z|\sigma_z)_{E_z} \qquad &(s \in \supp(m)).
  \end{alignat*}
  are measurable, we will have shown that $\sigma$ is invariant under $\supp(m)$ and hence $G$.
  To prove the measurability of $z \mapsto (\sigma_z|\sigma_z)$, observe that
  \begin{align*}
    (\sigma_z|\sigma_z)
    &= \left( \int_{B_z} v\otimes \cdots \otimes v \dxi_z(v) \mmid \int_{B_z} w\otimes \cdots \otimes w \dxi_z(w) \right) \\
    &= \int_{B_z\times B_z} (v\otimes \cdots \otimes v | w\otimes \cdots \otimes w) \dxi_z\otimes \xi_z(v,w).
  \end{align*}
  Since $(\xi_z\otimes \xi_z)_{z\in Z}$ is the disintegration of the relatively independent joining 
  $\xi\otimes_Z\xi$, it follows that the map $z \mapsto (\sigma_z|\sigma_z)$ is measurable. 
  Proving the measurability for the other scalar products requires merely notational adjustments.
  Thus, $s.\sigma = \sigma$ for $m$-almost every $s\in G$.
  
  We now return to our initial computation and thus obtain for $\zeta$-a.e.\ $z\in Z$ and 
  $m$-almost every $s\in G$
  \begin{align*}
   \int_{B_{z}} \phi(v) \,\mathrm{d}\pi(s)_*\xi_{s^{-1}z}(v) 
   &=  \left( u_1 \otimes \cdots \otimes u_k(z) \mmid (s.\sigma)_{z} \right) \\
   &=  \left( u_1 \otimes \cdots \otimes u_k(z) \mmid \sigma_z\right) \\
   &=  \int_{B_{z}} \phi(v) \dxi_z(v).
  \end{align*}
  Since $\phi \in \mathcal{F}$ was arbitrary and $\mathcal{F}$ countable, the disintegration 
  $(\xi_z)_{z\in Z}$ indeed satisfies the desired equivariance for $m$-almost every $s\in G$.
  Using the characterization of relative measure-preservation in terms of conditional expectations,
  see \cref{lem:rmp_char} and the continuity theorem \cref{thm:FEmagic}, this extends to 
  all $s\in \supp(m)$ and the closed subgroup generated by $\supp(m)$ which proves $G$-equivariance.
  
  Finally, assume $\supp(\xi)$ is not necessarily contained in $B$. Since $G$ acts unitarily, the
  weakly compact set 
  $B_n \defeq \{e \in E \mid \|e\| \leq n\}$ is $G$-invariant for every $n\in\N$ and the same arguments
  as above show that any disintegration of $\xi_n \defeq \frac{1}{\xi(B_n)} \xi|_{B_n}$ is 
  $G$-equivariant almost everywhere provided that 
  $\xi(B_n) > 0$. By the disintegration theorem \cref{thm:disintegration}, $(\xi_z|_{B_n})_{z\in Z}$ 
  is a disintegration of $\xi|_{B_n}$ and as such it is $G$-equivariant almost everywhere. Now, for any
  weakly measurable set $A \subset E_\uw$
  \begin{equation*}
    \xi_z(A) 
    = \lim_{n\to \infty} \xi_z|_{B_n}(A) 
    = \lim_{n\to \infty} \frac{1}{\xi(B_n)}\xi_z|_{B_n}(A).
  \end{equation*}
  Thus, the disintegration $(\xi_z)_{z\in Z}$ is also $G$-equivariant almost everywhere.
\end{proof}

\begin{remark}
  U.\ Bader pointed out to the author that the weak and norm Borel $\sigma$-algebras coincide
  for a separable Hilbert space. The same is likely true for separable Hilbert bundles
  over compact metrizable spaces but we do not flesh out the details since we do not require 
  this equality here. Knowing this, $E_\uw$ ($E$ endowed with the weak topology)
  may be replaced by $E$ in \cref{lem:bundle_inv}.
\end{remark}

\textbf{Proof of the dichotomy.} Now that we have proven the generalization \cref{lem:bundle_inv}
of the invariance lemma \cref{lem:ball_inv}, we are ready to prove the dichotomy \cref{thm:relstatdichotomy}. 
Recall from \cref{lem:moduleprojection} that for a subset $M \subset \uL^2(\uX|\uZ)$ we denote 
the $\uL^2$-closure of $M$ within $\uL^2(\uX|\uZ)$ by $\tilde{M}$.

\begin{proof}[Proof of \cref{thm:relstatdichotomy}]
  First, assume the extension $\pi\colon (X,\mu) \to (Z,\zeta)$ 
  is weakly mixing (and in particular, ergodic). 
  Suppose, for the sake of contradiction, that
  \begin{equation*}
   \xymatrix{
    (X, \mu) \ar[r]^{p} & (\tilde{Z}, \tilde{\zeta}) \ar[r]^q & (Z, \zeta)
   }
  \end{equation*}
  is a nontrivial relatively measure-preserving
  isometric intermediate extension of $(Z, \zeta)$. Then there exists a finitely-generated 
  $\uL^\infty(\uZ)$-submodule $M \subset \uL^2(\tilde{\uZ}|\uZ)$ such that $M$ is not contained 
  in $\uL^\infty(\uZ)$. Without loss of generality, assume that $M = \tilde{M}$ and 
  let $P_M\colon \uL^2(\tilde{\uZ}|\uZ) \to M$ be the orthogonal projection 
  onto $M$. By \cref{lem:moduleprojection}, $P_M\in \HS(\uL^2(\tilde{\uZ}|\uZ))$. Since relative 
  measure-preservation is equivalent to equivariance of conditional expectations (see \cref{lem:rmp_char}),
  the decomposition $\uL^2(\tilde{\uZ}|\uZ) = M \oplus M^\perp$ is readily verified to be invariant. Thus, 
  $P_M$ is $G$-equivariant. Via the isomorphism $\HS(\uL^2(\tilde{\uZ}|\uZ)) \cong \uL^2(\tilde{\uZ}\times_\uZ\tilde{\uZ}|\uZ)$ 
  from \cref{thm:KHiso}, we may write $P_M = I_k$ for a uniquely determined $k \in \uL^2(\tilde{\uZ}\times_\uZ\tilde{\uZ}|\uZ)$.
  As observed above, the $G$-equivariance of $P_M$ is equivalent to invariance of $k$. Since $P_M$ does not map
  into $\uL^\infty(\uZ)$, $k\neq \E_\uZ(k)\cdot \1_{\tilde{\uZ} \times_{\uZ} \tilde{\uZ}}$ which shows that the extension 
  $\tilde{\uZ} \times_\uZ \tilde{\uZ} \to \uZ$ is not ergodic. By extension, the 
  composite extension 
  \begin{equation*}
   (X\times_Z\tilde{Z}, \mu\otimes_Z\tilde{\zeta}) \to (\tilde{Z}\times_Z\tilde{Z}, \tilde{\zeta}\otimes_Z\tilde{\zeta}) \to (Z,\zeta)
  \end{equation*}
  is not ergodic either, which proves that the extension $\uX \to \uZ$ is not weakly mixing.

  Conversely, suppose $\pi \colon (X, \mu) \to (Z, \zeta)$ is not weakly mixing. Then there 
  exists an ergodic relatively measure-preserving extension 
  $q\colon (Y, \nu) \to (Z, \zeta)$ such that the relatively independent joining 
  $(X\times_Z Y, \mu\otimes_Z\nu) \to (Z,\zeta)$ is not an ergodic extension. This means
  that there is a nonconstant $G$-invariant function 
  $k \in \uL^\infty(X\times_Z Y, \mu\otimes_Z \nu)$ such that $k$ is not of 
  the form $f \cdot (\1_\uX \otimes_\uZ \1_\uY)$ for any $f\in \uL^\infty(\uZ)$.
  By \cref{lem:disint_model} and \cref{factor-algebra-corr}, we 
  may choose standard probability models and thus assume that $X$, $Y$, and $Z$ are compact metric 
  spaces, that $q\colon (Y,\nu) \to (Z,\zeta)$ is continuous, and that the disintegration of $\nu$ with respect to $q$ and $\zeta$ is 
  weak$^*$-continuous. By \cref{ex:canonical_disintegration}, we 
  thus obtain a natural continuous Hilbert bundle with total space $E = \bigcup_{z\in Z} \uL^2(Y_z, \nu_z)$.
  By \cref{prop:weaklycompact} and \cref{rem:weaktopmetrizable}, 
  $B_E \defeq \{ e \in E \mid \|e\| \leq 1 \}$ is a compact metric space with respect to the weak 
  topology on $E$. We claim that, after identifying $k$ with one of its representatives,
  \begin{equation*}
    p_k \colon X \to B_E, \quad x \mapsto k(x, \cdot)
  \end{equation*}
  is measurable with respect to the weak Borel $\sigma$-algebra on $B_E$. To see this, 
  note that $C(X\times_Z Y)$ is dense in $\uL^\infty(X\times_Z Y, \mu\otimes_Z \nu)$ with 
  respect to the $\uL^1$-norm so that we can find a sequence $(k_n)_{n\in\N}$ in $C(X\times_Z Y)$ 
  that converges to $k$ in $\uL^1(X\times_Z Y, \mu\otimes_Z \nu)$. Using a cut-off argument 
  and passing to a subsequence, we may assume that $|k_n| \leq 1$ and that $k_n(x,y) \to k(x,y)$ 
  for $\mu\otimes_Z \nu$-a.e.\ $(x, y) \in X\times_Z Y$. 
  Let $F \subset X \times_Z Y$ be a set of full measure where this convergence 
  holds. Setting $k_n$ and $k$ to zero on $F^c$, we may assume that $k_n(x,y) \to k(x,y)$ for
  every $(x,y) \in X \times_Z Y$. Therefore, by dominated convergence, $p_{k_n}(x) \to p_k(x)$
  for every $x\in X$. Since $p_{k_n}$ is continuous with respect to the weak topology on $B_E$
  and the pointwise limit of a sequence of measurable functions taking values in a metric space
  is again measurable, $p_k$ is measurable.
  
  Since for every $s\in G$, $k$ is almost everywhere $s$-equivariant, for every $s\in G$
  the map 
  $p_k\colon X \to B_E$ is $s$-equivariant almost everywhere. Thus, 
  $p_k$ yields a factor map $p_k \colon (X, \mu) \to (B_E, (p_k)_*\mu)$ of stationary 
  $(G,m)$-actions. Since the extension
  $q\colon (Y, \nu) \to (Z,\zeta)$ is relatively measure-preserving, the induced $G$-action 
  on $E$ is unitary. This allows for the (only) application of the stationarity assumption: We 
  can now apply \cref{lem:bundle_inv} to see that also the 
  extension $(B_E, (p_k)_*\mu) \to (Z,\zeta)$ is relatively measure-preserving.
  Our goal is to show that $(B_E\times_Z B_E, (p_k)_*\mu \otimes_Z (p_k)_*\mu) \to (Z,\zeta)$ 
  is not an ergodic extension and to derive from this the existence of an isometric 
  intermediate extension by means of the Kronecker dichotomy \cref{thm:nonsingularrelkrodichotomy}
  for nonsingular systems.
  Showing that $(B_E\times_Z B_E, (p_k)_*\mu \otimes_Z (p_k)_*\mu) \to (Z,\zeta)$ is not an ergodic
  extension is equivalent to finding a $G$-invariant function 
  $\Phi \in \uL^\infty(B_E\times_Z B_E, (p_k)_*\mu \otimes_Z (p_k)_*\mu)$ that does not lie in 
  $\uL^\infty(Z,\zeta)$, i.e., is not of the form $\Phi = \E_\uZ(\Phi)\cdot \1_{B_E\times_Z B_E}$.
  
  To find such a $\Phi$,
  let $K\colon \uL^2(X|Z) \to \uL^2(Y|Z)$ be the Hilbert--Schmidt homomorphism $I_k$ induced by 
  $k$ and consider $K^*K \colon \uL^2(X|Z) \to \uL^2(X|Z)$. An explicit computation shows that 
  the kernel of $K^*K$ is given by the convolution of the kernel of $K$ with itself,
  i.e., the $G$-invariant function
  \begin{equation*}
    \tilde{\Phi}\colon X\times_Z X \to \R, \quad 
    (x,x') \mapsto \int_{Y_{q(x)}} k(x, y) \overline{k(x', y)}\dnu_{q(x)} (y).
  \end{equation*}
  Set $\Phi\colon B_E \times_Z B_E \to \C$, $(e, f) \mapsto (e|f)$; then $\tilde{\Phi}$
  factorizes as $\tilde{\Phi} = \Phi \circ (p_k\times_Z p_k)$. Now, 
  \begin{equation*}
   \Phi = \E_\uZ(\Phi) \cdot \1_{B_E\times_Z B_E} \quad \iff \quad
   \tilde{\Phi} = \E_\uZ(\tilde{\Phi}) \cdot \1_{X\times_Z X}.
  \end{equation*}
  Thus, to conclude the proof that the extension
  $(B_E\times_Z B_E, (p_k)_*\mu \otimes_Z (p_k)_*\mu) \to (Z,\zeta)$ is not ergodic,
  we shall prove that $\tilde{\Phi} = \E_\uZ(\tilde{\Phi}) \cdot \1_{X\times_Z X}$ is impossible.
  
  Suppose, for the sake of contradiction, that 
  $\tilde{\Phi} = \E_\uZ(\tilde{\Phi}) \cdot \1_{X\times_Z X}$. Via 
  the isomorphism $\uL^2(X\times_Z X) \cong \HS(\uL^2(\uX|\uZ), \uL^2(\uX|\uZ))$, 
  this means that $K^*K = \E_\uZ(\tilde{\Phi})\cdot\E_\uZ$. 
  We thus see that $K^*K = K^*K\E_\uZ$ and since $K^*$ is 
  injective on the range of $K$, we can cancel $K^*$ and 
  conclude that $K = K\E_\uZ$. Comparing the kernels of these
  two Hilbert--Schmidt homomorphisms, one readily verifies that 
  this implies that
  $k = \1_\uX \otimes_\uZ f$ where 
  \begin{equation*}
    f\colon Y \to \C, \qquad y \mapsto \int_{X_{q(y)}} k(x,y)\dmu_{q(y)}(x).
  \end{equation*}
  Since $k$ is $G$-invariant, $f\in\uL^\infty(\uY)$ must be $G$-invariant.
  However, the extension $q\colon (Y,\nu) \to (Z,\zeta)$ was assumed to be
  ergodic, so $f = \E_\uZ(f)\cdot \1_\uY$. Thus, 
  $k = \E_\uZ(f)\1_\uX\otimes_\uZ \1_\uY$, i.e., $k$ is a fixed function 
  coming from the factor $(Z,\zeta)$. However, we assumed $k$ not to be 
  of this form, a contradiction. Thus, our assumption that 
  $\tilde{\Phi} = \E_\uZ(\tilde{\Phi}) \cdot 1_{X\times_ZX}$ must have 
  been false.

  To summarize: we have found a $G$-invariant function 
  $\Phi\in \uL^\infty( B_E\times_Z B_E)$ that is not of the form
  $\Phi = \E_\uZ(\Phi) \cdot \1_{B_E\times_Z B_E}$, so the extension 
  $(B_E\times_Z B_E, (p_k)_*\mu \otimes_Z (p_k)_*\mu) \to (Z,\zeta)$ 
  is not ergodic. The Furstenberg--Zimmer dichotomy for nonsingular
  actions \cref{thm:nonsingularrelkrodichotomy} now shows the existence of the desired 
  relatively measure-preserving isometric intermediate extension.
\end{proof}

As a corollary, we show that $m$-proximal extensions (defined in \cref{rem:condmeas})
are always weakly mixing.

\begin{corollary}\label{lem:proximalweakmixing}
  Let $\pi\colon \uX \to \uZ$ be an $m$-proximal extension of stationary $(G,m)$-actions. 
  Then $\pi$ is weakly mixing.
\end{corollary}
\begin{proof}
  By the dichotomy \cref{thm:relstatdichotomy}, it suffices to show that a proximal and 
  relatively measure-preserving extension must be an isomorphism. So suppose $\pi\colon (X,\mu) \to (Y,\nu)$ 
  is an extension with these two properties and assume without loss of generality that $X$ and $Y$ are 
  compact metrizable spaces. Since $\pi$ is relatively measure-preserving,
  \cref{lem:condmeasures} shows that 
  if $\mu = \int_Y \mu_y \dnu(y)$ is the disintegration of $\mu$ w.r.t.\ $\nu$, then
  almost surely
  \begin{equation*}
    \mu_\omega = \int_Y \mu_y \dnu_\omega(y).
  \end{equation*}
  Since $\pi$ is proximal, $\mu_y$ is a Dirac measure $\nu_\omega$-almost everywhere for almost every $\omega$.
  Note that the set $\{y\in Y \mid \mu_y \text{ is Dirac}\}$ is measurable since the set of Dirac measures
  in the compact space $Y$ is compact. Now, by the barycenter equation
  \begin{equation*}
    \nu(\{y\in Y \mid \mu_y \text{ is Dirac}\}) = \int_\Omega \nu_\omega(\{y\in Y \mid \mu_y \text{ is Dirac}\})\dP(\omega) = 1.
  \end{equation*}
  Thus, $\pi$ is an isomorphism by \cref{thm:disintegration}.
\end{proof}

\section{Structure theorem for stationary actions}\label{sec:stationaryFZ}

The main structure theorem is now an easy consequence of the stationary Kronecker dichotomy.

\begin{theorem}\label{thm:stationarymainthm}
  Let $\pi\colon \uX \to \uZ$ be an extension of stationary $(G,m)$-actions.
  Then there are 
  a weakly mixing extension $\alpha\colon \uX \to \uX_\ud$ and a distal and relatively 
  measure-preserving extension $\beta\colon \uX_\ud \to \uZ$ 
  such that the diagram
  \begin{equation*}
   \xymatrix{
    \uX \ar[rd]_-\alpha \ar[rr]^-\pi && \uZ \\
    & \uX_\ud \ar[ru]_-\beta &
   }
  \end{equation*}
  commutes.
\end{theorem}

\begin{corollary}
  Let $(G,m) \curvearrowright \uX$ be a stationary $(G,m)$-action.
  Then it is a weakly mixing extension of a distal and measure-preserving action
  $G \curvearrowright \uX_\ud$.
\end{corollary}

\begin{remark}
  It is important to remember that if $\pi\colon (X,\mu) \to (Y,\nu)$ is a factor map from 
  a nonsingular to a measure-preserving $G$-action, then $G\curvearrowright(X,\mu)$ need
  not be measure-preserving, even though $\pi$ is by definition a measure-preserving map.
  See \cref{rem:mRIM} for a counterexample.
\end{remark}

Similar to the Furstenberg--Zimmer structure theorem and its nonsingular version \cref{thm:nonsingularmainthm}, 
we will obtain the required tower of isometric extensions 
for stationary actions by iteratively applying the dichotomy proven in \cref{thm:relstatdichotomy}. If a 
stationary action always had a maximal measure-preserving factor, one could circumvent this and directly 
use \cref{thm:nonsingularmainthm} to avoid repeating the same argument´.
However, a system without a maximal measure-preserving factor can be found in \cite[Example 11]{FurstenbergGlasner2010}.

\begin{proof}[Proof of \cref{thm:stationarymainthm}]
  We proceed by transfinite recursion. For the induction start, set $\uX_0 = \uZ$.
  Now suppose $\mu$ is an ordinal and we have constructed a projective system 
  $((\uX_\eta)_{\eta < \mu}, (\pi_\eta^\sigma)_{\eta\leq \sigma < \mu})$ of isometric and
  measure-preserving
  extensions.
  \begin{itemize}
   \item If $\mu$ is a limit ordinal, set $\uX_\mu = \lim_{\from} \uX_\eta$.
   If $\uX \to \uX_\mu$ is weakly mixing, we are done. Otherwise, repeat the induction step.
   \item If $\mu$ is a successor cardinal and the extension $\uX \to \uX_{\mu-1}$ is 
   weakly mixing, we are done. Otherwise, by \cref{thm:relstatdichotomy} 
   there is a relatively measure-preserving isometric intermediate extension 
   $\uX \to \uX_\mu \to \uX_{\mu-1}$ and we repeat the induction step.
  \end{itemize}
  This recursion terminates after countably many steps, for otherwise $\uL^2(\uX)$ 
  would contain an uncountable orthonormal subset. This proves the desired decomposition.
\end{proof}

\begin{remark}
  Furstenberg's multiple recurrence theorem is trivially true for 
  weakly mixing transformations since they are weakly mixing of all orders, 
  see \cite[Theorem 9.31]{EFHN2015}. 
  Since weakly mixing is a generic property (see \cite[Theorem 2]{Halmos1944} or \cite[Theorem 3]{Alpern1976}), 
  this shows that for \enquote{most} measure-preserving 
  transformations, Furstenberg's structure theorem and the multiple recurrence theorem are trivially true.
  However, the Furstenberg structure theorem also holds for arbitrary group actions and in the general setting
  it is not true that generic actions are weakly mixing, i.e., have no Kronecker factor. A trivial example 
  are compact groups for which every action is purely Kronecker. More generally, if a second-countable locally
  compact group $G$ has property (T), then the measure-preserving ergodic (hence, weakly mixing) actions are not only 
  not generic but nowhere
  dense, and if $G$ does not have property (T), then the measure-preserving weakly mixing actions are generic, i.e., a 
  dense $G_\delta$-set (see \cite[Theorem 4.2]{KerrPichot2008}).
  For stationary actions of groups with property (T), the situation is similar, though there even are examples of 
  groups without property (T) for which a generic stationary action is not ergodic, see \cite{BowenHartmanTamuz2017}.
\end{remark}

\begin{example}\label{ex:mainexample}
  Pick some $\gamma \in (0,1)$ and let $\alpha, \beta \in [0,2\pi)$ be angles
  linearly independent over $\Q$. Set
  \begin{equation*}
    A = \begin{pmatrix}
         \gamma & 1 & 1 \\
         & \cos(\alpha) & -\sin(\alpha) \\
         & \sin(\alpha) & \cos(\alpha)
        \end{pmatrix}, \qquad
    B = \begin{pmatrix}
         \gamma & 1 & 1 \\
         & \cos(\beta) & -\sin(\beta) \\
         & \sin(\beta) & \cos(\beta)
        \end{pmatrix},
  \end{equation*}
  let $\Gamma_m$ be the subgroup of $\GL_3(\R)$ generated by $A$ and $B$, and set 
  $m \defeq \frac{1}{3}\delta_A + \frac{1}{3}\delta_B + \frac{1}{6}\delta_{A^{-1}} + \frac{1}{6}\delta_{B^{-1}}$.
  What are the stationary measures for the action of $(\Gamma_m,m)$ on $\P(\R^3)$? Certainly, 
  $[1:0:0]$ is a fixed point and so the corresponding 
  Dirac mass is an invariant measure for $\Gamma_m$. To classify the stationary probability measures
  on the $\Gamma_m$-invariant
  complement $\P(\R^3) \setminus [1:0:0]$, we consider the two-fold $\Gamma_m$-equivariant cover 
  $\R\times\uS^1 \to \P(\R^3) \setminus [1:0:0]$. Being a finite cover, this map is proper and hence 
  every stationary measure on $\P(\R^3) \setminus [1:0:0]$ lifts to a stationary measure on $\R\times \uS^1$
  which justifies passing to the cover.
  The action of $A$ on $\R\times \uS^1$ takes the form
  \begin{equation*}
   A\begin{pmatrix}
      r \\
      \cos(\theta) \\
      \sin(\theta)
    \end{pmatrix}
  = \begin{pmatrix}
     \gamma r + \cos(\theta) + \sin(\theta) \\
     \cos(\theta + \alpha) \\
     \sin(\theta + \alpha)
    \end{pmatrix}
  \end{equation*}
  Similar expressions hold for $B$, $A^{-1}$, and $B^{-1}$. Thus, the action of a 
  general $g\in \Gamma_m$ on $\R\times \uS^1$ is of the form
  \begin{equation*}
    \rho_g \colon \R\times \uS^1 \to \R^1\times\uS^1, \qquad
   \rho_g(r, p) = (a_gr + b_g(p), R_gp)
  \end{equation*}
  where $R_g$ is a rotation on $\uS^1$ (which we usually abbreviate by just $g$), 
  $a\colon \Gamma_m \to \R^\times$ is a group homomorphism, and $b_g\colon \uS^1 \to \R$ 
  is a continuous function for each $g\in\Gamma_m$.
  
  Suppose there is an $m$-stationary probability measure $\mu$ on the 
  $\Gamma_m$-invariant set $\R\times \uS^1$.
  We will derive a concrete description of $\mu$ which, a posteriori, will show that such a measure
  $\mu$ exists and is unique.
  Let $\pi\colon \R\times\uS^1 \to \uS^1$ denote the natural projection.
  Since $\alpha$ and $\beta$ were assumed to be algebraically independent, the action of $\Gamma_m$ on $\uS^1$
  is uniquely ergodic, so the push-forward $\nu \defeq \pi_*\mu$ 
  of $\mu$ is the Lebesgue measure on $\uS^1$.
  
  Disintegrate $\mu$ as
  \begin{equation*}
    \mu = \int_{\uS^1} \mu_p\otimes\delta_p \dnu(p).
  \end{equation*}
  Then due to the invariance of $\nu$, we can use the uniqueness of disintegrations 
  to conclude that for $g\in\Gamma_m$ the disintegration of $g_*\mu$ is given by
  \begin{equation*}
    (g_*\mu)_p = a_g\mu_{g^{-1}p}+ b_g(g^{-1}p) \qquad \text{for }\nu\text{-a.e. } p\in \uS^1
  \end{equation*}
  where the expression $a_g \mu_{g^{-1}p} + b_g(g^{-1}p)$ is to be understood as a dilation followed 
  by a translation.
  Iterating this, we see that for $\omega\in\Omega$ and $n\in\N$
  \begin{align*}
    ((\omega_1\cdots \omega_n)_*\mu)_p 
    = a_{\omega_1}&\cdots a_{\omega_{n}} \mu_{(\omega_1\cdots \omega_n)^{-1}p} \notag \\
    &+ \sum_{k=1}^n a_{\omega_1} \cdots a_{\omega_{k-1}}b_{\omega_k}\left((\omega_1\cdots \omega_k)^{-1}p\right).
  \end{align*}
  Since the affine maps contract on average w.r.t.\ $m$, i.e.,
  \begin{equation*}
    \sum_{g \in \Gamma_m} m(g)\log(a_g) < 1,
  \end{equation*}
  the strong law of large numbers yields that $a_{\omega_1}\cdots a_{\omega_n}\to 0$ for almost every $\omega \in \Omega$.
  As $b_g(p)$ is bounded uniformly 
  in $p \in \uS^1$ and $g\in\supp(m)$, the series 
  \begin{equation*}
   s_\omega(p) \defeq \sum_{k=1}^\infty a_{\omega_1} \cdots a_{\omega_{k-1}}b_{\omega_k}((\omega_1\cdots \omega_k)^{-1}p)
  \end{equation*}
  converges absolutely almost surely. We want to show that
  $a_{\omega_1}\cdots a_{\omega_{n}} \mu_{(\omega_1\cdots \omega_{n})^{-1}p}$ almost surely converges to $\delta_0$ 
  in the weak* topology. However, to determine the conditional measures $\mu_\omega$, it suffices to prove
  this along a subsequence.
  
  Let $\epsilon > 0$. Since for almost every $p\in\uS^1$ one has $\lim_{M\to\infty} \mu_p([-M, M]) = 1$,
  there exists an $M > 0$ such that 
  \begin{equation*}
    \nu\Big(\underbrace{\left\{ p \in \uS^1 \mmid \mu_p([-M, M]) \geq 1-\epsilon  \right\}}_{\eqdef A_\epsilon} \Big) \geq 1 - \epsilon.
  \end{equation*}
  Since the action of $(\Gamma_m,m)$ on $\uS^1$ admits the unique stationary measure $\nu$, Breiman's law shows
  that for almost every $\omega\in\Omega$ and every $p\in\uS^1$, 
  the sequence $(\omega_1\cdots \omega_n p)_n$ equidistributes towards $\nu$. In particular, we can find a 
  subsequence $(n_k)_k$ such that $R_{\omega_1}\cdots R_{\omega_{n_k}} \to \id$ and thus 
  $R_{(\omega_1\cdots \omega_{n_k})^{-1}} \to \id$. Hence, 
  \begin{equation*}
    \1_{A_\epsilon} \circ R_{(\omega_1\cdots \omega_{n_k})^{-1}} \xrightarrow[k\to\infty]{\|\cdot\|_{\uL^1}} \1_{A_\epsilon}
  \end{equation*}
  and by replacing $(n_k)_k$ with a subsequence, we may assume that this convergence holds almost surely. 
  Thus, for almost every $p\in A_\epsilon$, eventually $(\omega_1\cdots \omega_{n_k})^{-1}p$ stays in $A_\epsilon$.

  For $\delta > 0$, let $f_\delta\in\uC_\uc(\R)$ with $f_\delta|_{[-\delta, \delta]} \equiv 1$ and such that $f_\delta$ 
  vanishes outside $[-2\delta, 2\delta]$.
  Then since $a_{\omega_1}\cdots a_{\omega_{n}} \to 0$, we see that along the subsequence $(n_k)_k$, for 
  almost every $p\in A_\epsilon$
  \begin{align*}
    \liminf_{k\to\infty}\left\langle f_\delta, a_{\omega_1}\cdots a_{\omega_{n_k}} \mu_{(\omega_1\cdots \omega_{n_k})^{-1}p}\right\rangle 
    &\geq \liminf_{k\to\infty}\mu_{(\omega_1\cdots \omega_{n_k})^{-1}p}([-M, M]) \\
    &\geq 1-\epsilon.
  \end{align*}
  Sending $\epsilon$ to $0$ and employing a diagonalization argument, we can find a set $A$ of full 
  measure and replace $(n_k)_k$ with a subsequence such that for every $p\in A$
    \begin{align*}
    \liminf_{k\to\infty}\left\langle f_\delta, a_{\omega_1}\cdots a_{\omega_{n_k}} \mu_{(\omega_1\cdots \omega_{n_k})^{-1}p}\right\rangle 
    &=1.
  \end{align*}
  Since $\delta > 0$ was chosen arbitrarily, we see that 
  \begin{equation*}
    \lim_{k\to\infty} a_{\omega_1}\cdots a_{\omega_{n_k}} \mu_{(\omega_1\cdots \omega_{n_k})^{-1}p} = \delta_0.
  \end{equation*}
    
  To summarize: we have shown that along the subsequence $(n_k)_k$, almost surely,  
  $(\omega_1\cdots \omega_{n_k})_*\mu$ does not just converge 
  in the weak* topology but also almost all of it's fiber measures with respect to $\nu$ converge in the weak* topology.
  Necessarily, they converge to the fiber measures of $\mu_\omega$:
  \begin{equation*}
    (\mu_\omega)_p =  \delta_{s_\omega(p)}\otimes \delta_p = \delta_{(s_\omega(p), p)}.
  \end{equation*}
  Thus, if an $m$-stationary measure $\mu$ exists on $\R\times\uS^1$, then
  \begin{equation*}
    \mu = \int_{\Omega} \int_{\uS^1} \delta_{(s_\omega(p), p)} \dnu(p) \dP(\omega)
  \end{equation*}
  which proves that there is at most one such measure. To prove that the measure defined by this integral
  is indeed $m$-stationary, for $g\in\Gamma_m$
  and $\omega \in \Omega$, denote by $g.\omega \in \Omega$ the concatenation $(g, \omega_1, \omega_2, \dots)$.
  A quick computation shows that
  \begin{equation*}
    s_{g.\omega}(p) = a_gs_\omega(g^{-1}p) + b_g(g^{-1}p).
  \end{equation*}
  and therefore
  \begin{align*}
    \sum_{g\in\supp(m)} m(g)g_*\mu 
    &= \int_{\Omega} \int_{\uS^1} \sum_{g\in\supp(m)} m(g) \delta_{(a_gs_\omega(p) + b_g(p), gp)} \dnu(p) \dP(\omega) \\
    &= \int_{\Omega} \int_{\uS^1} \sum_{g\in\supp(m)} m(g) \delta_{(a_gs_\omega(g^{-1}p) + b_g(g^{-1}p), p)} \dnu(p) \dP(\omega) \\
    &= \int_{\Omega} \int_{\uS^1} \sum_{g\in\supp(m)} m(g) \delta_{(s_{g.\omega}(p), p)} \dnu(p) \dP(\omega)\\
    &= \int_{\Omega} \int_{\uS^1} \delta_{(s_{\omega}(p), p)} \dnu(p) \,\ud (m \otimes \P)(\omega)\\
    &= \mu.
  \end{align*}
  Now, let $\overline{\mu}$ denote the pushforward of $\mu$ under the map $\R\times \uS^1 \to \P(\R^3) \setminus [1:0:0]$
  and $\overline{\nu}$ denote the pushforward of $\nu$ under the map $\uS^1 \to \P(\R^2)$.
  We have already seen that for almost every $\omega\in\Omega$, the conditional measure $\mu_\omega$ admits the disintegrations
  \begin{equation*}
    \mu_\omega = \int_{\uS^1} \delta_{(s_\omega(p), p)} \dnu(p).
  \end{equation*}
  To translate this into an expression for $\overline{\mu}_\omega$, observe that the 
  identity
  \begin{equation*}
    s_\omega(p) = \lim_{n\to\infty} \pi_\R\left(\omega_1\cdots \omega_n \big((0,\omega_1\cdots\omega_n)^{-1}p\big)\right)
  \end{equation*}
  shows that $s_\omega(-p) = -s_\omega(p)$, so $s_\omega\colon \uS^1 \to \R$ can 
  be regarded as a map $\P(\R^2) \to \R$. Since conditional measures are natural,
  a quick computation shows that
  \begin{equation*}
    \overline{\mu}_\omega = \int_{\P(\R^2)} \delta_{[s_\omega(p): p]} \,\ud\overline{\nu}([p]).
  \end{equation*}
  By \cref{thm:disintegration}, this shows that the factor map $(\P(\R^3), \overline{\mu}_\omega) \to (\P(\R^2),\overline{\nu})$ 
  is an isomorphism. Thus, the extension $(\P(\R^3), \overline{\mu}) \to (\P(\R^2), \overline{\nu})$ is proximal and hence 
  weakly mixing by \cref{lem:proximalweakmixing}.
\end{example}

\def\appendixname{Appendix}

\appendix \section{Hilbert--Schmidt homomorphisms on conditional \texorpdfstring{$\uL^2$}{L\textsuperscript{2}}-spaces} \label{sec:appendix}

Given probability spaces $\uX$ and $\uY$, one has an isomorphism $\uL^2(\uX\times\uY) \cong \HS(\uL^2(\uX), \uL^2(\uY))$
of Hilbert spaces given by mapping a function $k\in \uL^2(\uX\times\uY)$ to the corresponding Hilbert--Schmidt 
integral operator $I_k \colon \uL^2(\uX) \to \uL^2(\uY)$. The goal of this appendix is to explain the 
proof of a more general version of this theorem for conditional $\uL^2$-spaces.

\begin{reminder}\label{rem:HSbasics}
  We begin by recalling the relevant basic facts on Hilbert--Schmidt operators. Many 
  sources only consider Hilbert--Schmidt operators on a single Hilbert space,
  but we will need to understand operators between different spaces below, so we choose the 
  general setting.
  Let $\calH$ and $\calK$ be Hilbert spaces. A bounded operator $T \in \mathscr{L}(\calH, \calK)$
  is called a \textbf{Hilbert--Schmidt operator} if for some/any orthonormal basis $(e_n)_{n\in \N}$
  of $\calH$
  \begin{equation*}
    \sum_{n\in \N} \|Te_n\|^2 < +\infty.
  \end{equation*}
  This expression is independent of the chosen orthonormal basis and so we can define the 
  \textbf{Hilbert--Schmidt norm} of $T$ as
  \begin{equation*}
    \|T\|_{\mathrm{HS}} \defeq \left(\sum_{n\in \N} \|Te_n\|^2\right)^\frac{1}{2}
  \end{equation*}
  with respect to some orthonormal basis. Denote the collection of all Hilbert--Schmidt operators
  from $\calH$ to $\calK$ by $\HS(\calH, \calK)$ and recall that $\HS(\calH, \calK)$ is again a 
  Hilbert space with the inner product $(T|S) = \sum_{n\in \N} (Te_n|Se_n)$.
  
  If $\uX$ and $\uY$ are probability spaces, then
  every function $k\in \uL^2(\uX\times\uY)$ gives rise to a Hilbert--Schmidt operator 
  \begin{equation*}
    I_k\colon \uL^2(\uX)\to \uL^2(\uY), \quad I_kf(y) \defeq \int_\uX k(x,y)f(x)\dmu_X(x).
  \end{equation*}
  This assignment $k \mapsto I_k$ turns out to be a unitary isomorphism 
  $\uL^2(\uX\times\uY) \cong \HS(\uL^2(\uX), \uL^2(\uY))$. We refer the reader to 
  \cite[Section 3.3.1]{Sunder2016} for proofs of these statements. We also 
  prove a more general version of this for \emph{conditional $\uL^2$-spaces} below, 
  see \cref{thm:KHiso}.
\end{reminder}

As discussed in the main text, in the study of extensions $\pi\colon \uX \to \uZ$ 
of dynamical systems, the Hilbert 
space $\uL^2(\uX)$ is often replaced by the conditional $\uL^2$-space $\uL^2(\uX|\uZ)$. $\uL^2(\uX|\uZ)$ is a 
\textbf{Hilbert module} over $\uL^\infty(\uX)$, i.e., it is an $\uL^\infty(\uZ)$-module via the action 
\begin{equation*}
  \uL^\infty(\uZ)\times\uL^2(\uX|\uZ) \to \uL^2(\uX|\uZ), \quad (f, g) \mapsto (f\circ \pi) \cdot g
\end{equation*}
and has an $\uL^\infty(\uZ)$-valued inner 
product induced by the conditional expectation $\E_\Z\colon \uL^2(\uX|\uZ) \to \uL^\infty(\uZ)$,
\begin{equation*}
  ( \cdot | \cdot) \colon \uL^2(\uX|\uZ) \times \uL^2(\uX|\uZ) \to \uL^\infty(\uZ), \quad 
  (f, g) \mapsto \E_\Z(f\overline{g}).
\end{equation*}
As it turns out, general Hilbert modules are rather ill-behaved and do not satisfy 
many of the classical results from Hilbert space theory such as the Riesz-Fréchet 
representation theorem, existence of orthonormal bases, the spectral theorem, \dots 
Fortunately, $\uL^2(\uX|\uZ)$ belongs to a special class of Hilbert modules, so-called 
\textbf{Kaplansky--Hilbert modules}, which do satisfy counterparts for all classical 
Hilbert space results. Kaplansky--Hilbert modules differ from general Hilbert modules $E$
over a commutative $\uC^*$-algebra $\A$ in that $\A$ is a \textbf{Stone algebra}, i.e., 
every bounded subset of real elements has a supremum in $\A$, and $E$ satisfies a 
completeness property with respect to \emph{order convergence}.
The fact that $\uL^2(\uX|\uZ)$ and $\uL^\infty(\uZ)$ satisfy these completeness properties
corresponds to them satisfying a completeness property with respect to almost everywhere
convergence see \cite[Proposition 7.6 and Lemma 7.5]{EHK2021} and 
\cite[Corollary 7.8]{EFHN2015}. 
We forego the precise definition
of (Kaplansky--)Hilbert modules since the only example we require is the conditional 
$\uL^2$-space $\uL^2(\uX|\uZ)$; we refer the reader to \cite[Sections 1 and 2]{EHK2021}
for the details.

Given a Hilbert module $E$ over a $\uC^*$-algebra $A$, the induced 
\textbf{$\A$-valued norm} given by 
\begin{equation*}
  |\cdot|_\A \colon E \to \A, \qquad x \mapsto \sqrt{(x|x)}.
\end{equation*}
Quoting from \cite[Definition 1.5]{EHK2021}, this $\A$-valued norm induces a notion of 
order-convergence as follows: a net $(f_i)_{\alpha\in \A}$
in $\A$ \textbf{decreases to $0$} if 
\begin{equation*}
  i \leq j \quad \implies \quad 0 \leq  f_j \leq f_i \quad \text{and} \quad \inf\{f_i | i\in I\} = 0.
\end{equation*}
A net $(x_\alpha)_{\alpha \in A}$ in $E$ \textbf{order-converges} to $x\in E$ (in symbols: $\olim_\alpha x_\alpha = x$),
if there is a net $(f_i)_{i\in I}$ in $\A$ decreasing to zero and satisfying 
\begin{equation*}
  \forall i\in I \exists \alpha_i \forall \alpha \geq \alpha_i \colon | x_\alpha - x | \leq f_i.
\end{equation*}
A mapping $f\colon E\to F$ between Hilbert modules is \textbf{order-continuous} if $\olim_\alpha x_\alpha = x$
in $E$ implies $\olim_\alpha f(x_\alpha) = f(x)$ whenever $(x_\alpha)_{\alpha\in A}$ is a net in and $x$ an element 
of $E$.

A \textbf{morphism} of Hilbert modules $E$ and $F$ over a commutative $\uC^*$-algebra $\A$ 
is a bounded linear operator $T\colon E \to F$ such that $T(fx) = fTx$ for all $f\in\A$ 
and $x\in E$. The space of all morphisms from $E$ to $F$ is denoted by $\Hom(E;F)$ and the 
\textbf{dual module} of a Hilbert module $E$ is defined as
$E^* \defeq \Hom(E;\A)$. A morphism of Hilbert modules is called an \textbf{isomorphism} if 
it is bijective. By a theorem of 
Lance \cite{Lance1994}, an isomorphism $T\colon E \to F$ is isometric if and 
only if it is \textbf{$\A$-isometric}, i.e., if it preserves the $\A$-valued norms of $E$ 
and $F$. In particular, isometric isomorphisms of Hilbert modules are automatically order-continuous.

Moreover, for a Hilbert module $E$, there generally 
does not exist an orthonormal basis. To 
account for this, a subset $\mathcal{S} \subset E$ is called \textbf{suborthonormal} if 
it consists of pairwise orthogonal elements with $(e|e)^2 = (e|e)$ for all $e\in \mathcal{S}$
and such a suborthonormal system is called a \textbf{suborthonormal basis} if it is maximal with 
respect to set inclusion. It can be shown that every Kaplansky--Hilbert module admits a 
suborthonormal basis, see \cite[Proposition 2.11]{EHK2021}, and that for every suborthonormal basis 
$\mathcal{B} \subset E$ and every $x\in E$ one has the order limit
\begin{equation*}
  x = \sum_{e\in \mathcal{B}} (x|e)e.
\end{equation*}

We recall the definition of Hilbert--Schmidt homomorphisms between Hilbert--Kaplansky modules 
from \cite[Section 2.6]{EHK2021}.

\begin{definition}
  Let $E$ and $F$ be Kaplansky--Hilbert modules over a Stone algebra $\A$.
  Moreover, let $\mathcal{F}$ be the family of all finite suborthonormal subsets of $E$. A 
  homomorphism $A\in \Hom(E;F)$ is called a \textbf{Hilbert--Schmidt homomorphism} if 
  \begin{equation*}
    |A|_{\mathrm{HS}} \defeq \sup \left\{ \left(\sum_{x\in \mathcal{B}} |Ax|_\A^2\right)^\frac{1}{2} \mmid \mathcal{B} \in \mathcal{F} \right\}
  \end{equation*}
  exists in $\A_+ = \{f \in \A \mid f \geq 0\}$, i.e., if the subset of $\A$ in the above expression is bounded in $\A$.
  We write $\HS(E;F)$ for the $\A$-module of all Hilbert--Schmidt homomorphisms from $E$ to $F$ and $\HS(E)$ if $E = F$.
\end{definition}

\begin{lemma}
  Suppose $E$ and $F$ are Kaplansky--Hilbert modules over a Stone algebra $\A$ 
  and let $\mathcal{F}$ be the collection of finite suborthonormal
  subsets of $E$. Then the map
  \begin{equation*}
   \HS(E, F)\times\HS(E, F) \to \A, \quad 
   (A, B) \mapsto (A|B)_{\HS} \defeq \olim_{\mathcal{S}\in\mathcal{F}} \sum_{x\in \mathcal{S}} (Ax|Bx)
  \end{equation*}
  makes $\HS(E, F)$ a Kaplansky--Hilbert module over $\A$. Moreover, if 
  $\mathcal{B}$ is a suborthonormal basis of $E$, then
  \begin{equation*}
    (A|B)_{\HS} = \sum_{x\in \mathcal{B}} (Ax|Bx) \qquad (A, B \in \HS(E))
  \end{equation*}
  is an order-convergent series whose limit is independent of the choice 
  of suborthonormal basis.
\end{lemma}

The key result of this appendix is the following correspondence between certain fiberwise
integral operators and their kernels.

\begin{theorem}\label{thm:appendixmthm}
  Let $\pi\colon \uX \to \uZ$ and $\rho\colon \uY \to \uZ$ be measure-preserving maps between 
  standard probability spaces. Then the assignment
  \begin{equation*}
    I\colon \uL^2(\uX\times_\uZ \uY|\uZ) \to \HS(\uL^2(\uX|\uZ), \uL^2(\uY|\uZ)), \quad 
    (I_kf)(y) \defeq \int_{X_{\rho(y)}\times X_{\rho(y)}} k(x,y)f(x) \dmu_{\rho(y)} 
  \end{equation*}
  defines an isometric isomorphism of Hilbert modules.
\end{theorem}

For the case $\uX = \uY$ and $\pi = \rho$, this was already shown in 
\cite{EHK2021}; we shall closely follow the arguments there and will not
introduce new ideas beyond that.
The usual approach in the Hilbert space case is to observe the 
isomorphisms of Hilbert spaces
\begin{enumerate}[1)]
 \item $\uL^2(\uX\times\uY) \cong \uL^2(\uX)\otimes\uL^2(\uY)$,
 \item $\uL^2(\uX) \cong \uL^2(\uX)^*$, $f\mapsto (\cdot | \overline{f})$ by  the Riesz--Frech\'{e}t representation theorem, and
 \item $\uL^2(\uX)^* \otimes \uL^2(\uY) \cong \HS(\uL^2(\uX), \uL^2(\uY))$.
\end{enumerate}
Combined, these yield the canonical isomorphism $I\colon \uL^2(\uX\times \uY)\to \HS(\uL^2(\uX), \uL^2(\uY))$
of Hilbert spaces. Below, we collect the precise formulation of the 
three isomorphisms 1), 2), and 3) in the more general context of Kaplansky--Hilbert modules. We
start with the first isomorphism.

\begin{theorem}\label{thm:joiningiso}
  Let $\pi\colon \uX \to \uZ$ and $\rho\colon \uY\to \uZ$ be measure-preserving maps between 
  probability spaces. Then there is a unique isometric Hilbert module isomorphism
  \begin{equation*}
    W\colon \uL^2(\uX|\uZ)\otimes \uL^2(\uY|\uZ) \to \uL^2(\uX\times_\uZ \uY | \uZ)
  \end{equation*}
  with $W(f\otimes g) = (T_\pi f) \cdot (T_\rho g)$.
\end{theorem}

For a proof, see \cite[Theorem 7.11]{EHK2021} and \cite[Section 2.7]{EHK2021} for background information
on the tensor product of Hilbert--Kaplansky modules. We recall the second isomorphism, a Riesz--Fr\'{e}chet theorem
for Kaplansky--Hilbert modules, from \cite[Theorem 2.13]{EHK2021}.

\begin{theorem}\label{thm:rieszfrechet}
 Let $E$ be a a Kaplansky--Hilbert module over a Stone algebra $\A$. Then 
 the mapping 
 \begin{equation*}
  \Theta\colon E \to E^* = \Hom(E;\A), \quad y \mapsto \overline{y} \defeq (\cdot | y)
 \end{equation*}
 is $\A$-antilinear, bijective, and satisfies $|y| = |\overline{y}|$ for $y\in E$.
\end{theorem}

As a consequence, there is a canonical isometric isomorphism 
$\uL^2(\uX|\uZ) \cong \uL^2(\uX|\uZ)^*$ of Hilbert modules by 
means of the assignment $f\mapsto (\cdot|\overline{f})$. Moreover, one can also
derive from \cref{thm:rieszfrechet} that every morphism $A\colon E \to F$ between 
Kaplansky--Hilbert modules admits an \textbf{adjoint morphism} $A^*\colon E \to F$ that 
is uniquely determined by the property $(Ae|f) = (e|A^*f)$ for all $e\in E$ and $f\in F$,
see \cite[Corollary 2.14]{EHK2021}.
As for Hilbert spaces, it is immediate that $A^*A$ and $AA^*$ are zero if and only if $A$ is.

The counterpart for the third and final isomorphism $\HS(\uL^2(\uX), \uL^2(\uY)) \cong \uL^2(\uX)^*\otimes\uL^2(\uY)$ 
requires the following simple generalization of finite rank operators:
Given Kaplansky--Hilbert modules $E$ and $F$ as well as elements $x\in E$ and 
$y\in F$, one can define the operator
\begin{equation*}
  A_{y, x}\colon E \to F, \quad z \mapsto (z|x)y.
\end{equation*}
This is a bounded homomorphism of Hilbert modules. A finite sum of 
such operators is called a \textbf{homomorphism of $\A$-finite rank}.

\begin{proposition}\label{prop:tensorHSiso}
  Let $E$ and $F$ be Kaplansky--Hilbert modules over a Stone algebra $\A$. 
  Then there is a unique isometric Hilbert module isomorphism 
  \begin{equation*}
    V\colon E^*\otimes F \to \HS(E, F)
  \end{equation*}
  with $V(\overline{x}\otimes y) = A_{y, x}$ for all $x\in E$ and $y\in F$. 
\end{proposition}

For a proof and the details of tensor products of Kaplansky--Hilbert modules, 
see \cite[Section 2.7]{EHK2021}, in particular \cite[Proposition 2.25]{EHK2021}. 
There, the statement is proven for the case $E = F$ 
but the proof for the general case differs only in straightforward notational adjustments.

As in the case of Hilbert spaces, simply combine \cref{thm:joiningiso}, \cref{thm:rieszfrechet}, 
and \cref{prop:tensorHSiso} to derive the desired isomorphism in \cref{thm:appendixmthm}.

\printbibliography

\end{document}